\documentclass[a4paper,11pt]{amsart}
\usepackage[a4paper,inner=3cm,outer=3cm,top=4cm,bottom=4cm,pdftex]{geometry}

\usepackage[english]{babel}
\usepackage[utf8]{inputenc}
\usepackage{comment}
\usepackage{enumerate}
\usepackage{mathtools}

\usepackage{amssymb}
\usepackage{amsmath}
\usepackage{amsthm}
\usepackage{mathrsfs}
\newtheorem{theorem}{Theorem}[section]
\newtheorem{lemma}[theorem]{Lemma}
\newtheorem{definition}[theorem]{Definition}

\newtheorem{prop}[theorem]{Proposition}
\newtheorem{remark}[theorem]{Remark}
\newtheorem{convention}{Convention}
\newtheorem{keyprop}{Key Proposition}

\usepackage{bm}

\newcommand{\asum}{\sideset{}{^{\ast}}\sum}
\newcommand{\primesum}{\sideset{}{^{'}}\sum}
\newcommand{\psum}{\sideset{}{^{\#}}\sum}

\newcommand{\cond}[1]{\textnormal{cond}{#1}}
\usepackage[hyperfootnotes=false]{hyperref}

\makeatletter
\newcommand{\myitem}[1]{%
\item[#1]\protected@edef\@currentlabel{#1}%
}
\makeatother

\makeatletter
\renewcommand\subsection{\@startsection{subsection}{2}%
	\z@{.5\linespacing\@plus.7\linespacing}{.5\linespacing}%
	{\normalfont\bfseries}}

\renewcommand\part{\@startsection{part}{2}%
	\z@{0.5\linespacing\@plus2\linespacing}{\linespacing}%
	{\normalfont\large\scshape\bfseries\centering}}
\makeatother

\title{The exceptional set in Goldbach's problem with almost twin primes}

\author{Lasse Grimmelt}
\address{Mathematical Institute, University of Oxford, Oxford OX2 6GG, UK}
\email{grimmelt@maths.ox.ac.uk}

\author{Joni Ter\"{a}v\"{a}inen}

\address{Department of Mathematics and Statistics, University of Turku, 20014 Turku, Finland}
\email{joni.p.teravainen@gmail.com}

\begin{document}
	\begin{abstract}
		We consider the exceptional set in the binary Goldbach problem for sums of two almost twin primes. Our main result is a power-saving bound for the exceptional set in the problem of representing $m=p_1+p_2$ where $p_1+2$ has at most $2$ prime divisors and $p_2+2$ has at most $3$ prime divisors. There are three main ingredients in the proof: a new transference principle like approach for sieves,  a combination of the level of distribution estimates of Bombieri--Friedlander--Iwaniec and Maynard with ideas of Drappeau to produce power savings, and a generalisation of the circle method arguments of Montgomery and Vaughan that incorporates sieve weights.
	\end{abstract}
	
\maketitle

\setcounter{tocdepth}{1}
\tableofcontents
\part{Introduction and proof methods}

\section{Introduction}

In this paper, we study the exceptional set in the binary Goldbach problem with almost twin primes. There are different ways to define what an almost twin prime should be. We shall consider here those primes $p$ for which $p+2$ has few prime factors. 
We define 
\begin{align*}
\mathbb{P}_k\coloneqq \{m\in \mathbb{N} : m \text{ has at most } k \text{ prime factors} \},
\end{align*}
and abbreviate $\mathbb{P}_1$ as $\mathbb{P}$. Thanks to a celebrated result of Chen~\cite{chen}, we know that there are infinitely many primes $p$ such that $p+2\in \mathbb{P}_2$.

Based on standard heuristics of Hardy--Littlewood type, we expect that every large $m\equiv 4\pmod 6$ can be written as $m=p_1+p_2$ with $p_i,p_i+2\in \mathbb{P}$; needless to say, this is far out of reach as it would imply as special cases both the twin prime conjecture and the Goldbach conjecture (for large numbers congruent to $4\pmod 6$). However, it is natural to ask what can be said for fixed $k_1,k_2\geq 2$ about representations of the form 
\begin{align}\label{eqn0}
m=p_1+p_2 \quad \textnormal{ with }\quad p_i\in \mathbb{P},\,\, p_i+2 \in \mathbb{P}_{k_i}.
\end{align} 
We define the size of the exceptional set related to this problem by
\begin{align*}
E(N,k_1,k_2)\coloneqq |\{m\leq N: m\equiv 4\pmod 6,\,\, m\neq p_1+p_2\,\, \forall p_i\in \mathbb{P}\cap(\mathbb{P}_{k_i}-2) \}|;
\end{align*}
here the restriction to $m\equiv 4 \pmod 6$ is imposed since almost all elements of $\mathbb{P}\cap (\mathbb{P}_k-2)$ are $\equiv 5 \pmod 6$.

Our main result is a power-saving estimate for $E(N,2,3)$.
\begin{theorem} \label{MT1}
There is a constant $\delta>0$ such that
\begin{align*}
E(N,2,3)\ll N^{1-\delta}.
\end{align*}
\end{theorem}
Both $\delta$ and the implied constant in this theorem are effective and could in principle be computed. 

This generalises the celebrated result of Montgomery and Vaughan~\cite{mv} that all but a power-saving number of even integers up to $N$ are the sum of two primes, which implies in our notation\footnote{We use the notation $\mathbb{P}_{\infty}\coloneqq \mathbb{N}$.} 
\begin{align*}
E(N,\infty, \infty)\ll N^{1-\delta}.    
\end{align*}
The best known value of $\delta$ for this problem, after a series of improvements by several authors, including Chen and Pan, Chen and Liu, Li, and Lu~\cite{lu}, is due to a pre-print of Pintz~\cite{pintz2} with $\delta=0.28$. 

If we want both primes in~\eqref{eqn0} to be Chen primes, we are able to obtain the following weaker estimate for the exceptional set.
\begin{theorem} \label{MT2}
For any $A> 0$, we have
\begin{align*}
E(N,2,2)\ll_A N(\log N)^{-A}.
\end{align*}
Here the implied constant is ineffective.
\end{theorem}

Theorems~\ref{MT1} and~\ref{MT2} improve upon the following estimates. For any $A>0$,  we have
\begin{align}\begin{split}\label{eqn1}
E(N,5,7)&\ll_A N(\log N)^{-A},\\
E(N,3,\infty)&\ll_A N(\log N)^{-A},\\
E(N,2,7)&\ll_A N(\log N)^{-A};
\end{split}
\end{align}
the first result is due to Tolev~\cite{tolev}, the second one is due to Meng~\cite{meng}, and the third one (which improves on the other two) is due to Matom\"aki~\cite{matomaki}. All the implied constants appearing in~\eqref{eqn1} are ineffective. 

The ternary equation $m=p_1+p_2+p_3$  with $p_i$ being Chen primes was solved by Matom\"aki and Shao in~\cite{matomaki-shao} for all large $m$. Even though it is not explicitly stated there, the methods of~\cite{matomaki-shao} can be adapted to handle the binary case with only minor modifications (see~\cite[Proposition 2.1]{tera-goldbach} of the second author for a binary version of the method in~\cite{matomaki-shao}), and a back-of-the-envelope calculation seems to yield an estimate of the form
\begin{align*}
E(N,2,2)\ll N(\log\log\log N)^{-c}.
\end{align*} 

\subsection{An application to level of distribution of the M\"obius function}

Let $\mu$ denote the M\"obius function. As a side product of our proof of Theorem~\ref{MT1}, we obtain the following result about the distribution of $\mu$ in arithmetic progressions to very large moduli, when weighted by a suitably factorable function, see Definition~\ref{Def_triplyWF}.

\begin{theorem}[Level of distribution $3/5$ for the M\"obius function with triply well-factorable weights]\label{MT3} 
Let $k\geq 1$, $a\in \mathbb{Z}\setminus \{0\}$, $A\geq 1$ and $\varepsilon>0$ be fixed, with $\varepsilon$ sufficiently small. Let $N\geq 2$ and $P\leq N^{\varepsilon}$. Let $|\lambda_d|\leq \tau_k(d)$ be any triply well-factorable sequence. 
Then, there exists an effective constant $C_\varepsilon$ and an ineffective constant $C_A$ such that 
	\begin{align*}
	&\Bigg|\sum_{\substack{d\leq N^{3/5-\varepsilon}\\d^{-}>P\\(d,a)=1}}\lambda_d\Big(\sum_{\substack{n\leq N\\n\equiv a(d)}}\mu(n)-\frac{1}{\varphi(d)}\sum_{n\leq N}\mu(n)\Big)\Bigg| 
	\leq C_\varepsilon\min\{ C_A N(\log N)^{-A}, NP^{-1/200}\},
\end{align*}
where $d^-$ denotes the smallest prime divisor\footnote{Here we set $1^{-}\coloneqq \infty$.} of $d$.
\end{theorem}
Here the range $3/5$ and the notion of triply well-factorability are closely related to Maynard's~\cite[Theorem 1.1]{MaynardII}. Theorem~\ref{MT3} extends~\cite[Theorem 1.1]{MaynardII} by replacing $\Lambda$ by $\mu$, which is a similar but technically slightly more challenging problem as the M\"obius function is not supported on rough numbers only, and we are additionally able to produce a saving that goes beyond the Siegel--Walfisz savings.

\subsection{Acknowledgements}
The first author received funding from the European Research Council (ERC) under the 
European Union's Horizon 2020 research and innovation programme, grant 
agreement no. 851318. The second author was supported by a Titchmarsh Fellowship and Academy of Finland grant no. 340098. 

The authors thank Kaisa Matom\"aki and James Maynard for helpful comments. 

%thanks to Kaisa?

\section{Proof methods and limitations}

\subsection{The use of nonnegative models}

The proof of Theorem~\ref{MT1} is based on a nonnegative model approach, together with three Key Propositions, stated precisely in Section~\ref{subsec:keyprop}. The Key Propositions all deal with the problem of evaluating the binary additive convolution
\begin{align*}
	f*g(m)\coloneqq \sum_{n_1+n_2=m}f(n_1)g(n_2)
\end{align*}
for all $m\leq N$ outside a power-saving exceptional set. Here one or both of $f,g$ are of the form $\Lambda(n)\omega(n+2)$, with $\Lambda$ being the von Mangoldt function and $\omega$ some sieve weight; thus, these propositions concern binary additive problems for shifted primes weighted by sieve weights. These propositions will be stated in a fairly general sieve setup for possible applications to other additive sieve problems. 

We (informally) say that a function $f_{\textnormal{model}}$ is a \emph{model function} for $f:\mathbb{N}\to \mathbb{C}$, if  we have
\begin{align}\label{eq_fg}
	f*g(m)= f_{\textnormal{model}}*g(m)\bigl(1+o(1)\bigr)   
\end{align}
for all $g$ in a suitable space of ``test functions''  and all $m\leq N$ outside a set of size $O(N^{1-\delta})$.  For certain test functions, Key Propositions~\ref{prop1a} and~\ref{prop1b} tell us that the function
\begin{align}\label{eq_f1}
	f(n)=\Lambda(n)\omega(n+2)
\end{align}
can be modelled by a simpler function in the sense of~\eqref{eq_fg}. The simpler function is given by replacing the sieve $\omega$ by its pre-sieve component -- this concept is defined precisely later, but it roughly means that the sieve level becomes $N^{\delta}$ and only small primes (roughly up to $N^{c\delta}$) are sifted for. As these two Key Propositions do not immediately give us an asymptotic formula, but instead transfer the additive problem to simpler functions, they can be seen as inspired by Green's transference principle~\cite{green}.

Roughly speaking, if $f$ is as in~\eqref{eq_f1}, Key Proposition~\ref{prop1a} gives us~\eqref{eq_fg} with the just mentioned simpler model function in the case where $\omega$ is a sieve of level $N^{1/2-\varepsilon}$ and $g$ is a fairly arbitrary function; Key Proposition~\ref{prop1b} gives a similar statement extended to the case where $\omega$ is a suitable sieve of level somewhat larger than $N^{1/2}$ and $g$ is a pre-sieve of low level on shifted primes. The final ingredient, Key Proposition~\ref{prop2}, gives an asymptotic formula for the left-hand side of~\eqref{eq_fg} when both $f$ and $g$ are of the form~\eqref{eq_f1} with $\omega$ being a pre-sieve. The three Key Propositions act together to give Theorem~\ref{MT1}, as we shall see now informally and later rigorously in Section~\ref{sec:mainthm}.

Let
\begin{align*}
	\Lambda_k(n)\coloneqq \Lambda(n)1_{\mathbb{P}_k}(n+2)\rho(n+2,N^{\alpha_k}),    
\end{align*}
where $\alpha_k>0$ are some small constants and $\rho(\cdot,z)$ is the indicator of integers having no prime factors $\leq z$.
Then, to prove Theorem~\ref{MT1} it suffices to prove that
\begin{align*}
	\Lambda_3*\Lambda_2(m)\gg N^{0.99}    
\end{align*}
for all but $\ll N^{1-\delta}$ values of $m\leq N$, $m\equiv 4\pmod 6$. Let $\nu_k$ be a minorant function for $\Lambda_k$, and assume more specifically that $\nu_k(n)$ is of the form $\Lambda(n)\omega_k(n+2)$, where $\omega_k$ is some sieve weight satisfying
\begin{align*}
	1_{\mathbb{P}_k}(n)\rho(n,N^{\alpha_k})\geq \omega_k(n).   
\end{align*}
Central for our approach is that we use sieve weights that are composed of a pre-sieve and main-sieve. One can think of the pre-sieve as handling small prime factors with high accuracy and the main-sieve handling large prime factors with reduced accuracy. It will come that the main sieve component interacts only weakly with the additive problem as we will now sketch. 

Let $\omega^{\textnormal{pre-sieve},\pm}_k$ stand for the pre-sieve part of $\omega_k$ and 
\begin{align*}
	\nu^{\textnormal{pre-sieve},\pm}_k(n)=\Lambda(n)\omega^{\textnormal{pre-sieve},\pm}_k(n+2).
\end{align*}
 Here the $\pm$ sign indicates whether the pre-sieve is of upper bound or lower bound form.

With this notation, our proof strategy for Theorem~\ref{MT1} can be in simplified form outlined in the following chain of inequalities and approximations (valid outside a power-saving set of $m\leq N$):
\begin{align*}
	\Lambda_3*\Lambda_2(m)&\stackrel{\text{sieve}}{\geq} \nu_3*\Lambda_2(m)\\
	&\stackrel{\text{KP1}}{\approx} \nu_3^{\textnormal{pre-sieve}, -}*\Lambda_2(m)\\
	&\stackrel{\text{KP3}}{\approx} \nu_3^{\textnormal{pre-sieve}, +}*\Lambda_2(m)\\
	&\stackrel{\text{sieve}}{\geq} \nu_3^{\textnormal{pre-sieve},+}*\nu_2(m)\\
	&\stackrel{\text{KP2}}{\approx} \nu_3^{\textnormal{pre-sieve},+}*\nu_2^{\textnormal{pre-sieve},-}(m)\\
	&\stackrel{\text{KP3}}{\approx} \textnormal{[expected main term]}.
\end{align*}
Here, in the first step we applied the minorant function for $\Lambda_3$. In the second step, we used Key Proposition~\ref{prop1a} to replace $\nu_3$ by its pre-sieve component, which is much easier to understand. In the third step, we applied Key Proposition~\ref{prop2} that in this context acts as a fundamental lemma to replace the lower bound sieve $\nu_3^{\textnormal{pre-sieve},-}$ with the corresponding upper bound sieve $\nu_3^{\textnormal{pre-sieve},+}$; in the pre-sieve range,  these two functions are very close to each other in the $\ell^1$ norm. In the fourth step, we applied a minorant to $\Lambda_2$, this being possible thanks to the nonnegativity of the upper bound sieve $\nu_3^{\textnormal{pre-sieve},+}$. In the fifth step, we replaced $\nu_2$ (which turns out to be a sieve of level $N^{1/2+\eta}$ for some small $\eta>0$) by its pre-sieve part using Key Proposition 2. Finally, the resulting additive convolution is a version of $\Lambda*\Lambda(m)$ with additional pre-sieves. The case of $\Lambda*\Lambda(m)$  outside a power-saving exceptional set was handled by Montgomery and Vaughan in~\cite{mv}, and the sifted version can be evaluated with Key Proposition~\ref{prop2} (with the resulting asymptotic formula including the contribution of a possible Siegel zero). 

Note that if $g\geq \nu$ for some (not necessarily nonnegative) function $\nu$, then the inequality
\begin{align*}
	f*g(m)\geq f*\nu(m)    
\end{align*}
can be guaranteed to hold pointwise  only if $f$ is nonnegative. Therefore, in the fourth step above it was important that we were able to obtain a \emph{nonnegative} model function (otherwise, we would have had to apply the vector sieve, which would have lost a vital constant factor). Our exploitation of nonnegative models in additive problems is motivated by the earlier work~\cite{grimmelt-nonneg} of the first author, see also~\cite{matomaki-shao} where Bohr sets are used to produce nonnegative model functions. The Bohr set approach gives, as mentioned before, worse error terms, but is more flexible; it is not immediately clear how our strategy would adapt to the case of Maynard--Zhang type almost twin primes also considered in~\cite{matomaki-shao}.

The proof of Theorem~\ref{MT2} is comparatively simple, the only tool we use is a slightly generalised form of Key Proposition~\ref{prop1a} (see Proposition~\ref{prop1-generalized}). 

\subsection{Discussion of the Key Propositions}

The proof of each of Key Propositions~\ref{prop1a},~\ref{prop1b} and~\ref{prop2} starts with an application of the circle method, which translates the problem of understanding $f*g(m)$  into the realm of understanding the Fourier transforms of $f$ and $g$. Key Propositions~\ref{prop1a} and~\ref{prop1b} use pointwise Fourier information; Key Proposition~\ref{prop1a} uses pointwise information on $\widehat{f}$ on the whole of $[0,1]$, whereas Key Proposition ~\ref{prop1b} uses pointwise information on $\widehat{f}$ only on the major arcs (see Definition~\ref{def_major}) and a bound for $\widehat{g}$ in the complementing minor arcs. Key Proposition~\ref{prop2} uses a more precise treatment of the contribution of the major arcs, in the spirit of Montgomery and Vaughan's work~\cite{mv}.

We thus need different types of Fourier information for $f(n)=\Lambda(n)\omega(n+2)$ with the sieve $\omega$ having as large a level as possible. On the minor arcs, we can only reach the level $N^{1/2-\varepsilon}$ by the work of Matom\"aki~\cite{matomaki} on Bombieri--Vinogradov type bounds for exponential sums. Key Proposition~\ref{prop1a} follows from this result and a version of the Bombieri--Vinogradov theorem which gives power savings. Such a power-saving estimate can indeed be achieved with the large sieve if one uses a main term that takes into account all the characters of conductor $\leq N^{\delta}$, instead of only the principal character.

While Chen's approach provides a minorant for $\Lambda_2(n)$ using sieve switching and sieves of level $N^{1/2-\varepsilon}$, unfortunately we cannot use this minorant; see Subsection~\ref{subsec: limitations} for an explanation of this limitation. Thus, our approach is to do without sieve switching. Without sieve switching, one can still construct a minorant for numbers with at most two prime divisors, provided that one has level of distribution slightly larger than $1/2$, where the best known value is
$$0.511\ldots$$
by work of Greaves~\cite{Greaves1986}. However, we cannot use his type of sieves either, as we can break the $N^{1/2}$ barrier for primes only if the weights enjoy some nice factorability properties. Fourvry and Grupp~\cite{fouvry-grupp} found a working compromise and used Laborde's sieve~\cite{laborde-1979} together with level of distribution estimates (based on the dispersion method) of Fouvry~\cite{fouvry} and Bombieri--Friedlander--Iwaniec~\cite{bfi} to get a minorant for the Chen primes without switching. We extend their approach to cover the whole major arcs and give power savings. This then gives us Key Proposition~\ref{prop1b}. 

Fundamentally, a power saving for applications of the dispersion method is possible because one can apply the large sieve in the same way as in the power-saving Bombieri--Vinogradov estimate mentioned above; there is previous work on some dispersion method estimates with power savings by Drappeau~\cite{drappeau}. There are several additional complications for us when adapting his approach to the work of Fourvry and Grupp~\cite{fouvry-grupp} and the underlying dispersion estimates. In particular, we need to set up a version of the dispersion method that produces power savings and works in quite some generality. This rules out certain results of Fouvry~\cite{fouvry} used in~\cite{fouvry-grupp}. We rely on an extension of Maynard's recent work~\cite{MaynardII} to patch the gaps in our arithmetic information.

The proof of Key Proposition~\ref{prop2} follows the overall strategy of Montgomery and Vaughan's work~\cite{mv}, but the presence of sieve weights complicates matters. Their proof uses that for $(an,q)=1$ one has
\begin{align}\label{eq_introexpchar}
	e\Bigl(\frac{an}{q}\Bigr)=\frac{1}{\varphi(q)}\sum_{\chi(q)}\chi(an)\sum_{b(q)^{*}}\overline{\chi}(b)e\Bigl(\frac{b}{q}\Bigr)
\end{align}
and so they can translate from major arc integrals to character sums. Montgomery and Vaughan then reduce to primitive characters and keep careful track of the weight of each fixed pair of primitive characters. These weights can be called \emph{pseudo-singular series}, where the pair of primitive (trivial) characters to the modulus $1$ gives the classical singular series. They conclude the proof by bounding the pseudo-singular series by the classical one and applying Gallagher's prime number theorem~\cite{gal}. To approach Key Proposition~\ref{prop2}, we translate both the divisibility condition of the sieve weights as well as the additive phase from the circle method (similarly as in~\eqref{eq_introexpchar}) to multiplicative characters. We are then led to more complicated pseudo-singular series that contain sieve weights. The required results for those can be seen as a combination of Montgomery and Vaughan's pseudo-singular series bounds with a fundamental lemma of sieve theory. Here the difficulty lies in not incurring any additional losses. This is necessary, since in the power-saving exceptional set range the savings from Gallagher's prime number theorem and the fundamental lemma are very weak, only of the type $o(\textup{expected main term})$. Additionally, in order to ensure this type of saving even when the exceptional zero exists -- in which case the expected main term can be much smaller -- we sieve separately for primes dividing the exceptional modulus.

\subsection{Limitations}
\label{subsec: limitations}

With our current knowledge, the simplest way to construct a lower bound for Chen primes is by the use of sieve switching, and indeed we apply it in the proof of Theorem~\ref{MT2}. However, its application in a binary additive problem has issues when the exceptional character enters the picture. To get power saving, one has to consider distribution of primes in arithmetic progressions with moduli being a power of the summation length, so that the possible existence of a real zero of a real so-called exceptional character plays a role. We now give a rough sketch of why that is the case. With sieve switching, one constructs a lower bound of the form
\begin{align} \label{eq_limitations1}
	\Lambda_2(n)\geq \Lambda(n)\omega^-(n+2)-c_{\text{switching}} \omega^+(n)\Lambda_{E_3}(n+2)\coloneqq \Lambda_2^{-}(n),
\end{align}
where $\omega^\pm$ are upper and lower bound sieves and $\Lambda_{E_3}(n)$ is a suitably weighted indicator of integers having precisely three prime factors. For simplicity, we assume that the sieves are normalised in such a way that the existence of infinitely many Chen primes then follows from $1-c_{\text{switching}}>0$. 

If we then consider $\Lambda_2*\Lambda_2(m),$ by Key Proposition $\ref{prop1a}$ we can remove the main sieve components from $\omega^{\pm}$, and by Key Proposition $\ref{prop2}$ the lower bound sieve can be replaced with an upper bound one. If the exceptional character $\widetilde{\chi}\pmod{\tilde{r}}$ with bad zero $\widetilde{\beta}$  exists, then, ignoring the contribution of all other characters, we expect that in the additive setup we have
\begin{align*}
	\Lambda(n)\approx \Lambda_{E_3}(n)\approx \omega^+(n)(1-\widetilde{\chi}(n)n^{\widetilde{\beta}-1})\log n.
\end{align*}
These approximations can be heuristically justified by assuming that GRH holds outside the one exceptional zero $\widetilde{\beta}$. In that case, one can calculate 
\begin{align}
	\sum_{n\leq X}f(n)\omega^{\pm}(n+2)e\Bigl(\frac{an}{q}\Bigr)
\end{align}
with $q\leq X^{\delta}$ and for $f= \Lambda$ or $f=\Lambda_{E_3}$ with a high degree of precision. The result is that the main term is the same as the expected main term for the case $f(n)=\omega^+(n)(1-\widetilde{\chi}(n)n^{\widetilde{\beta}-1})$. In Part \ref{Part:IV}, we make this argument rigorous without assuming GRH.

We thus get the expected approximation in an additive setup
\begin{align*}
	\Lambda_2^{-}(n)\approx \omega^+(n)\omega^+(n+2)\bigl((1-c_{\text{switching}})-(\widetilde{\chi}(n)-c_{\text{switching}}\widetilde{\chi}(n+2))n^{\widetilde{\beta}-1})\bigr)\log n.
\end{align*}

In the binary additive problem $\Lambda_2^{-}*\Lambda_2(m)$, the function $\Lambda_2(n)$ may behave like 
\begin{align*}
	\omega^+(n)\omega^+(n+2)(1-\widetilde{\chi}(n)n^{\widetilde{\beta}-1})\log n
\end{align*}
(if we do not apply switching; doing otherwise only exacerbates the problem). We then get a certain Hardy--Littlewood type main term for the additive problem. Its shape is given by the following heuristics. The terms without characters in both summands give rise to the usual main term with singular series $\mathfrak{S}(m)$. Mixed terms where one summand contributes a character and the other does not are small. The terms with $\tilde{\chi}(n)$ give a secondary main term of the form $\mathfrak{S}(m)m^{\tilde{\beta}}$. Finally, the combination of $\widetilde{\chi}(n)$ with the $\widetilde{\chi}(n+2)$ term gives us a similar term with shifted singular series, $\widetilde{\mathfrak{S}}'(m+2)m^{\tilde{\beta}}$. 
So we expect a main term of the form
\begin{align}\label{MTSwitch}
	(1-c_{\text{switching}})\mathfrak{S}(m)m+\widetilde{\mathfrak{S}}(m)m^{\widetilde{\beta}}-c_{\text{switching}}\widetilde{\mathfrak{S}}'(m+2)m^{\widetilde{\beta}}.
\end{align}

As in \cite{mv}, if $m$ is divisible by the exceptional modulus and $\widetilde{\chi}(-1)=-1$, we have  $\widetilde{\mathfrak{S}}(m)=-\mathfrak{S}(m)$.  In that case
\begin{align*}
	\mathfrak{S}(m)m+\widetilde{\mathfrak{S}}(m)m^{\widetilde{\beta}}\approx \mathfrak{S}(m)(1-\widetilde{\beta})\log N.
\end{align*}
For those bad $m$ we have that $\widetilde{\mathfrak{S}}'(m+2)$ is small (as $\widetilde{r}\mid m$ implies that $(\widetilde{r},m+2)$ is small). We consequently get a negative main term
\begin{align*}
\eqref{MTSwitch}\approx \mathfrak{S}(m)\bigl((1-\widetilde{\beta})\log N-c_{\text{switching}}\bigr)<0,
\end{align*}
for any exceptional zero worth its name.

The issue just presented is the reason behind our application of level of distribution $>1/2$ results that are achieved with the dispersion method and  ultimately rely on the spectral theory of automorphic forms. With this increased level of distribution we no longer require sieve switching and avoid the above-mentioned issue. If we restricted our attention to $E(N,3,3)$, or if the exceptional zero does not exist, we could do without this heavy machinery.

We remark that the techniques we employ are insufficient to produce a power saving for sums of two Chen primes, even under the assumption that the exceptional zero does not exist. Our three Key Propositions allow us to remove the main sieve component of a sieve and exchange upper and lower bound pre-sieves. This is not enough to find a nonnegative model after an application of sieve switching as in~\eqref{eq_limitations1}. To obtain a power saving for $E(N,2,2)$ one consequently would need to on the one hand improve the nonnegative model technique we use and on the other hand find a way to deal with an exceptional zero, possibly in a completely different manner.

Alternatively, we could avoid the issue of sieve switching completely if a minor arc bound for $\Lambda(n)\omega(n+2)$ was proved for a sufficiently large level sieve (somewhat bigger than $N^{1/2}$), as we are able to handle the Fourier transform of this function on the major arcs. In another direction, a wide zero-free strip for the relevant $L$-functions would immediately  solve all our problems. However, breaking the $1/2$ barrier on the minor arcs in any form is a challenge, and the difficulty of proving a zero-free strip does not need any explanation.

\section{Notation}

We now introduce some notation that is used throughout the paper. We also define terms related to sieves and the circle method that we need to state our Key Propositions.

\subsection{Asymptotic notation}

We use the standard Landau and Vinogradov notation $O(\cdot)$, $o(\cdot)$, $\ll$, $\gg$. By $A\sim B$ we mean that $A<B\leq 2A$.

\begin{convention}
		By $c$ (respectively $C$) we denote small (respectively large) positive constants that can vary from line to line. Thus, for instance, an estimate $X\ll Y^{-c}$ means that there is some constant $\alpha>0$ such that $X\ll Y^{-\alpha}$.
	
	In contrast, $c_0$ and $c_1$ are fixed small constants that do not vary from line to line. It will come out of our calculations that the choices 
	\begin{align*}
		c_0&=1/1000 \\
		c_1&=c_0/100
	\end{align*}
are admissible.

\end{convention}

\subsection{Variables}

Throughout, we assume $N\geq P\geq 2$ with $N$ large. The variable $N$ stands for the range in Theorems~\ref{MT1} and~\ref{MT2} and the related range of the summands. The variable $P$ relates to sieves; primes up to $P$ will be handled by the pre-sieves. We later choose $P$ to be a small power of $N$.  To be consistent with other applications of the dispersion method, the usage of $N$ in Section~\ref{sec:beyond12} is different.

The variable $p$ always stands for primes, and the variables $m, n$ for natural numbers. We sometimes abbreviate $[D]=\{1,\ldots, D\}$.
%We use $n$ both as the typical number that we wish to represent as $p_1+p_2$ and as a general summation variable; this should cause no confusion as the two usages do not appear simultaneously in an equation. 

\subsection{Major and minor arcs}

In order to apply the circle method, we now fix our choice of the major and minor arcs.

\begin{definition}[Major and minor arcs]\label{def_major}
	Let $Q=N/P^{c_0}$. Define the \emph{major arcs} 
	\begin{align*}
		\mathfrak{M}(P)\coloneqq [0,1)\cap\bigcup_{1\leq q\leq P^{c_0}}\bigcup_{(b,q)=1}\left[\frac{b}{q}-\frac{1}{Q},\frac{b}{q}+\frac{1}{Q}\right] 
	\end{align*}
	and the \emph{minor arcs}
	\begin{align*}
		\mathfrak{m}(P)\coloneqq [0,1)\setminus \mathfrak{M}(P).
	\end{align*}
\end{definition}

Note that we are taking the major arcs to be slightly wider than usual (width $1/Q$ around $a/q$ instead of $1/(qQ)$) as this turns out to be technically a more convenient choice.

\subsection{$L$-functions}
We reserve the letters $\chi$ and $\psi$ to denote Dirichlet characters and denote by $L(s,\chi)$ the associated $L$-function.   We denote by $\chi^{(q)}_{0}$ the principal character $\pmod q$. We will also need the notion of exceptional zeros of $L$-functions.

\begin{definition}[Exceptional zero]\label{def_SZ}
	By an \emph{exceptional zero} of level $R\geq 2$, we mean a real zero $\tilde{\beta}\in (1-10^{-3}/\log R,1)$ of  some $L$-function $L(s,\tilde{\chi})$ with $\tilde{\chi}$ being primitive and having modulus $\leq R$. The character $\tilde{\chi}$ is then called an \emph{exceptional character}, and its modulus $\tilde{r}$ is called the \emph{exceptional modulus}. 
\end{definition}

By the Landau--Page theorem (in the form of~\cite{pintz}), if an exceptional zero $\tilde{\beta}$ of level $R\geq C$ exists, it must be unique, and $\tilde{\chi}$ must be real and unique.

Related to the exceptional zero is Siegel's theorem in which the implied constant is not effectively computable.

\begin{convention}
	Unless otherwise stated, the implied constants in our use of the $\ll$ and $O(\cdot)$ notation are effective. 
\end{convention}

\subsection{Arithmetic functions}

The symbols $\Lambda, \mu, \varphi$ stand for the von Mangoldt, M\"obius and Euler totient functions. By $\tau_k$ we denote the $k$-fold divisor function, and we let $\tau\coloneqq \tau_2$.

For an arithmetic function $a$, we define its $\ell^2$ norm as 
\begin{align*}
	\|a\|_2\coloneqq \left(\sum_{n}|a(n)|^2\right)^{1/2}.
\end{align*} We use the standard notation
\begin{align*}
	e_q(a)\coloneqq e\left(\frac{a}{q}\right).    
\end{align*}
For a real number $P\geq 2$ and integer $n\geq 1$, we factorise $n=n_{\leq P}\cdot n_{>P}$, where the \emph{smooth part} $n_{\leq P}$ and the \emph{rough part} $n_{>P}$ are given by
\begin{align*}
	n_{\leq P}\coloneqq \prod_{p\leq P}p^{v_p(n)},\quad n_{> P}\coloneqq \prod_{p> P}p^{v_p(n)},    
\end{align*}
where $v_p(n)$ is the largest integer $k$ such that $p^k\mid n$. For an integer $n\geq 2$, we also define $n^+$ (respectively $n^{-}$) as the largest (respectively smallest) prime factor of $n$ (with the convention $1^{+}=1$, $1^{-}=\infty$).

We often abbreviate $a\pmod q$ as $a(q)$. Sums over $a(q)$  are taken over a system of representatives of residues classes; similarly, sums over $\chi(q)$ are taken over characters to the modulus $q$. We denote by
\begin{align*}
	\sum_{a(q)^{*}},\quad \sum_{\chi(q)^{*}}    
\end{align*}
sums over primitive residue classes and primitive characters, respectively. 

For two integers $n_1, n_2$, we denote by $(n_1,n_2)$ and $[n_1,n_2]$ the greatest common divisor, respectively the least common multiple, of $|n_1|$ and $|n_2|$. By $\overline{n}$ we denote the inverse of $n\pmod q$ when $q$ is clear from context.

Given a character $\chi\pmod q$ and a factorisation $q=q_1q_2$ with $q_1,q_2\geq 1$ coprime, we can uniquely factorise $\chi=\chi^{(q_1)}\chi^{(q_2)}$, where $\chi^{(q_i)}\pmod{q_i}$ are characters. In particular, we have the complete factorisation of $\chi\pmod q$ as
\begin{align}\label{e23}
	\chi=\prod_{p \mid q}\chi^{(p^{\alpha(q)})},   \end{align}
where $\chi^{(p^{\alpha(p)})}$ is a character modulo $p^{\alpha(p)}$. We denote furthermore by $\chi^{(p^{\alpha^*(p))}}$ the $p$-component of the primitive character inducing $\chi$. By $\tau(\chi)$ we denote the usual Gau{\ss} sum $\sum_{a(q)}\chi(a)e(a/q)$.

We denote by $*$ and $\star$ the additive and multiplicative convolution operators, respectively. That is, for any two arithmetic functions $\alpha$, $\beta$, we write
\begin{align*}
	\alpha*\beta(m)&=\sum_{a+b=m}\alpha(a)\beta(b)\\
	\alpha\star\beta(m)&=\sum_{ab=m}\alpha(a)\beta(b).
\end{align*}

\subsection{Sieves}

The letter $\mathcal{P}$ is reserved for squarefree integers.  We further use the following notation for the case that it is the product of all primes in a certain range: 
\begin{align*}
	\mathcal{P}(w,z)\coloneqq \prod_{w< p\leq z}p,\quad \mathcal{P}(z)\coloneqq \mathcal{P}(1,z).    
\end{align*}
We also need the related rough number indicators
\begin{align*}
	\rho(n,z)\coloneqq 1_{(n,\mathcal{P}(z))=1},\quad \rho(n,w,z)\coloneqq 1_{(n,\mathcal{P}(w,z))=1}.
\end{align*}

\begin{definition}[Sieve] We say that $\omega$ is a \emph{sieve of range $\mathcal{P}$, level $D$ and order $k$}, if 
	\begin{align*}
		\omega(n)=\sum_{d\mid n}\lambda_d
	\end{align*}
	for some coefficient sequence $|\lambda_d|\leq \tau_k(d)$ supported only on $\{ d\leq D:\,\, d|\mathcal{P}\}$. If the range is $\mathcal{P}(1,z)$ for some $z$, we sometimes say that the range is $z$.
\end{definition}

To describe the interaction of a sieve $\omega$ with a sequence that has a local density function $1/\varphi(d)$, we define
\begin{align}\label{def_mathcV}
	\mathcal{V}(\omega)\coloneqq \sum_d \frac{\lambda_d}{\varphi(d)}.
\end{align}
One expects that the sieve weights $\lambda$ approximate the M\"obius function, so that $\mathcal{V}(\omega)$ is comparable to
\begin{align}\label{def_VP}
	V(\mathcal{P})\coloneqq \sum_{d|\mathcal{P}}\frac{\mu(d)}{\varphi(d)}=\prod_{p|\mathcal{P}}\left(1-\frac{1}{p-1}\right).
\end{align}
If $\mathcal{P}=\mathcal{P}(z_1,z_2)$, we  simply denote the quantity above by $V(z_1,z_2)$.

Several of our results will depend on the fact that the appearing sieve weights can be suitably factorised. 
\begin{definition}[Well-factorable sequences] \label{Def_WF} We say that a sequence $\lambda:[D]\to \mathbb{C}$ is \emph{well-factorable of level $D$ and order $k$} if $|\lambda_d|\leq \tau_k(d)$ and, for any $R,S\geq 1$ with $D=RS$, there exist sequences $|\alpha(d)|, |\beta(d)|\leq \tau_k(d)$ supported on $[1,R]$ and $[1,S]$, respectively, such that $\lambda=\alpha\star\beta$.
\end{definition}
As is well known, the important linear sieve weights are essentially well-factorable.

The notion of \emph{triply well-factorable} weights was introduced by Maynard~\cite{MaynardII} as a stronger condition than well-factorability.\footnote{Strictly speaking, Maynard asks for the sequences to be $1$-bounded. This is purely technical and we chose to keep in line with what we need in our application in the case of Well-factorable sequences.}

\begin{definition}[Triply well-factorable sequences] \label{Def_triplyWF} We say that a sequence $\lambda:[D]\to \mathbb{C}$ is \emph{triply well-factorable of level $D$ and order $k$} if $|\lambda_d|\leq \tau_k(d)$ and, for any $R,S,T\geq 1$ with $D=RST$, there exist sequences $|\alpha(d)|, |\beta(d)|,|\gamma(d)|\leq \tau_k(d)$ supported on $[1,R]$, $[1,S]$, $[1,T]$, respectively, such that $\lambda=\alpha\star\beta\star \gamma$.
\end{definition}

\section{Key Propositions}\label{subsec:keyprop}

In this section, we state our three Key Propositions for proving Theorem~\ref{MT1}, which give information about binary additive  convolutions of a certain type outside a power-saving exceptional set.

Before stating the Key Propositions, we need to define the types of sieves that appear in them. When working with a problem involving sieves of large level, it is often very convenient to assume that the sieves at hand factor as the product of a \emph{pre-sieve} and a \emph{main sieve}. Thus, we wish to work with sieves $\omega$ of level $D$ and range $z$ that factor as $\omega=\omega_1\omega_2$ with $\omega_1,\omega_2$ sieves of ranges $\mathcal{P}(P)$ and $\mathcal{P}(P,z)$. Then, the pre-sieve $\omega_1$ is amenable to the fundamental lemma of sieve theory, and $\omega_2$ has been separated from the influence of the small prime factors. In particular, since  $P$ is larger than the major arc cut-off $P^{c_0}$, one may hope that $\omega_2$ behaves much like a constant function in the setup of the circle method. 

While the aforementioned linear sieve is suitable for sieving for rough numbers and enjoys good factorisation properties, the situation is more delicate if one sieves for numbers with few prime factors. In that case, one can no longer expect well-factorable behaviour. We now define the required level and partial well-factorisation properties that we are able to handle on the major as well as minor arcs.

\begin{definition}[Admissible main sieve]\label{def_admis}
Let $N\geq 2$. We call an arithmetic function $\omega:\mathbb{N}\to \mathbb{R}$ an \emph{admissible main sieve} with parameters $P, \varepsilon, k$ if 
\begin{itemize}
    \item $\omega(n)=\sum_{\substack{d\mid n}}\lambda(d)$ is a sieve of level $N^{1/2-\varepsilon}$, range $\mathcal{P}$ and order $k$ where $\mathcal{P}^{-}>P$. 
    \item  $\lambda=\sum_{j\leq C\log N}\alpha_j\star\lambda'_j$, where for each $j$ there exists $t_j\geq 0$ such that $|\alpha_j|\leq 1_{[Y,2Y]}(j)$ for $Y=N^{t_j}$, $0\leq t_j\leq 1/3-\varepsilon$ and $\lambda'_j$ is well-factorable of level $N^{1/2-t_j-\varepsilon}$ and order $k$. 
\end{itemize}
\end{definition}

\begin{remark}
The reason for limiting the level of the sieve to slightly less than $N^{1/2}$ in this definition is the Bombieri--Vinogradov theorem, or more precisely a version of it with additive character twists and well-factorable weights by Matom\"aki~\cite{matomaki}. See also~\cite[Hypothesis 6.3]{matomaki-shao}. 
\end{remark}

Our first Key Proposition allows us to ``transfer'' a convolution of the form~\eqref{eq_fg} involving admissible sieves into a much simpler convolution involving only their pre-sieve parts. Here the test function $g$ can be fairly arbitrary, as long as the loss of taking its $\ell^2$ norm can be compensated for.

\begin{keyprop}[Transferring a convolution of sieves to its pre-sieve part I]\label{prop1a}
 Let $k\geq 1$ and $\varepsilon \in (0,10^{-4})$ be fixed. Let $N\geq 3$ and $(\log N)^{C}\leq P\leq N^{\varepsilon}$. Let
\begin{align*}
f(n)\coloneqq \Lambda(n)\omega_1(n+2)\omega_2(n+2),
\end{align*}
where $\omega_1$ is a sieve of range $P$, level $D_0\leq N^{5\varepsilon}$ and order $k$, and $\omega_2$ is an admissible main sieve with parameters $P,\varepsilon,k$. Let $g:[1,N]\to \mathbb{C}$ be any function. Then, for all $m\in [1,N]$ apart from $\ll N/P^{c_1}$ exceptions, we have
\begin{align*}
f*g(m)=\mathcal{V}(\omega_2)f_{\textnormal{pre-sieve}}*g(m)+O(\|g\|_2N^{1/2} P^{-c_1}),
\end{align*}
where 
\begin{align*}
f_{\textnormal{pre-sieve}}(n)\coloneqq \Lambda(n)\omega_1(n+2).
\end{align*}
\end{keyprop}

Note that, unlike for the case of detecting numbers with at most three prime divisors, there is no admissible main sieve to lower bound numbers with at most two prime divisors. Since the level of admissible main sieves is restricted to $N^{1/2-\varepsilon}$, the parity problem prevents admissible main sieves from detecting products of at most two primes without the use of sieve switching; see~\cite[page 175]{greaves} for a detailed explanation.

Fouvry and Grupp~\cite{fouvry-grupp} successfully constructed a lower bound for Chen primes without the use of sieve switching, by showing that sieves of the following form can be handled on shifted primes.

\begin{definition}[Fouvry--Grupp sieve] \label{Def_FGsieve}
Let  $N\geq 2$ and 
\begin{align*}
g_{\varepsilon}(t)\coloneqq \begin{cases}4/7 & \text{ if } 0\leq t \leq 2 / 7 -\varepsilon\\
11/20 & \text{ if } 2/7-\varepsilon < t \leq 1/3 -\varepsilon \\
1/2 & \text{ if } 1/3-\varepsilon < t \leq 1/2 -\varepsilon.
\end{cases}
\end{align*}
We call an arithmetic function $\omega:\mathbb{N}\to \mathbb{R}$ a \emph{Fouvry--Grupp sieve} with parameters $P,\varepsilon,k$, if 
\begin{itemize}
    \item $\omega(n)=\sum_{\substack{d\mid n}}\lambda(d)$ is a sieve of level $N^{4/7-\varepsilon}$, range $\mathcal{P}$ and order $k$ with $\mathcal{P}^->P$. 
    \item  $\lambda=\sum_{j\leq C \log N}\alpha_j\star\lambda'_j$ where for each $j$ there exists $t_j\geq 0$ such that  $\alpha_j= 1_{[Y,2Y]\cap \mathbb{P}}(j)$ with $Y=N^{t_j}$ and $\lambda'_j$ is well-factorable of level $N^{g_\varepsilon(t_j)-t_j-\varepsilon}$ and order $k$.
\end{itemize}

\end{definition}

\begin{remark}
Note that the class of Fouvry--Grupp sieves includes sequences that may have a level of distribution as large as $N^{4/7-\varepsilon}$. The exponent $4/7$ here corresponds to the level of distribution in the work of Bombieri--Friedlander--Iwaniec~\cite{bfi}. Observe that if $t_j\leq 2/7-\varepsilon$, then $\omega$ is well-factorable, so the $4/7$ level of distribution result in~\cite{bfi} can be used to control the behaviour of $\Lambda(n)\omega(n+2)$. If instead $2/7-\varepsilon \leq t_j\leq 1/3-\varepsilon$, Fouvry and Grupp succeed by using level of distribution results from~\cite{bfi} and~\cite{fouvry}, and also exploiting that $\alpha_j$ is assumed to be a prime indicator function and so can be opened with combinatorial decompositions.
\end{remark}

The second Key Proposition extends Key Proposition~\ref{prop1a} by replacing an admissible main sieve by a Fouvry--Grupp sieve. Since we are lacking a minor arc bound for the primes twisted by Fouvry--Grupp sieves, the price paid for this extension is a restricted choice of test functions.

\begin{keyprop}[Transferring a convolution of sieves to its pre-sieve part II]\label{prop1b}
 Let $k\geq 1$ and $\varepsilon\in (0,10^{-9})$ be fixed. Let $N\geq 3$ and $(\log N)^{C}\leq P\leq N^{\varepsilon}$. Let
\begin{align*}
f(n)\coloneqq \Lambda(n)\omega_1(n+2)\omega_2(n+2),
\end{align*}
where $\omega_1$ is a sieve of range $P$, level $D_0\leq N^{\varepsilon}$ and order $k$ and $\omega_2$ is a Fouvry--Grupp sieve with parameters $P, \varepsilon, k$. Let \begin{align*}
g(n)\coloneqq \Lambda(n)\omega_3(n+2),
\end{align*}
where $\omega_3$ is a sieve of level $D_0\leq N^{\varepsilon}$ and order $k$. Then, for all $m\in [1,N]$ apart from $\ll N/P^{c_1}$ exceptions, we have
\begin{align*}
f*g(m)=\mathcal{V}(\omega_2)f_{\textnormal{pre-sieve}}*g(m)+O((N+\|g\|_2N^{1/2}) P^{-c_1}),
\end{align*}
where 
\begin{align*}
f_{\textnormal{pre-sieve}}(n)\coloneqq \Lambda(n)\omega_1(n+2).
\end{align*}
\end{keyprop}

Key Propositions~\ref{prop1a} and~\ref{prop1b} allow us to remove main sieve components in binary problems, thus reducing the task of estimating the binary additive convolutions of the primes twisted by sieves to the case where the sieves are mere pre-sieves.
The final main ingredient for the proof of Theorem~\ref{MT1} is an extraction of the main term of the binary additive convolution (for most $n$) in the case where both the involved functions only contain a pre-sieve component. We now define the relevant sifting ranges.

\begin{definition}\label{def_PtildPdag} Let $\widetilde{r}$ denote the exceptional modulus to level $P^{c_0}$, and set
\begin{align*}
\widetilde{\mathcal{P}}\coloneqq \prod_{\substack{p|\widetilde{r}\\ p>2}}p,    
\end{align*}
with the interpretation that the product is $1$ if $\tilde{r}$ does not exist. Define furthermore 
\begin{align*}
\mathcal{P}^\dagger\coloneqq  \prod_{\substack{p\nmid \widetilde{r}\\ 2<p\leq P}}p.
\end{align*}
\end{definition}

Although the dimension of the sieving for almost primes of the form $p+2$ is one, our pre-sieves need to be able to handle higher dimensions. A similar phenomenon occurs in the study of sieve weights in short intervals, see for example~\cite[Chapter 6.10]{cribro} and~\cite{matoalmost}. The combinatorial $\beta$ sieve (see~\cite[Chapter 6.4]{cribro} for its precise definition) is our choice of pre-sieve. We will see that taking $\beta=750$ suffices for our purpose, confer $\eqref{eq_betacond1},~\eqref{eq_betacond2}$. Treating the sieves more carefully would make it possible to reduce the value of $\beta$ considerably, but it seems inherent to our approach that the natural choice for problems of dimension $1$, i.e. $\beta=2$, is not sufficient.

\begin{definition}[Admissible pre-sieve]\label{def_admps}
Let $N\geq 2$. We call $\omega$ an \emph{admissible pre-sieve with parameters $P,D_0$} if $\omega(n)=\omega_0(n) 1_{(n,\widetilde{\mathcal{P}})=1}$, where $\omega_0$ is an upper or lower bound beta sieve with $\beta=750$ (see~\cite[Chapter 6.4]{cribro}), of range $\mathcal{P}^\dagger$ and level $D_0\geq P^{1000}$.
\end{definition}

\begin{remark}
The upper and lower bound sieves just defined have the property that
\begin{align*}
\omega^{-}(n)\leq \rho(n,P)\leq \omega^{+}(n).
\end{align*}
Further, since $D_0\geq P^{1000}$, in both upper and lower bound case they fulfil a fundamental lemma to the effect that (recalling the definitions~\eqref{def_mathcV} and~\eqref{def_VP})
\begin{align}\label{vomega}
\mathcal{V}(\omega)=V(P)\bigl(1+100\theta e^{-\frac{\log D_0}{\log P}})\bigr),\quad \textnormal{ for some }|\theta|\leq 1.
\end{align}
See~\cite[Lemma 6.8]{cribro}. 
Note that in the definition of $\omega$ the sieve is required to be exact on the divisors of the exceptional modulus. This will be important in Key Proposition~\ref{prop2} for handling the contribution of the exceptional character. 
\end{remark}

Our third Key Proposition states that we can insert two admissible pre-sieves into the method of Montgomery--Vaughan
\cite{mv} while only incurring an additional error term of fundamental lemma type. As in~\cite{mv}, the main term and error terms depend on the existence and location of a possible exceptional zero. 

\begin{keyprop}[Sums of two sifted primes]\label{prop2}
 There exist functions $\mathcal{M}(m)$, $\mathcal{E}(m)$ such that the following holds. For any choice of $\nu_1, \nu_2$ of the form
\begin{align*}
\nu_i(n)\coloneqq \Lambda(n)\omega_i(n+2),
\end{align*}
where $\omega_i$ are admissible pre-sieves with parameters $P\leq N$, $D_0\leq N^{1/1000}$, we have
\begin{align}\label{eq_prop2_0}
\nu_1*\nu_2(m)=mV(P)^2 \mathfrak{S}(m)\Bigl( \mathcal{M}(m)\bigl(1+O(e^{- c\frac{\log D_0}{\log P}})\bigr)+O\bigl(e^{100\sqrt{\log N}}P^{-c_1}+e^{- c\frac{\log N}{\log P}}\mathcal{E}(m)\bigr)\Bigr)
\end{align}
for all but $\ll N/P^{c_1}$ values of $m\leq N$, where the singular series is given by
\begin{align*}
 \mathfrak{S}(m)\coloneqq 1_{2\mid m}\cdot  2 \prod_{\substack{p\mid m(m+4)\\ p\neq 2}}\Bigl(1+\frac{p-4}{(p-2)^2} \Bigr)\prod_{\substack{p\mid m+2\\ p\neq 2}}\Bigl(1+\frac{2}{p-2} \Bigr)\prod_{\substack{p \nmid  m(m+2)(m+4)\\ p\neq 2}}\Bigl(1-\frac{4}{(p-2)^2} \Bigr).   
\end{align*}
Furthermore, the main term  $\mathcal{M}(m)$ and the error $\mathcal{E}(m)$ satisfy the following:
\begin{enumerate}[(i)]
\item If there are no exceptional zeros of level $P^{c_0}$, then
\begin{align}\label{eq_prop2_1}
\mathcal{M}(m)=\mathcal{E}(m)=1.
\end{align}

\item If there is an exceptional zero $\widetilde{\beta}$ of level $P^{c_0}$ to the modulus $\widetilde{r}$, then 
\begin{align}\label{eq_prop2_2}
  \mathcal{M}(m)\geq  1-m^{\widetilde{\beta}-1}\prod_{\substack{p\mid \widetilde{r}\\ p\nmid m}}\frac{21}{25}-O(\widetilde{r}^{-0.99}),\quad \mathcal{E}(m)= (1-\widetilde{\beta})\log P.
\end{align}
\end{enumerate}
\end{keyprop}

\begin{remark} We observe that $\mathcal{M}(m)>0$ for all $m\geq 2$ if $P$ (and hence $\tilde{r}$) are large enough. Indeed, 
\begin{align*}
1-m^{\tilde{\beta}-1}=1-\exp((\tilde{\beta}-1)\log m)\geq \frac{1}{10} \min\{(1-\tilde{\beta})\log m,1\} \geq c/(\widetilde{r}^{1/2}(\log \widetilde{r})^2),   
\end{align*}
for some $c>0$, where we used the elementary inequality $1-e^{-x}\geq \frac{1}{10}\min\{x,1\}$ and the classical (effective) bound for exceptional zeros. This last bound is $>B\tilde{r}^{-0.9}$ for any $B$ if $\tilde{r}$ is large enough in terms of $B$.
\end{remark}

\part{Sieves and transference}

\section{Proof of Theorem~\ref{MT1} assuming Key Propositions}\label{sec:mainthm}

In this section, we deduce Theorem~\ref{MT1} from Key Propositions~\ref{prop1a}, \ref{prop1b}, \ref{prop2}. To do so, we first construct sieve minorants for the Chen primes and the primes $p$ with $p+2\in \mathbb{P}_3$ that have the shape required in Definitions~\ref{def_admis} and~\ref{Def_FGsieve}.

\subsection{Setting up the sieves}

The goal of this subsection is to construct an admissible main sieve (cf. Definition~\ref{def_admis}) and a Fouvry--Grupp sieve (cf. Definition~\ref{Def_FGsieve}) that are lower bounds for integers with at most $3$, respectively $2$, prime divisors. A key ingredient in the construction is the application of Iwaniec's linear sieve with well-factorable weights.

\begin{definition}
We say that a multiplicative function $g:\mathbb{N}\to [0,1]$ is a \emph{local density function of dimension $\kappa$} if we have
\begin{align*}
\kappa \log \log z-C \leq  \sum_{p\leq z}\frac{g(p)}{p}\leq \kappa \log \log z+C   
\end{align*}
for all $z\geq 3$. 
\end{definition}

\begin{lemma}[Well-factorable linear sieve]\label{lem_wflinearsieve} 

Let $\varepsilon>0$ and $D>z>P>D^{\varepsilon^2/(1+\varepsilon^9)}>2$. Let  $s=\frac{\log D}{\log z}$. There are sieves $\omega_{\textup{LIN}}^\pm(n)=\omega_{\textup{LIN}}^\pm(n;D,z,P,\varepsilon)=\sum_{d\mid n} \lambda^\pm_d$ with range $\mathcal{P}(P,z)$ and level $D$ and order $1$ with the following properties.
\begin{enumerate}
\item We have $$\omega_{\textup{LIN}}^-(n)\leq \rho(n,P,z)\leq \omega_{\textup{LIN}}^+(n).$$ 
\item The weights are ``$J(\varepsilon)$ well-factorable'' in the following sense. There exists some $J(\varepsilon)$ depending only on $\varepsilon$ such that  \begin{align*}
\lambda_d^\pm=\sum_{1\leq j \leq J(\varepsilon)}\lambda^{\pm,(j)}_d,
\end{align*}
where $\lambda^{\pm,(j)}_d$ are well-factorable weights of level $D$ and order $1$.
\item For any local density function $g$ of dimension $1$, we have
\begin{align*}
\sum_{d}\lambda^{-}_d g(d)&\geq V(P,z)\{f(s)+O\bigl(\varepsilon+(\log D)^{-1/6}\bigr) \}\quad  \textnormal{ if } s\geq 2+\varepsilon,\\
\sum_{d}\lambda^+_d g(d)&\leq V(P,z)\{F(s)+O\bigl(\varepsilon+(\log D)^{-1/6}\bigr) \}\quad \textnormal{ if } s\geq 1+\varepsilon,
\end{align*}
where $f,F$ stand for the functions of the linear sieve (see~\cite[Chapter 12.1]{cribro}).
\end{enumerate}
\end{lemma}

\begin{proof}
This follows from the arguments in the proof of~\cite[Proposition 12.18]{cribro}. In fact, since we do not have a pre-sieve in the statement, the proof is somewhat simpler. 
\end{proof}

We now construct an admissible main sieve to lower bound $\mathbb{P}_3$ numbers. This is done with the help of a simple weighted sieve.
\begin{lemma} \label{lem_admms}
For any sufficiently small $\varepsilon>0$ and $N^\varepsilon \leq P\leq N^{1/10}$, there exists an admissible main sieve $\omega_{\operatorname{M}}^-(n)$ with parameters $P,\varepsilon,C$ such that for $n\leq N$ we have
\begin{itemize}
\item $\omega_{\operatorname{M}}^-(n)\rho(n,P)\leq \rho(n,N^{1/10})1_{\mathbb{P}_3}(n)$,
\item $\mathcal{V}(\omega_{\operatorname{M}}^-)\gg V(P,N^{1/10})$. 
\end{itemize}
\end{lemma}

\begin{proof}
Put $z=N^{1/10}$. Suppose that $y$ satisfies 
\begin{align}
y^3z> N,\quad yz<N^{1/2-\varepsilon}. \label{eq_admms1}
\end{align}
Then, for $n\leq N$, we have
\begin{align*}
1_{\mathbb{P}_3}(n)\rho(n,z)&\geq \Big(1-\frac{1}{2}\sum_{z\leq p<y}1_{p|n}\Big)\rho(n,z)\\
&=\rho(n,P)\Big(\rho(n,P,z)-\frac{1}{2}\sum_{z\leq p<y}1_{p|n}\rho(n,P,z)\Big). \nonumber
\end{align*}

We split the summation over $p$ dyadically and use lower and upper bound sieves as given by Lemma~\ref{lem_wflinearsieve} on the two occurrences of $\rho(n,P,z)$ to define
\begin{align*}
\omega_{\textnormal{M}}^-(n)&\coloneqq \omega_{\textup{LIN}}^{-}(n,N^{1/2-\varepsilon},\mathcal{P}(P,z),\varepsilon)\\
&-\frac{1}{2}\sum_{\substack{K\\\exists j\geq 0:\,\, K=2^j z}}\sum_{K\leq p<\min\{y,2K\}}1_{p|n}\omega_{\textup{LIN}}^+(n,N^{1/2-\frac{\log K}{\log N}-\varepsilon},\mathcal{P}(P,z),\varepsilon).
\end{align*}
The corresponding sieve weights are given by
\begin{align*}
\lambda(d)\coloneqq \lambda^{-}_d-\frac{1}{2}\sum_{\substack{K\\\exists j\geq 0:\,\, K=2^j z}} \sum_{K\leq p<\min\{y,2K\}} \sum_{d=pd'}\lambda^+_{d'}(K),
\end{align*}\
where $\lambda_d^{-}$ and $\lambda^+_{d'}(K)$  are the sieve coefficients corresponding to  $\omega_{\textup{LIN}}^{-}(n,N^{1/2-\varepsilon},\mathcal{P}(P,z),\varepsilon)$ and $\omega_{\textup{LIN}}^+(n,N^{1/2-\frac{\log K}{\log N}-\varepsilon},\mathcal{P}(P,z),\varepsilon)$, respectively. 

To calculate $\mathcal{V}(\omega_{\textnormal{M}}^-)$, we apply the third part of Lemma~\ref{lem_wflinearsieve} separately on each of the sieves. This gives us
\begin{align*}
\mathcal{V}(\omega_{\textnormal{M}}^-)\geq V(P,z)\left(\bigl(f(s)+O(\varepsilon)\bigr)-\frac{1}{2}\sum_{\substack{K\\\exists j\geq 0:\,\, K=2^j z}} \sum_{K\leq p<\min\{y,2K\}} \frac{1}{\varphi(p)}(F(s_K)+O(\varepsilon))\right),
\end{align*}
where $s\coloneqq \frac{\log N}{2\log z}$, $s_K\coloneqq s-\frac{\log K}{\log z}$. In our range of summation, $\frac{\log p}{\log z}=\frac{\log K}{\log z}+O(\frac{1}{\log z})$ and $F(s_K)\gg 1$, so we get
\begin{align*}
\sum_{\substack{K\\\exists j\geq 0:\,\, K=2^j z}} \sum_{K\leq p<\min\{y,2K\}} \frac{1}{\varphi(p)}(F(s_K)+O(\varepsilon))=\sum_{z\leq p < y} \frac{F(s-\frac{\log p}{\log z})(1+O(\varepsilon))}{p}.
\end{align*}
By comparing the sum to an integral, we conclude that
\begin{align*}
\mathcal{V}(\omega_{\textnormal{M}}^-)\geq V(P,z)\left(f(s)-\frac{1}{2}\int_{1}^{\frac{\log y}{\log z}} \frac{F(s-t)}{t}dt+O(\varepsilon)\right).
\end{align*}
Finally, set $y=N^{1/3-\varepsilon}$ which is compatible with~\eqref{eq_admms1} and also makes $\omega_{\textnormal{M}}^{-}$  an admissible main sieve. By a simple numerical calculation (see for example~\cite[Appendix A]{matomaki-shao} for details)
\begin{align*}
\mathcal{V}(\omega_{\textnormal{M}}^-)&\gg  V(P,z)
\end{align*}
for any $\varepsilon>0$ sufficiently small.
\end{proof}

Next we import the construction from~\cite{fouvry-grupp} to construct a Fouvry--Grupp sieve that minorises numbers with at most two prime divisors.

\begin{lemma}\label{lem_fogrs}
For any $P\leq N^{1/15}$ and $0< \varepsilon \leq c$, there exists a Fouvry--Grupp a sieve $\omega_{\operatorname{FG}}^-$ with parameters $P, \varepsilon,C$ such that, for all $n\leq N$,
\begin{itemize}
\item $\omega_{\operatorname{FG}}^-(n)\rho(n,P)\leq \rho(n,N^{1/15})1_{ \mathbb{P}_2}(n) $,
\item $\mathcal{V}(\omega_{\operatorname{FG}}^-)\gg  V(P,N^{1/15})$. 
\end{itemize}

\end{lemma}
\begin{proof}
The statement is a consequence of the construction and calculations in~\cite[Section IV]{fouvry-grupp}. The only necessary modification is a restriction of the range to $\mathcal{P}(P,N^{1/15})$, but this is unproblematic.
\end{proof}

We will also require a simple upper bound sieve. 
\begin{lemma}\label{lem_Mupperbound}
For any $P\leq N^{1/10}$ and $\varepsilon\in (0,1/10)$ there exists an admissible main sieve $\omega^{+}_{\textnormal{M}}$ with parameters $P,\varepsilon,C$ such that, for all $n\leq N$,
\begin{itemize}
\item $\rho(n,P,N^{1/10})\leq \omega_{\operatorname{M}}^+(n)$,
\item $\mathcal{V}(\omega^+_{\operatorname{M}})\leq (F(4)+O(\varepsilon)) V(P,N^{1/10})$.
\end{itemize}
\end{lemma}

\begin{proof}
We set $\omega_{\textnormal{M}}^+(n)=\omega_{\textup{LIN}}^+(n,N^{1/2-\varepsilon},N^{1/10},P,\varepsilon)$. The statement follows immediately from Lemma~\ref{lem_wflinearsieve}.
\end{proof}

\subsection{The main proof}

We require one last ingredient before we begin the proof, namely a type of vector sieve inequality. However, we do not apply it on the vector $(n_1,n_2)$ given by the two variables $n_1,n_2$ in the equation $n_1+n_2=m$, but instead on $(a_i,b_i)$, where $a_ib_i=n_i$ and $a_i$ contains (with multiplicity) all prime divisors of $n_i$ up to $P$. This idea of constructing a lower bound sieve by composition goes back to Selberg~\cite{selberg}. See also~\cite{greaves}, where this approach appears on multiple occasions throughout the book.

\begin{lemma}\label{lem_lowerboundsievcomp}
Let $A,B\geq 0$ and $A^\pm, B^\pm$ such that 
\begin{align*}
AB^{-}&\leq AB\\
\max\{B^-,0\}&\leq B^+ \\
A^-\leq A &\leq A^+.
\end{align*}
Then
\begin{align*}
A^+B^- +(A^- -A^+)B^+  \leq AB.
\end{align*}
\end{lemma}
\begin{proof}
We have
\begin{align*}
AB&\geq AB^-=A^+B^- +(A-A^+)B^-.
\end{align*}
Since $A-A^+\leq 0$, we can bound this from below by
\begin{align*}
\geq A^+B^- +(A-A^+)B^+.
\end{align*}
As $B^+\geq 0$, this is
\begin{align*}
\geq A^+B^- +(A^- -A^+)B^+, 
\end{align*}
as required.
\end{proof}

We have now gathered all the necessary tools to prove our main theorem, assuming the Key Propositions.

\begin{proof}[Proof of Theorem~\ref{MT1}, assuming Key Propositions~\ref{prop1a},~\ref{prop1b}, and~\ref{prop2}]

Let $\varepsilon>0$ be fixed and sufficiently small (the limiting factor being the constructions of the previous subsection).
Let 
\begin{align}\label{eq:D0}
D_0=N^{\varepsilon^2}, \quad P=D_0^{1/A},    
\end{align}
where the constant $A>1000$ is chosen sufficiently large (but with $A$ small enough in terms of $1/\varepsilon$).

We define the functions
\begin{align*}
\Lambda_{2}(n)&\coloneqq \Lambda(n)1_{\mathbb{P}_2}(n+2)\rho(n+2,N^{1/15}),\\
\Lambda_{3}(n)&\coloneqq \Lambda(n)1_{\mathbb{P}_3}(n+2)\rho(n+2,N^{1/10}).
\end{align*}
To prove Theorem~\ref{MT1} it suffices to show the existence of a $\delta>0$ such that
\begin{align}\label{eq_MTproofgoal}
\Lambda_{3}*\Lambda_2(m)\gg N^{0.99}\,\, \text{ for }\,\, m\leq N, m=4(6),\,\, \text{ with }\,\, \ll N^{1-\delta}\,\, \text{ exceptions, }
\end{align}
since the number of representations of $m\leq N$ in the form $m=p_1^i+p_2^j$ with $p_1,p_2$ primes and $\max\{i,j\}\geq 2$ is certainly $\ll N^{1/2}$.  
We use Lemma~\ref{lem_lowerboundsievcomp} with 
\begin{align*}
B&=1_{\mathbb{P}_3}(n+2)\rho(n+2,P,N^{1/10}) \\
A&=\rho(n+2,P) \\
B^-&=\omega_{\textnormal{M}}^-(n+2) &\text{ as in Lemma }~\ref{lem_admms} \text{ with parameters } P,\varepsilon,C \\
B^+&=\omega_{\textnormal{M}}^+(n+2) &\text{ as in Lemma }~\ref{lem_Mupperbound} \text{ with parameters } P,\varepsilon,C \\
A^-&=\omega^-(n+2)  &\text{ as in Definition }~\ref{def_admps} \text{ with parameters } P,D_0 \\
A^+&=\omega^+(n+2)&\text{ as in Definition }~\ref{def_admps} \text{ with parameters } P,D_0
\end{align*}
to get
\begin{align}
\Lambda_{3}(n)&\geq \Lambda(n)\omega^+(n+2)\omega_{\textnormal{M}}^-(n+2)+\Lambda(n)\left(\omega^-(n+2) -\omega^+(n+2) \right)\omega_{\textnormal{M}}^+(n+2) \nonumber \\
&\coloneqq g_1^{\textnormal{M}}(n)+g_2^{\textnormal{M}}(n), \label{eq_g1g2}
\end{align}
say. Since $\Lambda_2(n)\geq 0$ we have 
\begin{align}\label{firstineq}
\Lambda_{3}*\Lambda_2(m)\geq g_1^{\textnormal{M}}*\Lambda_2(m)+g_2^{\textnormal{M}}*\Lambda_2(m).
\end{align}
By construction $\omega^-(n+2)\leq \omega^+(n+2)$, and so $g_2^{\textnormal{M}}(n)\leq 0$. Consequently, we can use the majorant
\begin{align*}
\Lambda_2(n)&\leq \Lambda(n)\omega^+(n+2)\omega_{\textnormal{M}}^+(n+2)\coloneqq g_3^{\textnormal{M}}(n),
\end{align*}
say, to bound
\begin{align*}
|g_2^{\textnormal{M}}*\Lambda_2(m)|\leq |g_2^{\textnormal{M}}*g_3^{\textnormal{M}}(m)|.
\end{align*}
We set 
\begin{align*}
g_1^{\textnormal{P}}(n)&\coloneqq \Lambda(n)\omega^+(n+2)\\
g_2^{\textnormal{P}}(n)&\coloneqq \Lambda(n)\left(\omega^-(n+2) -\omega^+(n+2) \right).
\end{align*}
Note that $\|g_i^{\textnormal{M}}\|_2,\|g_i^{\textnormal{P}}\|_2\ll N^{1/2}(\log N)^{O(1)}$ for both $i$.  Let us for this proof write $=_P$, $\geq_P$, $\leq_P$ to denote that the respective statement holds for all $m\leq N$ with $\ll N/P^{c_1}$ exceptions.

We have by Key Proposition~\ref{prop1a} that
\begin{align}\label{eq_g1'L2}
g_1^{\textnormal{M}}*\Lambda_2(m)=_P\mathcal{V}(\omega_{\textnormal{M}}^-)g_1^{\textnormal{P}}*\Lambda_2(m)+O\left(NP^{-c_1}\right),
\end{align}
and, now applying Key Proposition~\ref{prop1a} twice,
\begin{align}\label{eq_g2g3}
g_2^{\textnormal{M}}*g_3^{\textnormal{M}}(m)=_P \mathcal{V}(\omega_{\textnormal{M}}^+)^2 g_2^{\textnormal{P}}*g_1^{\textnormal{P}}(m)+O\left(NP^{-c_1}\right).
\end{align}
Here the repeated occurrence of $\mathcal{V}(\omega_{\textnormal{M}}^+)$ comes from the fact that the pre-sieve components of $g_1^{\textnormal{M}}$ and $g_3^{\textnormal{M}}$ are the same. From~\eqref{firstineq},~\eqref{eq_g1'L2},~\eqref{eq_g2g3} we conclude that 
\begin{align}\label{secondineq}
\Lambda_3*\Lambda_2(m)\geq_P \mathcal{V}(\omega_{\textnormal{M}}^-)g_1^{\textnormal{P}}*\Lambda_2(m)-|\mathcal{V}(\omega_{\textnormal{M}}^{+})^2g_2^{\textnormal{P}}*g_1^{\textnormal{P}}(m)|+O(NP^{-c_1}) .   
\end{align}

We first deal with the second term on the right of~\eqref{secondineq}. We can apply Key Proposition~\ref{prop2} once to each of the $\omega^{\pm}$ components of $g_2^{\textnormal{P}}$. The main term is the same in both cases and so it cancels out. We get
\begin{align} \label{eq_g2g1'bound}\begin{split}
|\mathcal{V}(\omega_{\textnormal{M}}^+)^2 g_2^{\textnormal{P}}*g_1^{\textnormal{P}}(m)|&=_P O\Bigl(m\mathcal{V}(\omega_{\textnormal{M}}^+)^2 V(P)^2 \mathfrak{S}(m)\bigl(\mathcal{M}(m)e^{- c\frac{\log D_0}{\log P}}+e^{100\sqrt{\log N}}P^{-c_1}\\
&\qquad \qquad+e^{- c\frac{\log N}{\log P}}\mathcal{E}(m)\Bigr).
\end{split}
\end{align}

We are left with lower bounding the first term on the right of~\eqref{secondineq}.  We apply sieves similarly as before, but we can replace the admissible main sieve with a Fouvry--Grupp sieve, as $g_1^{\textnormal{P}}$ admits a minor arc bound. 
We have $g_1^{\textnormal{P}}(n)\geq 0$ as $\omega^{+}$ is an upper bound sieve. An application of Lemma~\ref{lem_lowerboundsievcomp}, similar to the one in~\eqref{eq_g1g2}, gives now
\begin{align}\label{eq_g1'L22}
g_1^{\textnormal{P}}*\Lambda_2(m)&\geq  g_1^{\textnormal{P}}*(g_1^{\textnormal{FG}}+g_2^{\textnormal{M}})(m),
\end{align}
where $g_2^{\textnormal{M}}$ is as before and $g_1^{\textnormal{FG}}$ is as $g_1^{\textnormal{M}}$, but with $\omega_{\textnormal{FG}}^-$ (as in Lemma~\ref{lem_fogrs})in place of $\omega_{\textnormal{M}}^-$.

By Key Proposition~\ref{prop1a} and Key Proposition~\ref{prop2} (applied similarly as in~\eqref{eq_g2g1'bound}) we have 
\begin{align}\label{g1'g2FG}\begin{split}
 g_1^{\textnormal{P}}*g_2^{\textnormal{M}}(m)&=_P \mathcal{V}(\omega_{\textnormal{M}}^-) g_1^{\textnormal{P}}*g_2^{\textnormal{P}}(m)+O(NP^{-c_1})\\
&=_P O\Bigl(m  \mathcal{V}(\omega_{\textnormal{M}}^-) V(P)^2 \mathfrak{S}(m)\bigl(\mathcal{M}(m)e^{-c\frac{\log D_0}{\log P}}+e^{100\sqrt{\log N}}P^{-c_1}\\
&\qquad\quad +e^{- c\frac{\log N}{\log P}}\mathcal{E}(m)\bigr)+NP^{-c_1}\Bigr).
\end{split}
\end{align}
Since $g_1^{\textnormal{P}}$ only contains one sieve of suitable range and level and the main sieve component of $g_2^{\textnormal{FG}}$ is a Fouvry--Grupp sieve, we deduce from Key Proposition~\ref{prop1b} that
\begin{align}\label{g1'g1FG}
 g_1^{\textnormal{P}}*g_1^{\textnormal{FG}}(m)=_P\mathcal{V}(\omega_{\textnormal{FG}}^-)g_1^{\textnormal{P}}*g_1^{\textnormal{P}}(m)+ O(NP^{-c}).
\end{align}
Finally, the sieve part of $g_1^{\textnormal{P}}$ is only an admissible pre-sieve, so Key Proposition~\ref{prop2} gives
\begin{align}\label{g1'g1'}
g_1^{\textnormal{P}}*g_1^{\textnormal{P}}(m)=_P mV(P)^2 \mathfrak{S}(m)\Bigl( \mathcal{M}(m)\bigl(1+O(e^{- c\frac{\log D_0}{\log P}})\bigr)+O(e^{100\sqrt{\log N}}P^{-c_1}+e^{- c\frac{\log N}{\log P}}\mathcal{E}(m))\Bigr).
\end{align}

We have
\begin{align*}
\mathcal{V}(\omega^-_{\textnormal{FG}})\asymp\mathcal{V}(\omega^-_{\textnormal{M}})\asymp \mathcal{V}(\omega^+_{\textnormal{M}})\asymp V(P,N)
\end{align*}
and so from~\eqref{secondineq},~\eqref{eq_g2g1'bound},~\eqref{eq_g1'L22},~\eqref{g1'g2FG},~\eqref{g1'g1FG}, and~\eqref{g1'g1'} we conclude that outside of a set of size $\ll N/P^{c_1}$ we have
\begin{align}\label{eq_L2L3lower}\begin{split}
&\Lambda_3*\Lambda_2(m)\\
&\gg m V(N)^2 \mathfrak{S}(m)\Bigl( \mathcal{M}(m)\bigl(1-Ce^{- c\frac{\log D_0}{\log P}}\bigr)-Ce^{100\sqrt{\log N}}P^{-c_1}-Ce^{- c\frac{\log N}{\log P}}\mathcal{E}(m)\Bigr)+O(NP^{-c_1}).
\end{split}
\end{align}

Suppose now that for a suitable choice of $\tilde{c},\tilde{C}$ (which we now fix) we have
\begin{align}\label{eq_Mgoal}
\mathcal{M}(m)(1-\tilde{C}e^{-\tilde{c} \frac{\log D_0}{\log P}})-\tilde{C}(e^{100\sqrt{\log N}}P^{-c_1}+e^{-\tilde{c}\frac{\log N}{\log P}}\mathcal{E}(m))\geq P^{-c_1/2}.
\end{align}
Then, using $\mathfrak{S}(m)\gg 1$, we see from~\eqref{eq_L2L3lower} that for $m\leq N, m\equiv 4 \pmod 6$ with $\ll N/P^{c_1/10}$ exceptions we have
\begin{align*}
\Lambda_3*\Lambda_2(m)\gg \frac{N}{(\log N)^2} P^{-c_1},   
\end{align*}
and so~\eqref{eq_MTproofgoal} holds with $\delta=\varepsilon^2c_1/(10A)$. We are then left with verifying~\eqref{eq_Mgoal}.

If $A$ is chosen large enough in~\eqref{eq:D0}, we have
\begin{align}\label{conditions}
 \tilde{C}e^{-\tilde{c}A}\leq 1/100,
\end{align}
with $\tilde{c}, \tilde{C}$ being the constants in~\eqref{eq_Mgoal}.
Then we have
\begin{align}\label{bound1}
1-\tilde{C} e^{- \tilde{c}\frac{\log D_0}{\log P}}\geq 99/100.
\end{align}
Furthermore,
\begin{align}
	\label{eq_bound2}
	\tilde{C}e^{100\sqrt{\log N}}P^{-c_1}\leq 1/1000
\end{align}
for all sufficiently large $N$.

We now split the rest of the argument into two cases.

\textbf{Case 1.} If the exceptional zero does not exist, then by~\eqref{bound1} we have
\begin{align*}\mathcal{M}(m)(1-\tilde{C} e^{- \tilde{c}\frac{\log D_0}{\log P}})\geq 99/100
\end{align*}
and
\begin{align*}
\tilde{C}(e^{100\sqrt{\log N}}P^{-c_1}+e^{- \tilde{c}\frac{\log N}{\log P}}\mathcal{E}(m))\leq \tilde{C}(e^{100\sqrt{\log N}}P^{-c_1}+e^{- \tilde{c}\frac{\log D_0}{\log P}})\leq \frac{1}{90}
\end{align*}
 by~\eqref{conditions} and~\eqref{eq_bound2}. Therefore~\eqref{eq_Mgoal} holds in this case in a much stronger from.

\textbf{Case 2.} If the exceptional modulus $\widetilde{r}$ with exceptional zero $\widetilde{\beta}$ exists, then by~\eqref{eq_prop2_2} we have
\begin{align*}
  \mathcal{M}(m)(1-\tilde{C} e^{- \tilde{c}\frac{\log D_0}{\log P}})&\geq \frac{99}{100}\Big(1-m^{\widetilde{\beta}-1}\prod_{\substack{p\mid \widetilde{r}\\ p\nmid m}}\frac{21}{25}-C\tilde{r}^{-0.99}\Big),\\
  \mathcal{E}(m)&= (1-\widetilde{\beta})\log P.
\end{align*}
We follow the arguments in~\cite[Section 8]{mv} and first discard those $m\leq N$ for which $(m,\widetilde{r})>P^{c_1/10}$.  The number of discarded $m$ is at most
\begin{align*}
\sum_{\substack{d\mid \widetilde{r}\\ d>P^{c_1/10}} } \sum_{\substack{m\leq N \\d \mid m}}1 \ll   NP^{-c_1/10}\tau(\widetilde{r}) \ll N/P^{c_1/5},
\end{align*}
where we used that $\widetilde{r}\leq P^{c_0}$. If there exists a prime such that $p\mid \widetilde{r}$, $p\nmid m$, then
\begin{align*}
\mathcal{M}(m)\geq 1-\frac{21}{25}-C\tilde{r}^{-0.99}\geq \frac{1}{7},
\end{align*}
say, provided that $N$ (and hence $\tilde{r}$) is large. Then we obtain~\eqref{eq_Mgoal} as in the case of no exceptional zeros (again, in a much stronger form). Since we are only considering $m$ with $(m,\widetilde{r})\leq P^{c_1/10}$, and since $\tilde{r}/2^j$ is squarefree for some $j\leq 3$, the other case that for all $p\mid \widetilde{r}$ we have $p\mid m$ can only happen if 
\begin{align}\label{bound2}
\widetilde{r}\leq 8 P^{c_1/10}.
\end{align}
By the inequality $1-e^{-x}\geq \min\{1,x\}/10$, for $P\leq m\leq N$ we get 
\begin{align*}
\mathcal{M}(m)&\geq\frac{1}{10} (1-\widetilde{\beta})\log P-C\tilde{r}^{-0.99}.
\end{align*}
By~\eqref{conditions}, we have
\begin{align*}
\tilde{C}e^{-\tilde{c}\frac{\log N}{\log P}}\mathcal{E}(m)\leq  \tilde{C}e^{-\tilde{c}A} (1-\widetilde{\beta})\log P\leq \frac{1}{100} (1-\tilde{\beta})\log P.
\end{align*}

Putting everything together, by using~\eqref{bound1},~\eqref{bound2}, and the classical (effective) bound for the distance of  exceptional zeros from $1$, we conclude that
\begin{align*}
 &\mathcal{M}(m)(1-\tilde{C}e^{-\tilde{c}\frac{\log D_0}{\log P}})-\tilde{C}(e^{100\sqrt{\log N}}P^{-c_1}+e^{-\tilde{c}\frac{\log N}{\log P}}\mathcal{E}(m))\\
 &\geq \frac{1-\tilde{\beta}}{20}\log P-C(\tilde{r}^{-0.99}-P^{-c_1/2})\\
 &\geq \frac{\alpha}{\tilde{r}^{1/2}(\log \tilde{r})^2}-CP^{-c_1/2}\\
 &\geq \frac{\alpha}{100}P^{-c_1/10}\end{align*}
for some (effective) $\alpha>0$, provided again that $N$ (and hence $P$) is large enough. Thus, we get~\eqref{eq_Mgoal} outside an exceptional set of size $O(N/P^{c_1/5})=O(N^{1-\frac{\varepsilon^2 c_1}{5 A}})$. This concludes the proof.
\end{proof}

\part{Level of distribution estimates with power savings}

\section{Bombieri--Vinogradov range -- Proof of Key Proposition~\ref{prop1a}}

\subsection{Reduction to exponential sums}\label{sec: prop1a}

Our task in this section is to prove Key Proposition~\ref{prop1a}. We shall in fact prove it in a slightly more general form of Proposition~\ref{prop1-generalized} below, as this generalisation will be  needed for Theorem~\ref{MT2}. We begin by reducing the problem to estimating exponential sums. This reduction is also needed for the proof of Key Propositions~\ref{prop1b} and~\ref{prop2}.

\begin{lemma}\label{le_fourier} Let $\eta>0$ and $N\geq 1$. Let $f,g:[1,N]\to \mathbb{C}$ be functions. Suppose that we have
\begin{align}\label{eq1b}
\sup_{\alpha \in \mathbb{R}}\left|\sum_{n\leq N}f(n)e(\alpha n)\right|\leq \eta N.    
\end{align}
Then, for all but $\ll \eta^{2/3} N$ integers $m\in [1,N]$, we have
\begin{align*}
|f*g(m)|\leq \eta^{2/3} N^{1/2}\|g\|_2.    
\end{align*}
\end{lemma}

\begin{proof}
Let $\mathcal{S}\subset [1,N]$ be the set of $m\leq N$ for which
\begin{align*}
|f*g(m)|\geq  \eta^{2/3} N^{1/2}\|g\|_2
\end{align*}
Pick unimodular complex numbers $c_m$ such that
\begin{align*}
 c_m(f*g(m))\geq  \eta^{2/3} N^{1/2}\|g\|_2 
\end{align*}
for $m\in \mathcal{S}$. 

Then, summing over $m\leq N$ and applying the orthogonality of characters, we obtain
\begin{align}\begin{split}\label{eq2}
 \eta^{2/3} N^{1/2}\|g\|_2|\mathcal{S}|&\leq\sum_{n_1,n_2\leq N}f(n_1)g(n_2)c_{n_1+n_2}1_{\mathcal{S}}(n_1+n_2)\\
&= \int_{0}^{1} F(\alpha)G(\alpha)S(-\alpha) \, d\alpha, 
\end{split}
\end{align}
where
\begin{align*}
F(\alpha)&\coloneqq \sum_{n\leq N}f(n)e(n\alpha),\\
G(\alpha)&\coloneqq \sum_{n\leq N}g(n)e(n\alpha),\\
S(\alpha)&\coloneqq \sum_{n\leq N}c_n1_{\mathcal{S}}(n)e(n\alpha).
\end{align*}

Now, by the assumption~\eqref{eq1b}, we can apply Cauchy--Schwarz and Parseval to~\eqref{eq2} to conclude that
\begin{align*}
 \eta^{2/3} N^{1/2}\|g\|_2|\mathcal{S}|&\ll \eta N\left(\int_{0}^1 |G(\alpha)|^2\, d\alpha\right)^{1/2} \left(\int_{0}^1 |S(\alpha)|^2\, d\alpha\right)^{1/2}\\
&\ll \eta N \|g\|_2|\mathcal{S}|^{1/2},
\end{align*}
which implies
\begin{align*}
|\mathcal{S}|\ll  \eta^{2/3} N,
\end{align*}
as desired.
\end{proof}

Recall our notions of major and minor arcs from Definition~\ref{def_major}. We also have a slight variant of the previous lemma where we assume only \emph{major arc} control on $f$, but require additionally \emph{minor arc} control on $g$; this will be needed in the proof of Key Proposition~\ref{prop1b} later on.

\begin{lemma}\label{lem_6.2}
Let $\eta>0$ and $N\geq 1$. Let $f,g:[1,N]\to \mathbb{C}$ be functions. Suppose that we have
\begin{align}\label{e40}
\sup_{\alpha \in \mathfrak{M}}\left|\sum_{n\leq N}f(n)e(\alpha n)\right|\leq \eta N    
\end{align}
and
\begin{align}\label{e41}
 \sup_{\alpha \in \mathfrak{m}}\left|\sum_{n\leq N}g(n)e(\alpha n)\right|\leq \eta N.    
\end{align}
Then, for all but $\ll  \eta^{2/3} N$ integers $m\in [1,N]$, we have
\begin{align*}
|f*g(m)|\leq  \eta^{2/3} N^{1/2}(\|f\|_2+\|g\|_2).  
\end{align*}
\end{lemma}

\begin{proof}
The proof is the same as that of Lemma~\ref{le_fourier}, except that we split the integral arising in that proof as
\begin{align*}
 \int_{0}^{1}F(\alpha)G(\alpha)S(-\alpha)\, d\alpha=\int_{\mathfrak{M}}F(\alpha)G(\alpha)S(-\alpha)\, d\alpha+\int_{\mathfrak{m}}F(\alpha)G(\alpha)S(-\alpha)\, d\alpha   
\end{align*}
and estimate the first term using~\eqref{e40} and the second using~\eqref{e41}.
\end{proof}

To handle the major arc exponential sums, we shall apply the following lemma to reduce matters to multiplicative characters.

\begin{lemma}\label{le_character}
Let $\eta>0$, $N\geq 1$. Let $f:[1,N]\to \mathbb{C}$ be a function. Suppose that
\begin{align}\label{charbound}
\max_{1\leq \ell\leq P^{c_0}}\max_{y\leq N/\ell}\left|\sum_{n\leq y}f(\ell n)\chi(n)\right|\ll \eta \exp(-\sqrt{\log P})P^{-3c_0/2} N   
\end{align}
uniformly for all Dirichlet characters $\chi$ of modulus $\leq P^{c_0}$. Then we have
\begin{align}\label{e7b}
\sup_{\alpha \in \mathfrak{M}}\left|\sum_{n\leq N}f(n)e(\alpha n)\right|\ll \eta N.    
\end{align}
\end{lemma}

\begin{proof}
By Definition~\ref{def_major}, we can write $\alpha \in \mathfrak{M}$ as $\alpha=a/q+\beta$ with $1\leq a\leq q\leq P^{c_0}$, $(a,q)=1$ and $|\beta|\leq P^{c_0}/N$. By the fundamental theorem of calculus, we have
\begin{align}\label{e8}
e(\beta n)=1+2\pi i \beta\int_{0}^{N}e(\beta y)1_{y\leq n}\, dy.    
\end{align}
Writing $e(\alpha n)=e(an/q)e(\beta n)$ and substituting this and~\eqref{e8} into~\eqref{e7b}, we see that it suffices to prove
\begin{align*}
\max_{y\leq N}\left|\sum_{n\leq y}f(n)e\left(\frac{an}{q}\right)\right|\ll \eta P^{-c_0}N   
\end{align*}
uniformly for $1\leq a\leq q\leq P^{c_0}$ with $(a,q)=1$. Writing $\ell=(n,q)$, we need
\begin{align}\label{e7c}
\max_{y\leq N}\left|\sum_{\ell\mid q}\sum_{n'\leq y/\ell}f(\ell n')1_{(n',q/\ell)=1}e\left(\frac{an'}{q/\ell}\right)\right|\ll \eta P^{-c_0}N.    
\end{align}

We use the Fourier expansion
\begin{align*}
e\left(\frac{an}{q/\ell}\right)=\frac{1}{\varphi(q/\ell)}\sum_{\chi\pmod{q/\ell}}\tau(\bar{\chi})\chi(an)  
\end{align*}
for $(n,q/\ell)=1$, where $\tau(\chi)$ is the Gau{\ss} sum. Applying the classical Gau{\ss} sum bound $|\tau(\chi)|\leq (q/\ell)^{1/2}$ for $\chi\neq \chi_{0}^{(q/\ell)}$ and $|\tau(\chi_0^{(q/\ell)})|\leq 1$ and the triangle inequality,~\eqref{e7c} is 
\begin{align*}
\ll  \sum_{\ell\mid q}\frac{\ell^{1/2}(\log \log q)}{q^{1/2}}\max_{\ell\mid q}\max_{y\leq N/\ell}\left|\sum_{n\leq y}f(\ell n)\chi(n)\right| \end{align*}
for some character $\chi$ of modulus dividing $q$. We observe that
\begin{align*}
\sum_{\ell\mid q}\frac{\ell^{1/2}}{q^{1/2}}=\sum_{u\mid q}\frac{1}{u^{1/2}}&=\prod_{p\mid q}\left(1+p^{-1/2}+p^{-1}+p^{-3/2}+\cdots\right)\\
&\ll \exp\left(O\left(\sum_{p\mid q}p^{-1/2}\right)\right)\\
&\ll \exp(\sqrt{\log q}).
\end{align*}
Appealing to~\eqref{charbound} completes the proof.
\end{proof}

In the rest of this section, we shall prove a proposition (Proposition~\ref{prop1-generalized} below) that contains Key Proposition~\ref{prop1a}, but involves a slight generalisation to weights other than the von Mangoldt function, as that will be needed in the proof of Theorem~\ref{MT2}.  

\begin{definition}[Weighted indicator of $E_3$ numbers corresponding to Chen's sieve]
Denote (similarly as in~\cite[eq. (6.2)]{matomaki-shao})
\begin{align*}
B_1&\coloneqq \{n=p_1p_2p_3:\,\, N^{1/10}\leq p_1< N^{1/3-\varepsilon}<p_2\leq (N/p_1)^{1/2}, p_3\geq N^{1/10}\}\\
 B_2&\coloneqq \{n=p_1p_2p_3:\,\, N^{1/3-\varepsilon}\leq p_1\leq p_2 \leq (N/p_1)^{1/2}, p_3\geq N^{1/10}\}
\end{align*}
and define normalised indicator functions for these sets as
\begin{align*}
\Lambda_{B_i}(n)\coloneqq \begin{cases}
\log n,\quad n=p_1p_2p_3\in B_i\\
0,\quad \quad \quad \textnormal{otherwise},\end{cases}
\end{align*}
and write
\begin{align*}
	\Lambda_{E_3^*}(n)=\frac{1}{2}\Lambda_{B_1}(n)+\Lambda_{B_2}(n)
\end{align*}
\end{definition}

The reason for considering $\Lambda_{E_3^*}$ will become clear after Lemma~\ref{le_chensieve} (a version of Chen's sieve inequality). Expanding out the definition of $\Lambda_{E_3^{*}}(n)$ and applying the Siegel--Walfisz theorem, for any character $\chi$ of modulus $\leq (\log N)^{A}$ we have
\begin{align}\label{sw}
  \sum_{n\leq N}\Lambda_{E_3^{*}}(n)\chi(n)=c_{E_3^*}N1_{\chi=\chi_0}+O_{A}(N/(\log N)^{A}),
\end{align}
where the implied constant is ineffective and $c_{E_3^*}=\frac{1}{2} c_{B_1}+c_{B_2}$ with
\begin{align*}
  c_{B_1}\coloneqq \int_{\substack{1/10\leq t_1\leq 1/3-\varepsilon\leq  t_2\leq (1-t_1)/2\\1-t_1-t_2\geq 1/10}}\frac{d t_1\, dt_2}{t_1t_2(1-t_1-t_2)}\\
  c_{B_2}\coloneqq \int_{\substack{1/3-\varepsilon\leq t_1\leq  t_2\leq (1-t_1)/2\\ 1-t_1-t_2\geq 1/10}}\frac{d t_1\, dt_2}{t_1t_2(1-t_1-t_2)}. 
\end{align*}   

In the rest of this section, let
\begin{align*}
\Lambda^{*}\in \{\Lambda,\Lambda_{E_3^*}\}.    
\end{align*}

\begin{prop}\label{prop1-generalized}
 Let $k\geq 1$, $a\neq 0$ and $\varepsilon \in (0,1/1000)$ be fixed. Let $N\geq 3$ and $(\log N)^{C}\leq P\leq N^{\varepsilon/10}$. Let
\begin{align*}
f(n)\coloneqq \Lambda^*(n)\omega_1(n+a)\omega_2(n+a),
\end{align*}
where $\omega_1$ is a sieve of range $P$, level $D_0\leq N^{\varepsilon/2}$ and order $k$ and $\omega_2$ is an admissible main sieve with parameters $P,\varepsilon,k$. Let $g:[1,N]\to \mathbb{C}$ be any function. Then, for all $m\in [1,N]$ apart from $\ll N/P^{c_1}$ exceptions, we have
\begin{align*}
f*g(m)=\mathcal{V}(\omega_2)f_{\textnormal{pre-sieve}}*g(m)+O(\|g\|_2N^{1/2} P^{-c_1}),
\end{align*}
where 
\begin{align*}
f_{\textnormal{pre-sieve}}(n)\coloneqq \Lambda^*(n)\omega_1(n+a).
\end{align*}
\end{prop}

Note that for $\Lambda^{*}=\Lambda$ and $a=2$, this implies Key Proposition~\ref{prop1a} (after adjusting $\varepsilon$).

 From Lemma~\ref{le_fourier}, we see that  Proposition~\ref{prop1-generalized} follows if we show that
\begin{align}\label{e2}
\sup_{\alpha\in \mathbb{R}}\left|\sum_{n\leq N}(f(n)-f_{\textnormal{pre-sieve}}(n))e(\alpha n)\right|\ll NP^{-3c_1/2}.
 \end{align}
 We shall prove~\eqref{e2} by considering the case of minor and major arc $\alpha$ separately. Recall the splitting
\begin{align*}
[0,1)=\mathfrak{M}\cup \mathfrak{m},    
\end{align*}
where the major and minor arcs are given by Definition~\ref{def_major}.

By Lemma~\ref{le_character}, it suffices to prove
 \begin{align}\label{e2_minor}
\sup_{\alpha\in \mathfrak{m}}\left|\sum_{n\leq N}f(n)e(\alpha n)\right|\ll NP^{-3c_1/2},\quad \sup_{\alpha\in \mathfrak{m}}\left|\sum_{n\leq N}f_{\textnormal{pre-sieve}}(n)e(\alpha n)\right|\ll NP^{-3c_1/2},
 \end{align}
 and
 \begin{align}\label{e2_major}
 \max_{\substack{\chi\pmod q\\q\leq P^{c_0}}}\max_{1\leq \ell\leq P^{c_0}}\max_{y\leq N}\left|\sum_{n\leq y}(f(\ell n)-f_{\textnormal{pre-sieve}}(\ell n))\chi(n)\right|\ll NP^{-3c_0/2-2c_1}.    
 \end{align}
 
\subsection{Bombieri--Vinogradov with power saving}

The large sieve allows us to estimate efficiently the correlation of a sequence with characters of large conductor on average.

	\begin{lemma}\label{lem_BDHrough}
	Let $Q\geq P\geq 2$. Let $\alpha_n$ be any sequence on $(M,M+N]$ with $M,N\geq 1$. Then 
	\begin{align*}
	\sum_{q\leq Q} \frac{1}{\varphi(q)}\sum_{\substack{\psi \pmod q \\ \cond(\psi)>P}} \Bigl|\sum_{M<n\leq M+N}\alpha_n \psi(n) \Bigr|^2 \ll (Q+N/P) (\log Q)\sum_{M<n\leq M+N}|\alpha_n|^2. 
	\end{align*}

\end{lemma}

\begin{proof}
	
	Writing $q=rs$ with $\psi=\psi_1\psi_2$, where $\psi_1\pmod r$ is primitive and $\psi_2\pmod s$ is principal, we can bound the sum in question with the large sieve inequality for multiplicative characters (see~\cite[(7.31)]{iw-kow}) by
	\begin{align*}
	\sum_{s\leq Q}\frac{1}{\varphi(s)}\sum_{P\leq r \leq Q}\frac{1}{\varphi(r)} \sum_{\psi_1 (r)^*} \Bigl|\sum_{n\leq N}\alpha_n1_{(n,s)=1} \psi_1(n) \Bigr|^2&\ll (Q+N/P) (\log Q)\sum_{M<n\leq M+N}|\alpha_n|^2 
	\end{align*}
	as stated.
\end{proof}

The Bombieri--Vinogradov theorem states that the approximation
\begin{align}\label{eq_BVapprox}
\sum_{\substack{n\leq N \\ n\equiv a(d)}}\Lambda(n)\approx \frac{1}{\varphi(d)}\sum_{\substack{n\leq N \\ }}\Lambda(n)
\end{align}
is valid on average for $d$ up to almost $N^{1/2}$ with a saving over the trivial bound of type $(\log N)^{-A}$. This saving comes from the application of the Siegel--Walfisz theorem for small conductor characters and the large sieve for large conductor characters. The term on the right-hand side of~\eqref{eq_BVapprox} is the contribution of the principal character modulo $d$, i.e. the unique character $\pmod{d}$ that has conductor $1$. If we instead include the contribution of all the characters with conductor up to a value of $P=(\log N)^{\Omega(1)}$, we can hope to get a better error term than $(\log N)^A$ with the help of Lemma~\ref{lem_BDHrough}. 

This idea was applied by Drappeau, and following the notation of~\cite[eq. (5.1)]{drappeau} we write
\begin{align}\label{eq_uPdef}
\mathfrak{u}_P(n;q)&\coloneqq  \frac{1}{\varphi(q)}\sum_{\substack{\psi(q)\\ \text{cond}(\psi)> P}}\psi(n)\\
&=	1_{n\equiv 1(q)}-\frac{1}{\varphi(q)}\sum_{\substack{\psi(q)\\ \text{cond}(\psi)\leq P}}\psi(n), \nonumber
\end{align}
With this we can state the following type I and II estimate (see also~\cite[Lemma 5.2]{drappeau}).
\begin{lemma}[Simple type I/II estimate]\label{le_I,IIsimple}
	Let $C\geq 1$ be fixed. Let $N_1, N_2,y \geq 1$ with $N_1N_2\asymp y$ and let $2\leq P\leq Q\leq y^{1/2}$. Then for any character $\chi$  with modulus at most $P$ we have 
	\begin{align}\label{eq_Isimple}
	\sum_{q\leq Q}\sum_{n_1\leq N_1} \tau(n_1)^C \max_{\substack{(a,q)=1\\ y_0\geq 1}}\left|\sum_{y_0\leq  n_2\leq y/n_1}\chi(n_2)\mathfrak{u}_{P}(n_1n_2\overline{a},q) \right|\ll_C N_1 Q^{3/2}P^{1/2}(\log y)^{O_C(1)}. 
	\end{align}
	Let further $\alpha_n$, $\beta_n$ be coefficient sequences of order $C$, supported on $[1,N_1]$ and $[1,N_2]$, respectively. Then it holds that
	\begin{align}\label{eq_IIsimple} 
	&\sum_{q\leq Q}\max_{(a,q)=1}\left|\sum_{n_1,n_2}\alpha_{n_1}\beta_{n_2} \mathfrak{u}_{P}(n_1n_2\overline{a};q)\right|
	\ll_C y^{1/2}(Q+y^{1/2}/P+N_1^{1/2}+N_2^{1/2})(\log y)^{O_C(1)}.
	\end{align}
\end{lemma}
\begin{proof}
	We have
	\begin{align*}
	\sum_{\substack{y_0\leq n_2\leq y/n_1}}\chi(n_2)\mathfrak{u}_{P}(n_1n_2\overline{a},q)=\frac{1}{\varphi(q)}\sum_{\substack{\psi(q)\\ \text{cond}(\psi)>P}} \psi(n_1\overline{a}) \sum_{y_0\leq n_2\leq y/n_1}\chi\psi(n_2)
	\end{align*}
	Since $\chi$ has modulus at most $P$, the character $\chi \psi $ (of modulus $\leq PQ$) is never principal. Applying the P\'olya--Vinogradov inequality and trivially bounding the number of $\psi(q),  \text{cond}(\psi)>P$ by $\varphi(q)$ proves~\eqref{eq_Isimple}.
	
	The bound~\eqref{eq_IIsimple} follows in a straightforward manner from the Cauchy--Schwarz inequality and Lemma~\ref{lem_BDHrough}.
\end{proof}

As a consequence, we get the following version of the Bombieri--Vinogradov theorem with an additional character twist and improved error term.

\begin{lemma}[Bombieri--Vinogradov with large savings]\label{le_BV_rough}
	Let $\chi$ be any character of modulus at most $P$, with $1\leq P\leq Q\leq N^{1/2}/P$. Then for any $y\leq N$ we have 
	\begin{align*}
	\sum_{q\leq Q}\max_{\substack{a(q)^* \\}} \left|\sum_{n\leq y} \Lambda^*(n)\chi(n)\mathfrak{u}_{P}(n\overline{a};q) \right|\ll N P^{-1/4} (\log N)^{O(1)}.
	\end{align*}
\end{lemma}

\begin{proof}
	We can assume that $y\geq N/P^{3}$, since otherwise a trivial triangle inequality estimate suffices. Applying Vaughan's identity in the case $\Lambda^*=\Lambda$ (or in the case $\Lambda^*=\Lambda_{E_3^*}$ the fact that each $n$ in the support of $\Lambda_{E_3^*}$ has a prime factor in $[N^{1/4},N^{1/2}]$), it suffices to prove that 
	\begin{align}\label{e38}\begin{split}
	\sum_{\substack{d\leq Q\\(d,a)=1}}\max_{\substack{a(q)^* \\}}\left|\sum_{\substack{n_1n_2\leq y\\n_1\sim N_1, }}a_{n_1}\chi(n_1)b_{n_2}\chi(n_2)\mathfrak{u}_{P}(n_1n_2\overline{a};q)\right|\ll N P^{-1/4} (\log N)^{O(1)},
	\end{split}
	\end{align}
	whenever either of the following holds:
	\begin{enumerate}
		\item $N_1\leq P^{3/2}$, $|a_n|\leq \tau(n)$ and $b_n\equiv 1$ or $b_n\equiv \log n$ (Type I case);
		\item $N_1\in [P^{3/2},y/P^{3/2}]$ and $|a_n|,|b_n|\leq \tau(n)\log n$ (Type II case).
	\end{enumerate}
	
	In the type I case, we may pull the sum over $n_1$ outside with the triangle inequality and apply Lemma~\ref{le_I,IIsimple}, obtaining for~\eqref{e38} a bound of
	\begin{align*}
	  \ll P^{3/2}Q^{3/2}P^{1/2}(\log y)^{O(1)}\ll P^{1/2}N^{3/4}(\log y)^{O(1)}\ll NP^{-1/4}(\log y)^{O(1)}.
	\end{align*}
	
	In the type II case, we first split the $n_1$ and $n_2$ variables in~\eqref{e38} into short intervals of logarithmic length $P^{-1/4}$, reducing matters to bounding
	\begin{align}\label{e38b}\begin{split}
	&\sum_{\substack{N_1',N_2'\in [P^{3/2},y/P^{3/2}]\\N_1'N_2'\leq y(1+P^{-1/4})^2\\N_i'=(1+P^{-1/4})^{j_i}\,\, \textnormal{for some } j_i\in \mathbb{N}}}\sum_{\substack{d\leq Q\\}}\max_{a(d)^*}\Bigg|\sum_{\substack{n_1n_2\leq N\\n_1\sim N_1\\N_1'\leq n_1\leq N_1'(1+P^{-1/4})\\N'\leq n\leq N'(1+P^{-1/4})}}a_{n_1}\chi(n_1)b_{n_2}\chi(n_2)\mathfrak{u}_{P}(n_1n_2\overline{a};q) \Bigg|.
	\end{split}
	\end{align}
	The contribution of $y/(1+P^{-1/4})^2\leq N_1'N_2'\leq N(1+P^{-1/4})^2$ is trivially admissible by the triangle inequality. This allows us to delete the cross condition $n_1n_2\leq y$ and the estimate~\eqref{e38b} follows from applying~\eqref{eq_IIsimple} of the previous lemma to the $\ll P^{1/2}(\log P)^2$ many short sums. 
\end{proof}

\subsection{Minor arc contribution}\label{sub: minor}

Our task in this subsection is to prove~\eqref{e2_minor}. Recall that in the statement of Proposition~\ref{prop1-generalized} we have $f(n)=\Lambda^{*}(n)\omega_1(n+a)\omega_2(n+a)$, where $\Lambda^{*}\in \{\Lambda, \Lambda_{E_3^*}\}$. 
Write $\omega_i(n)=\sum_{d\mid n}\lambda_{i}(d)$ with $|\lambda_i(d)|\leq \tau_k(d)$. Then the two exponential sums in~\eqref{e2_minor} that we need to bound become
\begin{align*}
&\sum_{d_1\leq D_0}\lambda_1(d_1)\sum_{d_2\leq N^{1/2-\varepsilon}}\lambda_2(d_2)\sum_{\substack{n\leq N\\n\equiv -a\pmod{ d_1d_2}}}\Lambda^{*}(n)e(n\alpha),\\
&\mathcal{V}(\omega_2)\sum_{d_1\leq D_0}\lambda_1(d)\sum_{\substack{n\leq N\\n\equiv -a\pmod{ d_1}}}\Lambda^{*}(n)e(n\alpha). \end{align*}
By the assumption on $\omega_2$, we may write
\begin{align*}
\lambda_1\star \lambda_2=\sum_{j\leq C\log N}(\lambda_1\star \alpha_j)\star \lambda_j'\coloneqq \sum_{j\leq C\log N}\gamma_j\star \lambda_j',    
\end{align*}
where $|\gamma_j(d)|\leq \tau_{k+1}(d)$, $|\lambda_j'(d)|\leq \tau_k(d)$, $\gamma_j$ is supported on $[N^{t_j},N^{t_j+\varepsilon/2}]$ for some $0\leq t_j\leq 1/3-\varepsilon$, and $\lambda_j'$ is well-factorable of level $N^{1/2-t_j-\varepsilon}$. Therefore, to conclude the minor arc analysis and show $\eqref{e2_minor}$, it suffices to prove the following result (recall $c_1=c_0/100$).  

\begin{lemma}[Bombieri--Vinogradov with minor arc twist and partially well-factorable weights]\label{le_matomaki} Let $\varepsilon \in (0,1/100)$ and $k\geq 1$ and $a\in \mathbb{Z}\setminus\{0\}$ be fixed. Let $(\log N)^{C}\leq P\leq N^{\varepsilon/10}$ with $C$ large enough in terms of $k$. Let $\xi_1,\xi_2$ satisfy
$|\xi_i(d)|\leq \tau_k(d)$ and suppose that $\xi_1$ is supported on $[1,D_1]$ for some $D_1\leq N^{1/3-\varepsilon/2}$  and that $\xi_2$ is well-factorable of level $N^{1/2-\varepsilon}/D_1$. Then, we have 
\begin{align}\label{e4}
\sup_{\alpha\in \mathfrak{m}}\left|\sum_{d_1}\xi_1(d_1)\sum_{\substack{d_2\\(d_1d_2,a)=1}}\xi_2(d_2)\sum_{\substack{n\leq N\\n\equiv a\pmod{d_1d_2}}}\Lambda^{*}(n)e(n\alpha)\right|\ll NP^{-c_0/20}.    
\end{align}
\end{lemma}

\begin{proof} By the Cauchy--Schwarz inequality, standard moment estimates for the divisor functions, and the fact that $P\geq (\log N)^{C}$ for $C$ large enough, it suffices to prove~\eqref{e4} in the case $|\xi_i(d)|\leq 1$, with the stronger bound $NP^{-c_0/9}$ on the right-hand side.   Note also that the logarithmic weight in the definition of $\Lambda^{*}$ can be disposed of with partial summation, replacing the summation over $n\leq N$ in~\eqref{e4} with summation over $n\leq N'$ for some $N'\leq N$. We may assume that $N'\geq N^{1-\varepsilon/10}$, as otherwise the claim follows from the Brun--Titchmarsh inequality.

After these reductions, the result follows from work of Matom\"aki~\cite{matomaki} and is proved in detail in~\cite{matomaki-shao}. See~\cite[ Lemmas 8.3, 8.4, 8.6]{matomaki-shao}. The considerations there give a saving of $P^{-c_0/8} (\log N)^C\ll P^{-c_0/9}$. Observe that the weight that we call $\xi_1$ is in~\cite{matomaki-shao} assumed to be supported on primes only (as we could also assume), but this is not required in the proof.  
\end{proof}

\subsection{Major arc contribution}\label{subsec:major}

We shall now complete the proof of Proposition~\ref{prop1-generalized} (and so also that of Key Proposition~\ref{prop1a}) by showing~\eqref{e2_major}. Let $\lambda(d)\coloneqq \lambda_1 \star (\lambda_2-\mathcal{V}(\omega_2)I)$, where $I(d)\coloneqq 1_{d=1}$. Note that for $2\leq \ell\leq P^{c_0}$ we have $\Lambda_{E_3}^{*}(\ell n)=0$, and $\Lambda(\ell n)=0$ unless $n$ is of the form $p^j$ with $j\geq 2$ and $p\mid \ell$. Therefore, our task is to prove 
\begin{align*}
\left|\sum_{(d,a)=1}\lambda(d)\sum_{\substack{n\leq y\\n\equiv -a\pmod{d}}}\Lambda^*(n)\chi(n)\right|\ll N P^{-3c_0/2-2c_1}
\end{align*}
uniformly for $y\leq N$ and characters $\chi$ of modulus $\leq P^{c_0}$. 

Since the two sieve weights $\lambda_1$ and $\lambda_2$ are supported only on integers consisting of primes $\leq P$ and $>P$, respectively, the sum that we are considering can be rewritten as
\begin{align*}
\sum_{(d,a)=1}\lambda_1\star \lambda_2(d)\sum_{\substack{n\leq y\\}}\Lambda^*(n)\chi(n)\left(1_{n\equiv -a(d)}-\frac{1_{n\equiv -a(d_{\leq P})}}{\varphi(d_{>P})}\right).
\end{align*}

To make Lemma~\ref{le_BV_rough} applicable, we rewrite for any $d$ with $(d,a)=1$ (denoting by $\chi_{0}^{(d_{>P})}$ the principal character $\pmod{d_{>P}}$ and by $\bar{a}$ the inverse of $a$ modulo $d$)
\begin{align*}\nonumber 
&1_{n\equiv -a\pmod{d}}-\frac{1_{n\equiv -a \pmod{d_{\leq P}}}}{\varphi(d_{>P})}\\
&=\nonumber\frac{1}{\varphi(d)}\sum_{\psi_1(d)}\psi_1(\overline{-a}n)-\frac{1}{\varphi(d)}\sum_{\psi_2(d_{\leq P})}\psi_2(\overline{-a}n)\chi_{0}^{(d_{>P})}(\overline{-a}n)-\frac{1_{(n,d_{>P})\neq 1}1_{n\equiv -a \pmod{d_{\leq P}}}}{\varphi(d_{>P})}\\
&=\nonumber \frac{1}{\varphi(d)}\sum_{\substack{\psi(d)\\ \text{cond}(\psi)>P}}\psi(\overline{-a}n)+O\Big(\frac{1_{(n,d_{>P})\neq 1}}{\varphi(d_{>P})}\Big)\\
&=\nonumber \mathfrak{u}_P(\overline{-a}n;d)+O\Big(\frac{1_{(n,d_{>P})\neq 1}}{\varphi(d_{>P})}\Big).
\end{align*}

The contribution of terms with $(n,d_{>P})\neq 1$ can be estimated trivially with the Brun--Titchmarsh inequality (noting that $(n,d_{>P})\neq 1$ implies the existence of a prime $p>P$ such that $p\mid n$, $p\mid d$). Therefore, after an application of the triangle inequality, we have reduced matters to showing
\begin{align*}
\left|\sum_{(d,a)=1}\lambda_1\star\lambda_2(d)\sum_{\substack{n\leq y\\}}\Lambda^*(n)\chi(n)\mathfrak{u}_P(\overline{-a}n;d)\right| \ll N P^{-3c_0-2c_1}.
\end{align*}

We can remove the weights $\lambda_1\star \lambda_2(d)$ with the Cauchy--Schwarz inequality; we thus reduce to proving 
\begin{align}\label{eq_prop1final_2}
\sum_{\substack{d\leq N^{1/2-\varepsilon}\\(d,a)=1}}\left|\sum_{\substack{n\leq y\\}}\Lambda^*(n)\chi(n)\mathfrak{u}_P(\overline{-a}n;d)\right| \ll N P^{-4(c_0+c_1)},
\end{align}
say. As the $d$ summation goes up to $ N^{1/2-\varepsilon}$, ~\eqref{eq_prop1final_2} follows from Lemma~\ref{le_BV_rough}. This completes the proof of~\eqref{e2_major} and therefore also that of Proposition~\ref{prop1-generalized} and Key Proposition~\ref{prop1a}.

\section{The case of two Chen primes -- Proof of Theorem~\ref{MT2}}

We shall next prove Theorem~\ref{MT3}. In addition to Proposition~\ref{prop1-generalized}, we will need the following lemma that allows us to replace $\Lambda$ and $\Lambda_{E_3^*}$ by pre-sieves in the range of saving of Theorem~\ref{MT3}. 

\begin{lemma}\label{le_bv_chen} Let $\varepsilon>0$ be fixed, and let $B\geq A\geq 1$ be large but fixed. Let $N\geq 2$ and let $(\log N)^{A^2}\leq P\leq (\log N)^{B}$.  Let $\omega^+$ be the upper bound admissible pre-sieve with parameters $P, N^{\varepsilon}$ as in Definition~\ref{def_admps}. Then we have  
	\begin{align}\label{eqqn2}
		\sup_{\alpha \in \mathbb{R}}\left|\sum_{n\leq N}(\Lambda(n)-V(P)^{-1}\omega^+(n)) \omega^+(n+2)e(\alpha n)\right|\ll N/(\log N)^{A}    
	\end{align}
	and
	\begin{align}\label{eqqn3}
		\sup_{\alpha \in \mathbb{R}}\left|\sum_{2<n\leq N}(\Lambda_{E_3^*}(n)-c_{E_3^*}V(P)^{-1}\omega^+(n)) \omega^+(n-2)e(\alpha n)\right|\ll N/(\log N)^{A} . 
	\end{align}
Here the implied constants are ineffective.
\end{lemma}

\begin{proof}
	Let $\Lambda^{*}\in \{\Lambda,\Lambda_{E_3^{*}}\}$, and let $c^{*}=1$ if $\Lambda^{*}=\Lambda$ and $c^{*}=c_{E_3^{*}}$ if $\Lambda^{*}=\Lambda_{E_3}$.
	
	In the case $\alpha\in \mathfrak{m}$, for $a\in \{-2,2\}$ we have
	\begin{align*}
	\sup_{\alpha \in \mathfrak{m}}\left|\sum_{2<n\leq N}\Lambda^{*}(n) \omega^+(n+a)e(\alpha n)\right|\ll N/(\log N)^{A}    
	\end{align*}
	by Lemma~\ref{le_matomaki}, since $\omega^{+}(n+a)=\sum_{d\mid n+a, d\leq N^{\varepsilon}}\lambda_d$ with $|\lambda_d|\ll 1$. Similarly, by slight modification of~\cite[Lemma 8.3]{matomaki-shao}, we have
	\begin{align*}
	\sup_{\alpha \in \mathfrak{m}}\left|\sum_{2<n\leq N}\omega^{+}(n) \omega^+(n+a)e(\alpha n)\right|\ll N/(\log N)^{A}    
	\end{align*}
	
	In the case $\alpha\in \mathfrak{M}$, by Lemma~\ref{le_character}, it suffices to prove
	\begin{align*}
	\left|\sum_{2<n\leq y}(\Lambda^{*}(\ell n)-V(P)^{-1}\omega^+(\ell n)) \omega^+(\ell n+a)\chi(n)\right|\ll N/(\log N)^{A}
	\end{align*}
	uniformly for $1\leq \ell\leq P^{c_0}$, $y\leq N/\ell$, $a\in \{-2,2\}$ and characters $\chi$ of modulus $\leq P^{c_0}$.
	
	We will reduce to the case $\ell=1$. Observe that for $1<\ell\leq P^{c_0}$, we trivially have
	\begin{align*}
	   \sum_{2<n\leq N}\Lambda^{*}(\ell n)\ll N^{1/2}(\log N), 
	\end{align*}
	and by H\"older's inequality and the estimate $\omega^{+}(n)\leq \tau(n)$ we have
	\begin{align*}
	 \sum_{2<n\leq N}\omega^{+}(\ell n)\omega^{+}(\ell n+a)\leq \left(\sum_{n\leq N}\omega^{+}(\ell n)\right)^{1/2}\left(\sum_{n\leq N}\tau(\ell n)^4\right)^{1/4}\left(\sum_{2<n\leq N}\tau(\ell n+a)^4\right)^{1/4}.   
	\end{align*}
	The second and third sum on the right are $\ll N(\log N)^{O(1)}$ by Shiu's bound~\cite[Theorem 1]{shiu}. The first sum in turn is by the fundamental lemma
	\begin{align*}
	   \sum_{d\leq D}\lambda_d\frac{\ell N}{[d,\ell]}+O(D)=N\prod_{p\leq P}\left(1-\frac{\ell}{[\ell,p]}\right)+O(Ne^{-(\log D)/(\log P)}+D)\ll Ne^{-(\log N)/(\log \log N)^2},  
	\end{align*}
	say. Hence, we may assume that $\ell=1$.
	
	Now if $\chi$ is non-principal of modulus $q\leq P^{c_0}$, the Siegel--Walfisz theorem and~\eqref{sw} give
	\begin{align*}
	  \left|\sum_{2<n\leq y}\Lambda^{*}(n)\chi(n)\right|\ll N/(\log N)^{A}.  
	\end{align*}
	Moreover, we have
	\begin{align}\label{eq:omega+}
	 \left|\sum_{2<n\leq y}\omega^{+}(n)\omega^{+}(n+a)\chi(n)\right|\leq \sum_{\substack{d_1,d_2\leq N^{\varepsilon}\\(d_1,q)=1}}\left|\sum_{\substack{2<n\leq y\\n\equiv 0\pmod{d_1}\\n\equiv -a\pmod{d_2}}}\chi(n)\right|,   
	\end{align}
	and the inner sum is $\ll qd_1d_2$, since for any $u\geq 1, v$ coprime to $q$ we have
	\begin{align*}
	 \left|\sum_{n\leq y}\chi(un+v)\right|=\left|\frac{1}{\varphi(u)}\sum_{\psi(u)}\overline{\psi}(-v)\sum_{n\leq uy+v}\chi\psi(n)\right|\leq qu.   
	\end{align*} 
	Therefore,~\eqref{eq:omega+} is small enough.
	
	Finally, assume that $\chi=\chi_0$ is principal.  By the prime number theorem,~\eqref{sw} and the fundamental lemma, we have
	\begin{align*}
	 \sum_{2<n\leq y}\Lambda^{*}(n)\omega^{+}(n+a)\chi_0(n)=\sum_{d\leq N^{\varepsilon}}\lambda_d c^{*}\frac{y}{\varphi(d)}+O(N/(\log N)^{A})=c^{*}V(P)y+O(N/(\log N)^{A}).  
	\end{align*}
	On the other hand, the fundamental lemma also gives
	\begin{align*}
	\sum_{2<n\leq y} \omega^{+}(n)\omega^{+}(n+a)\chi_0(n)=(1+O((\log N)^{-A}))V(P)^2y,   
	\end{align*}
	and combining these we obtain the claim.
\end{proof}

We shall lastly need a well-known sieve inequality that is a simplified version of Chen's sieve.

\begin{lemma}[Chen's sieve]\label{le_chensieve} Let $\Lambda_2(n)=\Lambda(n)1_{\mathbb{P}_2}(n+2)\rho(n+2,N^{1/15})$ as before. Let $\varepsilon>0$ be a small enough fixed number. Let $P\leq e^{\sqrt{\log N}}$ and $D_0=N^{\varepsilon^2}$. Then there exist admissible main sieves $\omega_{\textnormal{M}}^{+}$, $\omega_{\textnormal{M}}^-$ with parameters $P,\varepsilon,2$, and an admissible pre-sieve $\omega^{+}$ with parameters $P, D_0$ such that

\begin{itemize}
    \item $\Lambda_2(n)\geq \Lambda(n)\omega^+(n+2)\omega_{\textnormal{M}}^-(n+2)-\omega^+(n)\omega_{\textnormal{M}}^+(n)\Lambda_{E_3^*}(n+2)+O(\mathcal{E}(n))$\\ with $\sum_{n\leq N} |\mathcal{E}(n)|\ll N e^{-\sqrt{\log N}}$, 
    \item $\mathcal{V}(\omega_{\textnormal{M}}^{-})-c_{E_3^*}\mathcal{V}(\omega_{\textnormal{M}}^{+})\gg V(P,N)$. 
\end{itemize}
\end{lemma}

\begin{proof} 
	Without the pre-sieves, this follows directly from the construction in~\cite[Appendix A]{matomaki-shao}. By restricting the sieves constructed  there to primes larger than $P$ we can incorporate an upper bound pre-sieve into the second summand in a straightforward manner.
	
	In the first summand this is done similarly as in the proof of Theorem~\ref{MT1} by the vector sieve inequality, Lemma~\ref{lem_lowerboundsievcomp}, with $A^{-}, A^+$ being lower and upper bound pre-sieves with parameters $P, D_0$.  In contrast to the proof of Theorem~\ref{MT1}, we can now  bound the contribution of the  $A^- - A^+$ term in a straightforward manner and include it in the sieve error term. Indeed, the fundamental lemma and the estimate $|\omega_{\textnormal{M}}^{-}(n)|\ll e^{O((\log N)/(\log P))}$ give us
	\begin{align*}
&\bigl| \sum_{n\leq N} \Lambda(n)(\omega^+(n+2)-\omega^-(n+2))\omega_{\textnormal{M}}^-(n+2)\bigr|\\
&\ll (\log N)e^{O(\frac{\log N}{\log P})}\sum_{n\leq N} (\omega^+(n+2)-\omega^-(n+2))\\
&\ll (\log N)e^{O(\frac{\log N}{\log P})}\exp\Big(-\frac{1}{10}\frac{\log D_0}{\log P}\log \frac{\log D_0}{\log P}\Big)N\\
&\ll Ne^{-\sqrt{\log N}}.
\end{align*}
	
	This, and the related term with $\Lambda_{E_3^*}$, can be absorbed in the function $\mathcal{E}(n)$. 
\end{proof}

\begin{remark}
	We would not need to choose $\beta=750$ in the definition of admissible pre-sieves for this proposition to hold, as the relatively small choice of $P$ makes the fundamental lemma saving much stronger. We keep $\beta=750$ only to not increase the number of employed sieves even further.
\end{remark}

We are now ready to prove Theorem~\ref{MT2}.
\begin{proof}[Proof of Theorem~\ref{MT2}]
We allow constants in this proof to be ineffective.

Set $P=(\log N)^{A^2}$. For brevity, for any function $f$ denote $f'(n)\coloneqq f(n+2)$. Recall that we want to obtain an exceptional set of size $O(N(\log N)^{-A})$, with the constant being ineffective. By Lemma~\ref{le_chensieve}, we can now estimate
\begin{align*}
	\Lambda_2*\Lambda_2(m)\geq (\Lambda {\omega^{+}}'{\omega_{\textnormal{M}}^-}' -\omega^+\omega^+_{\textnormal{M}}\Lambda_{E_3^{*}}^{'})*\Lambda_2(m)+O\left(\frac{N}{(\log N)^{A}}\right).
\end{align*}

By Proposition~\ref{prop1-generalized}, we can reduce to pre-sieves with parameters $P,N^{\varepsilon}$, and we have outside a sufficiently small exceptional set
\begin{align*}
	(\Lambda {\omega^{+}}'{\omega_{\textnormal{M}}^-}' -\omega^+\omega^+_{\textnormal{M}}\Lambda_{E_3^{*}}^{'})*\Lambda_2(m)= (\Lambda {\omega^{+}}'\mathcal{V}(\omega_{\textnormal{M}}^-)-\omega^+\Lambda_{E_3^{*}}^{'}\mathcal{V}(\omega^+_{\textnormal{M}}))*\Lambda_2(m)+O\left(\frac{N}{(\log N)^{A}}\right).
\end{align*}
We can further simplify by applying Lemmas~\ref{le_bv_chen} and~\ref{le_fourier} to get outside the exceptional set that 
\begin{align*}
	&(\Lambda{\omega^{+}}'\mathcal{V}(\omega_{\textnormal{M}}^-)-\omega^+\Lambda_{E_3^{*}}^{'}\mathcal{V}(\omega^+_{\textnormal{M}}))*\Lambda_2(m)
	\\
	&=V(P)^{-1}(\mathcal{V}(\omega_{\textnormal{M}}^-)-c_{E_3^*}\mathcal{V}(\omega_{\textnormal{M}}^+))\omega^+{\omega^+}'*\Lambda_2(m)+O\left(\frac{N}{(\log N)^{A}}\right)
\end{align*}
Here $\omega^+(n)\omega^+(n+2)$ consists of two upper bound sieves and so is nonnegative. So we can lower bound the remaining $\Lambda_2$ in the same way and get outside of a sufficiently small exceptional set
\begin{align}\label{eq:omega_lower}
&V(P)^{-1}(\mathcal{V}(\omega_{\textnormal{M}}^-)-c_{E_3^*}\mathcal{V}(\omega_{\textnormal{M}}^+))\omega^+{\omega^+}'*\Lambda_2(m)\\
&\geq V(P)^{-2} (\mathcal{V}(\omega_{\textnormal{M}}^-)-c_{E_3^*}\mathcal{V}(\omega_{\textnormal{M}}^+))^2 \omega^+{\omega^+}'*\omega^+{\omega^+}'(m)+O\left(\frac{N}{(\log N)^{A}}\right).\nonumber
\end{align}

By the second statement of Lemma~\ref{le_chensieve} and Mertens's theorem, we have $\mathcal{V}(\omega_{\textnormal{M}}^-)-c_{E_3^*}\mathcal{V}(\omega_{\textnormal{M}}^+)\gg V(P,N)\asymp (\log P)/(\log N)$. Moreover, also by Mertens's theorem we have $V(P)^{-2}\asymp (\log P)^2$. 

By a simple calculation (that is a simpler case of our considerations in Section~\ref{sec:majarcMT}) and the fundamental lemma, one sees for all $m\leq N$ that
\begin{align*}
	 \omega^+{\omega^+}'*\omega^+{\omega^+}'(m)\gg V(P)^4\mathfrak{S}(m)m+O\left(\frac{N}{(\log N)^{A}}\right).
\end{align*}
As $\mathfrak{S}(m)\gg 1$ for $m\equiv 4\pmod 6,$ combining this with~\eqref{eq:omega_lower}  completes the proof.
\end{proof}

\section{Beyond the \texorpdfstring{$1/2$}{1/2} barrier -- Proof of Key Proposition~\ref{prop1b}}
\label{sec:beyond12}
	
	\subsection{Setup}
	In this section we prove Key Proposition~\ref{prop1b} which states that we can replace the main sieve component of a Fouvry--Grupp sieve by a constant on the major arcs. To do this, we extend the results of Bombieri--Friedlander--Iwaniec~\cite{bfi} and Maynard~\cite{MaynardII} to produce quantitatively stronger savings by including the contribution of small conductor characters in the main term, in the spirit of Lemma~\ref{le_BV_rough}. The proof strategy for these results is to use the dispersion method to translate the problem into sums of Kloosterman sums. Those then can be handled in an efficient manner by appealing to the work of Deshouillers--Iwaniec~\cite{di} on the spectral theory of automorphic forms. A similar power-saving result was obtained by Drappeau in~\cite{drappeau} in one special case of the dispersion method. We generalise his strategy to handle most of the arithmetic information in~\cite{bfi} and~\cite{MaynardII}. While our application to a Fouvry--Grupp sieve as in Definition~\ref{Def_FGsieve} is somewhat special and motivated by Theorem~\ref{MT1}, we in particular also obtain all cases that are required for the important $4/7$ and $3/5$ results of~\cite{bfi} and~\cite{MaynardII}. Thus, as a side product we also obtain Theorem~\ref{MT3}.
	
	We make extensive use of the $\mathfrak{u}_P$ notation from~\eqref{eq_uPdef} and note the trivial bound 	(see also~\cite[eq. (5.2)]{drappeau})
	\begin{align}\label{eq_trivbound}
	|\mathfrak{u}_P(n;q)|\ll 1_{n\equiv 1(q)}+\frac{P\tau(q)}{q},
	\end{align}
	which comes from the fact that $q$ has $\leq \tau(q)$ divisors $\leq P$ and to each such modulus there are $\leq P$ characters.
	
	We start by reducing Key Proposition~\ref{prop1b} to the following result that involves the distribution of primes to moduli beyond the $1/2$ barrier with certain factorability properties and a character twist.
	
	\begin{prop}[Power-saving level of distribution estimates beyond the $1/2$ barrier]\label{thm_bfi2} Let $a\in \mathbb{Z}\setminus \{0\}$ and $C\geq 1$ be fixed. Let $\epsilon>0$ be small and $\epsilon'>0$ sufficiently small in terms of $\epsilon$. Let $\xi_1,\xi_2$ satisfy
		\begin{itemize}
			
			\item $\xi_1$ is well-factorable of level $S_1$ and order $C$, $\xi_2$ is supported in $[1,S_2]$, and $|\xi_i(d)|\leq \tau(d)^C$ for $i\in \{1,2\}$;
			
			\item One of the following holds:
			
			(i) $S_2\leq S_1$, $S_1S_2\leq N^{4/7-\epsilon}$.
			
			(ii) $\xi_2(d)=\Lambda(d)$,  $S_1S_2\leq N^{11/20-\epsilon}$, $S_2\leq N^{1/3-\epsilon}$. 
		\end{itemize}
		Then for $1\leq P\leq N^{\epsilon'}$ and uniformly for primitive characters $\chi$ of modulus $\leq P$, we have 
		\begin{align*}
		\sum_{d_1\geq 1}\sum_{\substack{d_2\geq 1\\(d_1d_2,a)=1}}\xi_1(d_1)\xi_2(d_2)\sum_{\substack{n\leq N\\}}\Lambda(n)\chi(n)\mathfrak{u}_P(\overline{a}n;d_1d_2)\ll N(\log N)^{O_C(1)}P^{-1/200} .
		\end{align*}
	\end{prop}

\begin{proof}[Proof of Key Proposition~\ref{prop1b} assuming Proposition~\ref{thm_bfi2}]
By Lemma~\ref{lem_6.2}, Lemma~\ref{le_character}, and Lemma~\ref{le_matomaki} (which is used to bound the Fourier transform of $g$), it suffices to show that
\begin{align*}
 \max_{\substack{\chi\pmod q\\q\leq P^{c_0}}}\max_{1\leq \ell\leq P^{c_0}}\max_{y\leq N}\left|\sum_{n\leq y}(f(\ell n)-f_{\textnormal{pre-sieve}}(\ell n))\chi(n)\right|\ll NP^{-3c_0/2-2c_1}.        
\end{align*}

Now, arguing as in Subsection~\ref{subsec:major}, it suffices to show, uniformly for characters $\chi$ of modulus $\leq P$ and $N/P<y\leq N$, that
	\begin{align}\label{eq_prop2final_2}
	\left|\sum_{d}\lambda_1\star\lambda_2(d)\sum_{\substack{n\leq y\\}}\Lambda(n)\chi(n)\mathfrak{u}_P(\overline{-a}n;d)\right| \ll N P^{-3c_0/2-2c_1},
	\end{align}
	where $\lambda_1, \lambda_2$ are the pre-sieve and main sieve weights, respectively. The Fouvry--Grupp main sieve $\lambda_2$ is of the shape required for an application of Proposition~\ref{thm_bfi2}, after replacing the prime indicator function by the von Mangoldt function, using summation by parts. The pre-sieve can be absorbed into $\xi_1$. Furthermore, after introducing an admissible error, we can assume that $\chi$ is primitive.
\end{proof}
	
	Our definition of a Fouvry--Grupp sieve (see Definition~\ref{Def_FGsieve}) and the related level of distribution estimate of Proposition~\ref{thm_bfi2} are motivated by the result of Fouvry--Grupp~\cite{fouvry-grupp} that shows that the primes obey suitable level of distribution estimates to be combined with such sieves. As in their~\cite[Lemma 4]{fouvry-grupp} we deduce Proposition~\ref{thm_bfi2} from a general estimate on the distribution of bilinear forms in arithmetic progressions. To be more in line with usual notation, we change the notation for the remainder of this section and denote by $X$ the summation range that is called $N$ (or the related $y$) in all other sections of the paper. We define the general bilinear sum we are considering as
	\begin{align}\label{eq_Ddef}
	\mathcal{D}\coloneqq  \sum_{\substack{ q\sim Q}}\sum_{\substack{r\sim R\\(qr,a)=1}}\gamma_q\delta_r\sum_{\substack{m\sim M\\n\sim N}}\alpha_m\beta_n \mathfrak{u}_P(mn\overline{a};qr).
	\end{align}
	
	Throughout this section we often use the following assumptions and notation.
	\begin{convention}\label{def_BFIrough_assumptions} Let $\epsilon>0$ be sufficiently small in absolute terms and $\epsilon'>0$ be sufficiently small in terms of $\epsilon$ (Given $\epsilon$, the largest admissible choice of $\epsilon'$ is not the same in every lemma). Fix $a\in \mathbb{Z}\setminus \{0\}$ and $C\geq 1$. Let $X\asymp MN$ with 
		\begin{align}\label{eq_BFI_Rough_1}
		M,N\geq X^\epsilon,
		\end{align}
		and let $Q,R$ be given with
		\begin{align}
		QR\leq X^{1-\epsilon}
		\end{align}
		Let $\alpha_m, \beta_n, \gamma_q, \delta_r$ be coefficient sequences of order $C$, supported respectively on $m\sim M, n\sim N, q\sim Q, r\sim R$. Let $f_0$ be a nonnegative smooth function supported on $[1/2,5/2]$ which is identically equal to $1$ on $[1,2]$.
	\end{convention}
	
	With this we can state the following result about the distribution of bilinear forms in arithmetic progression to large moduli, considerably extending the permissible level over Lemma~\ref{le_I,IIsimple}.

	\begin{prop}[Distribution of bilinear sums]\label{theorem_BFI_rough} Assume Convention~\ref{def_BFIrough_assumptions} and let 
	\begin{align*}
	 \mathcal{D}= \sum_{\substack{q\sim Q}}\sum_{\substack{r\sim R\\(qr,a)=1}}\gamma_q\delta_r\sum_{\substack{m\sim M\\n\sim N}}\alpha_m\beta_n \mathfrak{u}_P(mn\overline{a};qr)   
	\end{align*}
	as in~\eqref{eq_Ddef}. Assume further that one of the following holds:
		\begin{enumerate}
			\myitem{(T.1)} $QR\leq X^{1/2-\epsilon}$. \label{(T.1)}
		   \myitem{(T.2)} $\gamma=\gamma^{'}\star \gamma^{''}$, with $\gamma^{'},\gamma^{''}$ supported respectively on $[Q_1,2Q_1]$ and $[Q_2,2Q_2]$ (so $Q\asymp Q_1 Q_2$) with\\ $X^{\epsilon}R\leq N\leq X^{-\epsilon}\min\{X^{1/2}Q_1^{-1/2}Q_2^{-1},X^2Q_1^{-5}Q_2^{-2}R^{-1},XQ_1^{-2}Q_2^{-3/2}R^{-1/2}\}.$\label{(T.2)}
			\myitem{(T.3)} $X^{\epsilon}R\leq N\leq X^{-\epsilon}\min\{X^{1/2}Q^{-1}R^{1/2},X^{2/5}Q^{-2/5},X^{1/2}Q^{-3/4}\}$.\label{(T.3)}
			\myitem{(T.4)} $\alpha_m=\chi(m)1_{m\in I}$ for some primitive character $\chi$ of modulus $\leq P$ and some interval $I\subset [M,2M]$ and $X^{1-\epsilon}\geq M\geq X^{\epsilon}\max\{Q,X^{-1}QR^4,Q^{1/2}R,X^{-2}Q^3R^4\}$.\label{(T.4)}
			\myitem{(T.5)} $\gamma_q=1_{q\in I}$ for some interval $I\subset [Q,2Q]$ and $X^{\epsilon}R \leq  N \leq  X^{1/3-\epsilon}R^{-1/3}$.\label{(T.5)}    
		\end{enumerate}

		Then for any $P\leq X^{\epsilon'}$, we have 
		\begin{align*}
		\mathcal{D}\ll X(\log X)^{O_C(1)}P^{-1/7}
		\end{align*}
	\end{prop}
The case \ref{(T.1)} follows immediately from Lemma~\ref{le_BV_rough} and the Cauchy--Schwarz inequality. Roughly speaking, \ref{(T.2)} is a consequence of Lemma 8.1 of Maynard~\cite{MaynardII}, \ref{(T.3)} is related to Theorem 2 of Bombieri--Friedlander--Iwaniec~\cite{bfi}, and \ref{(T.5)} is a combination of~\cite[Theorem 6]{bfi} with Drappeau's estimate for sums of Kloosterman sums in arithmetic progressions (\cite[Theorem 2.1]{drappeau}).
	\begin{remark}
		Throughout this section we optimise neither the dependence of $\epsilon'$ on $\epsilon$ nor the exponent of $P$. With slightly more effort, along the lines of~\cite{drappeau} it should be possible to improve the estimate in Proposition~\ref{theorem_BFI_rough} to $X(\log X)^{O_C(1)}P^{-1}$.
	\end{remark}
	\begin{remark}\label{re_bfi1}
		We are missing the cases (S.2) and (S.5) from the set of seven bilinear estimates in~\cite[Lemma 4]{fouvry-grupp}, but we compensate for this by having a stronger case \ref{(T.2)}. Both of the missing cases, (S.2) and (S.5), originate from work of Fouvry~\cite{fouvry}. Fouvry restricts the coefficients $\beta_n$ to be supported only on $n$  with not too many prime divisors to introduce certain coprimality conditions, see~\cite[Lemme 7]{fouvry}. We were unable to get a power saving with this strategy. While the alternative proof of~\cite[Lemma 4, (S.2)]{fouvry-grupp} given in~\cite[Section 10]{bfi} does not use Fouvry's coprimality approach and looks like it could give a power saving, there seems to be an issue at the bottom of \cite[p. 231]{bfi}. To see this, consider there the case $a=h_1=h_2=q_0=q_1=q_2=1$, $n_1=6$, $n_2=5$, $n_3=3$, $n_4=10$, so $\delta_1=2$, $\delta_2=1$. Then the first expression for the terms in the exponential in the proof of~\cite[Lemma 8]{bfi} is
		\begin{align*}
		\frac{\overline{n_2}}{n_1}-\frac{\overline{n_4}}{n_3}=\frac{-1}{6}-\frac{1}{3}=-\frac{1}{2},
		\end{align*}
		whereas the second expression is
		\begin{align*}
		\Big(\frac{n_3 n_4}{\delta_1\delta_2}-\frac{n_1n_2}{\delta_1\delta_2}\Big)\frac{\overline{n_2n_4/\delta_1\delta_2}}{n_1n_3}=\Big(\frac{30}{2}-\frac{30}{2}\Big)\frac{\overline{n_2n_4/\delta_1\delta_2}}{n_1n_3}=0.
		\end{align*}
		The authors thank James Maynard for making them aware of this issue.
	\end{remark}    
	
	\subsection{Combinatorial dissection}

	We now apply a combinatorial dissection to (both of the) von Mangoldt functions in the statement of Proposition~\ref{thm_bfi2} to reduce it to suitable bilinear sum estimates.
	
	\begin{proof}[Proof of Proposition~\ref{thm_bfi2}, assuming Proposition~\ref{theorem_BFI_rough}]
		The proofs of the cases (i) and (ii) are based on the proofs of~\cite[Theorem 10]{bfi} and~\cite[Theorem 2]{fouvry-grupp}.
		
		We start with Heath-Brown's identity
		\begin{align*}
		  \Lambda(n)=-\sum_{1\leq j\leq J}(-1)^j\binom{J}{j}\sum_{\substack{n=n_1\cdots n_{2j}\\n_i\geq X^{1/J}\implies i\leq j}}(\log n_1)\mu(n_{j+1})\cdots \mu(n_{2j})  
		\end{align*}
		with $J=7$. Observe that the additional character $\chi$ in the statement of Proposition~\ref{thm_bfi2} carries through Heath-Brown's identity via multiplicativity and can be absorbed in the sequences $\alpha$ and $\beta$ unless we are in case \ref{(T.4)}. 
		
		We split each of the variables $n_i$ in Heath-Brown's identity to ranges of the form $N_i< n_i\leq (1+\Delta)N_i$ with $\Delta=P^{-1/200}$. We get the trivial estimate 
		\begin{align*}
		\Delta X (\log X)^{O_C(1)}
		\end{align*}
		for the ranges not covered precisely and the bound
		\begin{align*}
		\Delta^{-14}X(\log X)^{O_C(1)}P^{-1/7}
		\end{align*}
		for the remaining ranges, provided that they fit into one of the cases of Proposition~\ref{theorem_BFI_rough}. Recalling that $\Delta=P^{-1/200}$, we get a good enough error term.
		
		Let us first consider the case (i), that is $S_2\leq S_1$, $S_1S_2\leq N^{4/7-\epsilon}$. In that case the combined weight $\xi_1*\xi_2$ is well-factorable (see~\cite[Lemma 5]{fouvry-grupp}) and so is dealt with, without a power saving error term, in~\cite[Section 10]{bfi}. The proof there is based on an application of~\cite[Theorem 1]{bfi},~\cite[Theorem 2]{bfi}, and~\cite[Theorem 5*]{bfi}. We can replace these three by \ref{(T.2)} (with $Q_2=1$), \ref{(T.3)}, and \ref{(T.4)} respectively to get the desired improved error term.

		To handle the remaining case (ii), we follow the ideas of Fouvry--Grupp~\cite[Section III.]{fouvry-grupp} and in particular apply the same combinatorial decomposition as described in~\cite[Section III.1]{fouvry-grupp}. Write $S_i=X^{\theta_i}$ and recall that we are in the case $\theta_1+\theta_2\leq 11/20-\epsilon, \theta_2\leq 1/3-\epsilon$, and  $\theta_2\geq \theta_1$ (as otherwise we can apply (i)), and the weight $\xi_1$ is well-factorable. We define the intervals
		\begin{align*}
		I_1&\coloneqq [0,2\epsilon] \\
		I_2&\coloneqq (2\epsilon,\theta_1+2\epsilon]\\
		I_3&\coloneqq (\theta_1+2\epsilon,\theta_2+2\epsilon]\\
		I_4&\coloneqq (\theta_2+2\epsilon,3/7]\\
		I_5&\coloneqq (3/7,1].
		\end{align*}
		This differs from the intervals given in~\cite[Section III.3]{fouvry-grupp} only in that we are combining their $J_2$ and $J_3$. Let $N_i=X^{\nu_i}$ in our splitting of $n_i$ into intervals $(N_i,(1+\Delta)N_i]$. Let $\nu$ denote any nonempty subsum of the $\nu_i$. 
		
		If there is a $\nu\in I_4$, we apply \ref{(T.2)} with 
		\begin{align*}
		M&=X^{1-\nu}\\
		N&=X^\nu \\
		Q_1&\leq X^{\theta_1+\theta_2-\nu+2\epsilon}\\
		Q_2&=1\\
		R&\leq X^{\nu-2\epsilon}.
		\end{align*}
	This is essentially the same as in ~\cite[Section III.3]{fouvry-grupp}.
		
		Assume now that there is a $\nu\in I_2$. We start by using the factorability of $\xi_1$ and want to apply \ref{(T.2)} with
		\begin{align*}
		M&=X^{1-\nu}\\
		N&= X^\nu\\
		Q_1&= X^{\theta_1-\nu+2\epsilon}\\
		Q_2&= X^{\theta_2}\\
		R&= X^{\nu-2\epsilon}.
		\end{align*}
		The second and third statements in the minimum clearly make no problem, as 
		\begin{align*}
		\nu-2\epsilon+\epsilon\leq \nu \leq \epsilon+\min\{2-5\theta_1-2\theta_2+4\nu-8\epsilon,1-2\theta_1-(3/2)\theta_2+(3/2) \nu-3\epsilon \},
		\end{align*}
		the worst case here being $\nu=2\epsilon, \theta_1=\theta_2=11/40$.
		Thus, we can apply \ref{(T.2)} as long as
		\begin{align}
		\nu+2\theta_2+\theta_1\leq 1-\epsilon. \label{eq_1}
		\end{align}
		
		If $\eqref{eq_1}$ is not fulfilled, we apply \ref{(T.3)} with 
		\begin{align*}
		M&=X^{1-\nu}\\
		N&=X^{\nu}\\
		Q&\leq X^{11/20-\nu+2\epsilon}\\
		R&\leq X^{\nu-2\epsilon}.
		\end{align*}
		The second and third condition are again easily seen to be fulfilled, the worst case being $\theta_1=\theta_2=11/20$, and for example we have $\nu\leq 2/5(1-11/20+\nu)$ as $33/200\leq 9/50$. In order to check the condition 
		\begin{align}\label{eq_FG3.5}
		N\leq X^{1/2}Q^{-1}R^{1/2},
		\end{align}
		we take advantage of the fact that~\eqref{eq_1} can be assumed to be false. Recalling $\theta_2<1/3$, we can assume
		\begin{align*}
		\theta_1+\theta_2+\nu > 2/3-\epsilon.
		\end{align*}
		As $QR=X^{\theta_1+\theta_2}$ this means that
		\begin{align*}
		X^{1/2}Q^{-1}R^{1/2}\geq X^{5/6-(1/2)\delta}Q^{-3/2}N^{-1/2}.
		\end{align*}
		So~\eqref{eq_FG3.5} follows from
		\begin{align*}
		N^{3/2}Q^{3/2}\leq X^{5/6-(1/2)\epsilon}
		\end{align*}
		which holds for small enough $\epsilon$ as $33/40< 5/6$. 
		
		In the next step, Fouvry and Grupp decompose $\xi_2$ with the help of Heath-Brown's identity. As the arguments can be applied mostly unchanged, we are brief. Similarly as at the end of~\cite[Section III.5]{fouvry-grupp} we can reduce the critical range $I_3$ to an interval $I_3^*=[\theta_1^*+2\epsilon,\tau+2\epsilon]$. Indeed, the extensions of $I_2$ can be easily checked with our modified argument. The enlargement of $I_5$ and the following application of~\cite[(S.6), (S.4), (S.7)]{fouvry-grupp} that completes the treatment of $I_1, I_3, I_5$ can be done exactly as there by replacing the three cases with our \ref{(T.4)}, \ref{(T.3)}, \ref{(T.5)}. 
			\end{proof}

	\subsection{Initial reduction and special case}
	In this subsection we do an initial reduction of Proposition~\ref{theorem_BFI_rough} and prove the case \ref{(T.4)}. 
	
	\begin{lemma}\label{lem_initred}
		It suffices to show Proposition~\ref{theorem_BFI_rough} with the following technical modifications.
		\begin{enumerate}
			\item In the cases \ref{(T.2)}, \ref{(T.3)}, \ref{(T.4)}, and \ref{(T.5)} one may assume that $QR>X^{1/2-\epsilon}$.
			\item In the case \ref{(T.5)} one may assume that $Q^2R\leq X$, and in the cases \ref{(T.2)}, \ref{(T.3)}, and \ref{(T.5)} one may assume that $QN^{3/2}\leq X^{1-\epsilon}$ and that $Q^2 RN\leq X^{2-\epsilon}$.
			\item In the cases \ref{(T.2)}, \ref{(T.3)}, and \ref{(T.5)} one may assume that $\beta_n$ is supported on squarefree integers only, if the estimate \begin{align}\label{eq_reduction_strongerbound}
			\mathcal{D}\ll \|\alpha\|_2 \sqrt{XN} (\log X)^{O_C(1)}P^{-1/6}
			\end{align}
			is obtained for that case.
			\item In the case \ref{(T.4)} one may replace $1_{m\in I}$, $I=[M_1,M_2]\subset [M,2M]$ by a smooth indicator, i.e. take $\alpha_m=\chi(m)f_M(m/M)$ for some smooth function $f_M$ supported on $[(M_1/M)(1-M^{-\epsilon}),(M_2/M)(1+M^{-\epsilon})],$ equal to $1$ on $[M_1/M,M_2/M]$ with $\|f_M^{(j)}\|_\infty\ll_j M^{\epsilon j}$.
			\item In the case \ref{(T.5)} one may replace $1_{q\in I}$ with $I=[Q_1,Q_2]\subset [Q,2Q]$ by a smooth indicator, i.e. take $\gamma_q=f_Q(q/Q)$ for some smooth function $f_Q$ supported on $[(Q_1/Q)(1-Q^{-\epsilon}),(Q_2/Q)(1+Q^{-\epsilon})],$ equal to $1$ on $[Q_1/Q,Q_2/Q]$ with $\|f_Q^{(j)}\|_\infty\ll_j Q^{\epsilon j}$.
			
		\end{enumerate}
	\end{lemma}
	\begin{proof}
		Statement (1) is clear, as else we get the result by \ref{(T.1)}.
		
		We now show (2). To see that we may assume in the case \ref{(T.5)} that
		\begin{align*}
		Q^2R\leq X,    
		\end{align*}
		note that we are counting
		\begin{align*}
		mn=a+qrs,    
		\end{align*}
		where $m\sim M$, $n\sim N$, $q\sim Q$, $r\sim R$ and $X\sim QRs$. As both $q$ and $s$ are unweighted variables, they play the same role, except that $s$ runs through an interval that depends on $m,n,r$. By splitting all variables into intervals of multiplicative length $X^{-\epsilon'}$, we get rid of the dependence of these intervals on $m,n,r'$; for this to work, we need $S\coloneqq X/(QR)\gg X^{\epsilon}$, which we do have by assumption. Hence, either $Q^2R\leq X$ or $S^2R\leq X$, and we may assume that the former holds.

		The bound $QN^{3/2}\leq X^{1-\epsilon} $ follows in  the case \ref{(T.2)} from $N\leq X^{1/2-\epsilon}Q_1^{-1/2}Q_2^{-1}$, in the case \ref{(T.3)} from $N\leq X^{2/5-\epsilon}Q^{-2/5}$, and in the case \ref{(T.5)} from the assumption $Q^2R\leq X$. The bound $Q^2 RN\leq X^{2-\epsilon}$ follows from $QN^{3/2}\leq X^{1-\epsilon}$ and the assumption $QR\leq X^{1-\epsilon}$. 
		
		To reduce to squarefree $n$ and obtain statement (3), we follow a routine strategy (see for example~\cite[Section 5.2]{drappeau} for a similar but more elaborate approach) and write $n=kn'$ with $n'$ squarefree. Let $\mathcal{K}$ denote the set of square-full integers. Then
		\begin{align*}
		\mathcal{D}=\sum_{\substack{k \\ k\in \mathcal{K}}}\sum_{\substack{ q\sim Q}}\sum_{\substack{r\sim R\\(qr,a)=1}}\gamma_q\delta_r\sum_{\substack{m\sim M\\n'\sim N/k}}\alpha_m \mu^2(n')\beta_{kn'} \mathfrak{u}_P(mn'k\overline{a};qr).
		\end{align*}
		By the trivial bound~\eqref{eq_trivbound} we can estimate the contribution of $k>K$ by
		\begin{align*}
		\ll P X (\log X)^{O_C(1)} K^{-1/2}.
		\end{align*}
		
		Write now
		\begin{align*}
		\alpha'_m&= 1_{k|m} \alpha_{m/k}\\
		\beta'_n&=k^{-\epsilon}\mu^2(n')\beta_{kn'}
		\end{align*}
		so that
		\begin{align*}
		&\sum_{\substack{ q\sim Q}}\sum_{\substack{r\sim R\\(qr,a)=1}}\gamma_q\delta_r\sum_{\substack{m\sim M\\n'\sim N/k}}\alpha_m \mu^2(n')\beta_{kn'} \mathfrak{u}_P(mnk\overline{a};qr)\\
		&=k^{\epsilon}\sum_{\substack{ q\sim Q}}\sum_{\substack{r\sim R\\(qr,a)=1}}\gamma_q\delta_r\sum_{\substack{m'\sim kM\\n'\sim N/k}}\alpha'_{m'} \beta'_{n'} \mathfrak{u}_P(m'n'\overline{a};qr).
		\end{align*}
		If $K<X^{\epsilon/2}$ we can apply~\eqref{eq_reduction_strongerbound} with the assumed improved estimate and $\epsilon/2$ taking the role of $\epsilon$ to get
		\begin{align*}
		\sum_{\substack{ q\sim Q}}\sum_{\substack{r\sim R\\(qr,a)=1}}\gamma_q\delta_r\sum_{\substack{m'\sim kM\\n'\sim N/k}}\alpha'_{m'} \beta'_{n'} \mathfrak{u}_P(m'n'\overline{a};qr)&\ll \|\alpha'\|_2 \sqrt{XN} (\log X)^{O_C(1)}P^{-1/6}\\
		&\ll X (\log X)^{O_C(1)}P^{-1/6} k^{-1/2+\epsilon} .
		\end{align*}
	   Consequently we can bound the contribution of $k\leq K$ by
	   \begin{align*}
	   X (\log X)^{O_C(1)}P^{-1/6} \sum_{\substack{k\leq K \\ k \in \mathcal{K}}} k^{-1/2+2\epsilon}\ll X (\log X)^{O_C(1)}P^{-1/6} K^{2\epsilon}.
	   \end{align*}
		We obtain the desired result after choosing $K=P^{3}$.
		
 	    In statement (4), the error induced by introducing a smooth cutoff can be estimated with the help of the trivial bound~\eqref{eq_trivbound} by
 	    \begin{align*}
 	 	XP (\log X)^{O_C(1)} M^{-\epsilon'},
 	    \end{align*}
 	    and as $P\leq X^{\epsilon'}$, $M>Q^{1/2}R\geq X^{1/4-\epsilon'/2}$, this is sufficient.
	    
	    Similarly as in (4), in statement (5) the error from replacing the cutoff by a smooth function can be estimated by
	    \begin{align*}
	    XP (\log X)^{O_C(1)} Q^{-\epsilon},
	    \end{align*}
	    which is sufficient, now for~\eqref{eq_reduction_strongerbound}, as we have $R\leq X^{1/3}R^{-1/3}$ and $QR>X^{1/2-\epsilon'}$. 
	\end{proof}
	
	We also need the following truncated version of the  Poisson summation formula on several occasions. 
\begin{lemma}[Truncated Poisson summation]\label{lem_Poiss}
	Let $\epsilon$ be small, $f$ be a smooth function supported on $[-10,10]$ with $\|f^{(j)}\|_\infty \ll_j X^{\epsilon j}$, and let $M,q\leq X$. Then we have
	\begin{align*}
\sum_{m\equiv a(q)}f\bigl(\frac{m}{M}\bigr)=\frac{M}{q}\sum_{|h|\leq H}\widehat{f}\bigl(\frac{hM}{q}\bigr)e\bigl(\frac{-ah}{q}\bigr)+O_{\epsilon}(X^{-100})
	\end{align*}
	for any choice $H>X^{2\epsilon}q/M$.
\end{lemma}
\begin{proof}
	By Poisson summation,
	\begin{align*}
\sum_{m\equiv a(q)}f\bigl(\frac{m}{M}\bigr)=\frac{M}{q}\sum_{h}\widehat{f}\bigl(\frac{hM}{q}\bigr)e\bigl(\frac{-ah}{q}\bigr).
	\end{align*}
	The bound $\|f^{(j)}\|_\infty \ll_j X^{\epsilon j}$ together with integration by parts gives $\widehat{f}(t)\ll_j X^{\epsilon j}t^{-j}$. So we can bound the contribution of $|h|>X^{2\epsilon}q/M$ by
	\begin{align*}
	\ll_j \frac{M}{q} X^{\epsilon j} \sum_{h>X^{2\epsilon}q/M} \Big(\frac{hM}{q}\Big)^{-j}\ll_j \frac{M}{q} X^{\epsilon j+(-j+1)2\epsilon}, 
	\end{align*}
	which is sufficiently small, after choosing $j$ large enough in terms of $\epsilon$.
\end{proof}

Given a character $\chi$ to the modulus $q$ Gauß sums of the form
\begin{align*}
	c_\chi(a):=\sum_{b(q)^*}\chi(b)e_q(ba)
\end{align*}
will appear on several occasions. They are multiplicative in the following sense. 	If $q=q_1q_2$ with $(q_1,q_2)=1$ then we have
\begin{align*}
	c_\chi(a)=c_{\chi^{(q_1)}}(\overline{q_2}a)c_{\chi^{(q_2)}}(\overline{q_1}a).
\end{align*}
It thus suffices to study them from prime powers and we have the following evaluation.
\begin{lemma}\label{lem_ccalc}
	Let $\chi$ be a character to the modulus $p^\alpha$ and let $p^{\alpha_0}$ denote the modulus of the primitive character, $\chi^*$, inducing $\chi$. Let $p^{\alpha_m}=(p^\alpha,m)$. If $\alpha_0>\alpha-\alpha_m$ then $c_\chi(m)=0$. If $\alpha_0\leq \alpha-\alpha_m$ then
	\begin{align*}
		c_{\chi}(m)=\overline{\chi}^*(m/p^{\alpha_m})\chi^*(p^{\alpha-\alpha_m-\alpha_0})\mu(p^{\alpha-\alpha_m-\alpha_0})\frac{\varphi(p^\alpha)}{\varphi(p^{\alpha-\alpha_m})}\tau(\chi^*)
	\end{align*}
\end{lemma}
\begin{proof}
	This is~\cite[Lemma 5.4]{mv} for prime power moduli.
\end{proof}
	
	We now prove case \ref{(T.4)} of Proposition~\ref{theorem_BFI_rough} (based on the work of Bombieri--Friedlander--Iwaniec~\cite{bfi}), which is different from the other cases in that it does not rely on the dispersion method.
	\begin{lemma}[Variant of~\cite{bfi} Section 12]\label{lem_T4}
		Proposition~\ref{theorem_BFI_rough} is true in the case \ref{(T.4)}.
	\end{lemma}
	\begin{proof}
		Recall that $\mathcal{D}$ is given by~\eqref{eq_Ddef} and that by Lemma~\ref{lem_initred} we can assume to be in the case $\alpha_m=f_M(m/M)\chi(m)$, where $\chi\pmod s$ is a primitive character to some modulus $s \leq P$.
		
		We restrict $r$ into a fixed residue class $l$ modulo $s$ and further restrict $q$ such that $(q,s)=s_q$. Taking the values of $l$ and $s_q$ that give the largest contribution, it suffices to show
		\begin{align*}
		D\ll \frac{MN}{X^{\epsilon'} \tau(s)s},
		\end{align*}
		where
		\begin{align*}
		D\coloneqq \sum_{\substack{q,r \\ \substack{(q,s)=s_q,r\equiv l(s) \\ (qr,a)=1}  }}\gamma_q \delta_r \sum_{m,n}f_M(m/M)\chi(m) \beta_n \mathfrak{u}_P(mn\overline{a};qr).
		\end{align*}
		Set $W=[qr,s]$ and $H=X^{2\epsilon'}W/M$.  By Poisson summation (Lemma~\ref{lem_Poiss}), we have
		\begin{align}\label{eq_lemT4_1}
		\sum_{\substack{m\\ m\equiv a \overline{n} (qr)}}f_M(m/M)\chi(m)=\frac{M}{W}\sum_{\substack{c(W)\\ c\equiv a\overline{n}(qr)}}\chi(c)\sum_{|h|\leq H}\widehat{f_M}\bigl(\frac{hM}{W}\bigr)e\bigl(\frac{-ch}{W}\bigr)+O_{\epsilon'}(X^{-100})
		\end{align}
		and
		\begin{align}\nonumber
		&\frac{1}{\varphi(qr)}\sum_{\substack{\psi(qr)\\ \text{cond}(\psi)\leq P}}\psi(\overline{a}n)\sum_m f_M(m/M) \chi\psi(m)\\
		&=\frac{M}{\varphi(qr)W}\sum_{\substack{\psi(qr)\\ \text{cond}(\psi)\leq P}}\psi(\overline{a}n)\sum_{c(W)}\chi\psi(c)\sum_{|h|\leq H}\widehat{f_M}\bigl(\frac{hM}{W}\bigr)e\bigl(\frac{-ch}{W}\bigr)+O_{\epsilon'}(X^{-100}).\label{eq_lemT4_2}
		\end{align} 
		The error terms are obviously admissible. 
		
		We first consider the case $h=0$. We have 
		\begin{align*}
		\sum_{\substack{c(W)\\ c\equiv a\overline{n}(qr)}}\chi(c)=\chi(a\overline{n})1_{s|qr}
		\end{align*}
		and
		\begin{align*}
		\sum_{c(W)}\chi\psi(c)=\varphi(W)1_{\text{cond}(\chi\psi)=1}.
		\end{align*}
		As $s\leq P$, we have that  $s|qr$ is equivalent to the existence of a (necessarily unique) character $\psi \pmod {qr}$, $\text{cond}(\psi)\leq P$ such that $\text{cond}(\chi\psi)=1$. For this character it holds that $\psi(\overline{a}n)=\chi(a\overline{n})$. For $s|qr$ we have furthermore
		\begin{align*}
		\frac{1}{W}=\frac{\varphi(W)}{W\varphi(qr)}
		\end{align*}
		and so the $h=0$ contributions to the right-hand side of~\eqref{eq_lemT4_1} and~\eqref{eq_lemT4_2} are equal.
		
		We now consider the $h\neq 0$ terms on the right-hand side of~\eqref{eq_lemT4_2}. The conductor of $\chi\psi$ is at most $Ps$ and so by Lemma \ref{lem_ccalc} and the classical bound for Gau\ss{} sums we have
		\begin{align*}
		\sum_{c(W)}\chi\psi(c)e\bigl(\frac{-ch}{W}\bigr)\ll \sqrt{Ps}(W,h).
		\end{align*}
		Therefore, the $h\neq 0$ terms in~\eqref{eq_lemT4_2} contribute to $D$ at most
		\begin{align*}
		M\sqrt{Ps}\sum_{\substack{q,r \\ \substack{(q,s)=s_q,r\equiv l(s) \\ (qr,a)=1}  }}\frac{|\gamma_q| |\delta_r| H \tau(W)}{W \varphi(qr)} \sum_{n}|\beta_n|\ll X^{3\epsilon'}\sqrt{Ps}N,
		\end{align*}
		which is admissible by~\eqref{eq_BFI_Rough_1} and the assumptions $P,s\leq X^{\epsilon'}$.
		
		To bound the contribution of the $h\neq 0$ terms on the right-hand side of~\eqref{eq_lemT4_1} to $D$ we write $W=W_1W_2$, and analogously $\chi=\chi_{W_1}\chi_{W_2}$, with $W_1=(W,(qr)^\infty)$ and $s_r=(s,r)$ (which is fixed by $r\equiv l(s)$), and we observe that
		\begin{align*}
		\sum_{\substack{c(W)\\ c\equiv a\overline{n}(qr)}}\chi(c)e\bigl(\frac{-ch}{W}\bigr)&=\sum_{\substack{c_1(W_1)\\ c_1\equiv a\overline{n}(qr)}}\chi_{W_1}(c_1)e\bigl(\frac{-c_1h\overline{W_2}}{W_1}\bigr)\sum_{\substack{c_2(W_2)\\ }}\chi_{W_2}(c)e\bigl(\frac{-c_2h\overline{W_1}}{W_2}\bigr) \\
		&=\sum_{d(s,(s_qs_r)^\infty)}\chi_{W_1}(a\overline{n}+dqr)e\bigl(\frac{-(a\overline{n}+dqr)h\overline{W_2}}{qr (s,(s_qs_r)^\infty) }\bigr)\sum_{\substack{c_2(W_2)\\ }}\chi_{W_2}(c_2)e\bigl(\frac{-c_2h\overline{W_1}}{W_2}\bigr)\\
		&=e\bigl(-ah\frac{\overline{nW_2}}{qr (s,(s_q s_r)^\infty) }\bigr)\sum_{d(s,(s_q s_r)^\infty)}\chi_{W_1}(a\overline{n}+dqr)e\bigl(\frac{-dh\overline{W_2}}{ (s,(s_qs_r)^\infty) }\bigr)\\
		&\times \sum_{\substack{c_2(W_2)\\ }}\chi_{W_2}(c_2)e\bigl(\frac{-c_2h\overline{W_1}}{W_2}\bigr).
		\end{align*}
		We recall that $r\equiv l(s)$ and that $W_2 (s,(s_q s_r)^\infty)\leq s$. Therefore, it suffices to show
		\begin{align*}
		D'\ll \frac{N}{X^{\epsilon'}\tau(s)s^2},
		\end{align*}
		where
		\begin{align*}
		D'\coloneqq \sum_{\substack{q,r \\ \substack{(q,s)=s_q,r\equiv l(s) \\ (qr,a)=1}  }}\frac{\gamma_q \delta_r}{W} \sum_n \beta_n \sum_{1\leq |h|\leq H} \widehat{f_M}\bigl(\frac{Mh}{W}\bigr)e\bigl(-ah\frac{\overline{nW_2}}{qr (s,(s_q s_r)^\infty) }\bigr)e(\frac{-d'h}{s})
		\end{align*}
		for any fixed $d' (s)$. By definition $W=\frac{qrs }{[s_q,s_r]}$, and we can separate the variables by 
		\begin{align*}
		\widehat{f_M}\bigl(\frac{Mh}{W}\bigr)=\frac{q s}{[s_q, s_r] M}\int_{-\infty}^\infty f_M( \frac{\xi q s}{M[s_q, s_r]}) e\bigl(\frac{\xi h}{r}\bigr)d\xi.
		\end{align*}
		Thus,
		\begin{align*}
		D'&\ll X^{\epsilon'} Q^{-1}\sum_{\substack{q\sim Q\\ n\sim N}}\Bigl|\sum_{1\leq |h|\leq H}\sum_{\substack{r \\ (r,a)=1\\ r\equiv l (s) }} \frac{\delta_r}{r}e\bigl(\frac{\xi h}{r}\bigr)e(\frac{-d'h}{s})e\bigl(-ah\frac{\overline{nW_2}}{qr(s,(s_qs_r)^{\infty}) }\bigr)\Bigr|\\
		&\ll X^{\epsilon'} Q^{-1}\sum_{\substack{Q\leq q'\leq QX^{\epsilon'}\\ N\leq n'\leq X^{\epsilon'}N}}\Bigl|\sum_{1\leq |h|\leq H}\sum_{\substack{r \\ (r,a)=1\\ r\equiv l (s) }} \frac{\delta_r}{r}e\bigl(\frac{\xi h}{r}\bigr)e(\frac{-d'h}{s})e\bigl(-ah\frac{\overline{n'}}{q'r }\bigr)\Bigr|\\
		&\ll X^{\epsilon'} Q^{-1}R^{-1}\sum_{\substack{Q\leq q'\leq QX^{\epsilon'}\\ N\leq n'\leq X^{\epsilon'}N}}\Bigl|\sum_{1\leq |h|\leq H}\sum_{\substack{r \\ (r,a)=1 }} \delta(h,r)e\bigl(-ah\frac{\overline{n'}}{q'r }\bigr)\Bigr|.
		\end{align*}
		for some $|\delta(h,r)|\leq 1$. This expression is  as in~\cite[after eq. (12.2)]{bfi}, so the proof there now goes through (and gives a power-saving). 
	\end{proof}

	\subsection{Dispersion method and Kloosterman sums}
	The  cases \ref{(T.2)}, \ref{(T.3)}, \ref{(T.5)} of Theorem~\ref{theorem_BFI_rough} are proved with the dispersion method. We split its application into two lemmas.
	
	\begin{lemma}[Dispersion of bilinear sums]\label{lem_congtodisp}
		Assume Convention~\ref{def_BFIrough_assumptions} and let $\mathcal{D}$ be as in~\eqref{eq_Ddef}. Assume further that
		\begin{align}
		P&\leq X^{\epsilon'} \label{eq_congtodisp_cond01} \\
		R &\leq N \label{eq_congtodisp_cond02}X^{-\epsilon}\\
		QN^{3/2}&\leq X^{1-\epsilon} \label{eq_congtodisp_cond03}\\
		QR&\leq X^{1-\epsilon}. \label{eq_congtodisp_cond04}
		\end{align}
		Then we have
		\begin{align*}
		\mathcal{D}\ll \|\alpha\|_2 \sqrt{XN} P^{-1/6}(\log X)^{O_C(1)}+(MR)^{1/2}(\log X)^{O_C(1)} \Bigl(\sum_{\nu\leq P^{1/2}}\sum_{\substack{q_0\leq P^{1/2}\\ (q_0,a\nu)=1}}|\mathscr{E}(q_0,\nu)|\Bigr)^{1/2},
		\end{align*}
		where
		\begin{align}\label{eq_mscrEdef}
		\mathscr{E}(q_0,\nu)&\coloneqq \sum_{\substack{r\sim R\\(r,a\nu)=1}}\sum_{\substack{q_1, q_2 \\ \substack{(a\nu,q_1q_2)=1\\ (q_1,q_2)=1}}}\gamma_{q_0q_1}\overline{\gamma_{q_0q_2}}\sum_{\substack{n_1,n_2 \\ \substack{(n_1,n_2)=1 \\ n_1\equiv n_2 (q_0r)\\(n_i,q_0q_ir)=1}}}\beta_{\nu n_1}\overline{\beta_{\nu n_2}}\left(\sum_{\substack{m\\ m \nu n_i\equiv a(q_0q_i r)}}f_0\bigl(\frac{m}{M}\bigr)-\frac{M\widehat{f_0}(0)}{q_0q_1q_2r} \right).
		\end{align}
		
	\end{lemma}
	
	\begin{proof}
		At the heart of the dispersion method lies the Cauchy--Schwarz inequality. Its application gives us
		\begin{align*}
		\mathcal{D}^2&\leq \|\delta\|^2 \|\alpha\|^2 \mathscr{S}(M,N,Q,R)\\
		&\leq R \|\alpha\|^2 \ (\log X)^{O_C(1)}\mathscr{S}(M,N,Q,R),
		\end{align*}
		where 
		\begin{align*}\begin{split}
		\mathscr{S}(M,N,Q,R)&\coloneqq \sum_{\substack{r\sim R\\ (r,a)=1}}\sum_{\substack{m\sim M\\}}\Bigg|\sum_{\substack{q\sim Q\\ \substack{n\sim N \\(q,a)=1}}}\gamma_{q} \beta_n \mathfrak{u}_P(mn\bar{a};qr)\Bigg|^2
		\end{split}.
		\end{align*}
		
		Define $\mathscr{S}^*=\mathscr{S}^*(M,N,Q,R)$ similarly to $\mathscr{S}(M,N,Q,R)$ but with an additional smooth weight $f_0(\frac{m}{M})$ on the $m$ variable. Then $\mathscr{S}(M,N,Q,R)\leq \mathscr{S}^{*}(M,N,Q,R)$.
		We recall the definition of $\mathfrak{u}_P$ in~\eqref{eq_uPdef} and expand the square in the definition of $\mathscr{S}^*$ to write
		\begin{align*}
		\mathscr{S}^*=\mathscr{S}_1-2\textnormal{Re}(\mathscr{S}_2)+\mathscr{S}_3,
		\end{align*}
		where 
		\begin{align*}
		\mathscr{S}_1&\coloneqq \sum_{\substack{r\sim R, \,m \\ \substack{(r,a)=1\\ }}} f_0\bigl(\frac{m}{M}\bigr) \sum_{\substack{q_1, q_2 \\ \substack{ (a,q_1q_2)=1 \\ }}}\gamma_{q_1}\overline{\gamma_{q_2}} \sum_{\substack{n_1,n_2 \\ mn_i\equiv a(q_i r)}}\beta_{n_1}\overline{\beta_{n_2}}\\
		\mathscr{S}_2&\coloneqq \sum_{\substack{r\sim R, \,m \\ \substack{(r,a)=1\\ }}} f_0\bigl(\frac{m}{M}\bigr) \sum_{\substack{q_1, q_2 \\ \substack{ (a,q_1q_2)=1 \\ }}}\frac{\gamma_{q_1}\overline{\gamma_{q_2}}}{\varphi(q_2r)} \sum_{\substack{\psi(q_2r)\\ \text{cond}(\psi)\leq P}}\psi(m \overline{a})\sum_{\substack{n_1,n_2 \\ mn_1\equiv a(q_1 r)\\ \substack{\\ }}}\beta_{n_1}\overline{\beta_{n_2}} \psi(n_2)\\
		\mathscr{S}_3&\coloneqq \sum_{\substack{r\sim R, \,m \\ \substack{(r,a)=1\\ }}} f_0\bigl(\frac{m}{M}\bigr) \sum_{\substack{q_1, q_2 \\ \substack{ (a,q_1q_2)=1 \\ }}}\frac{\gamma_{q_1}\overline{\gamma_{q_2}}}{\varphi(q_1r)\varphi(q_2r)}\sum_{\substack{\psi_i(q_ir)\\ \text{cond}(\psi_i)\leq P}} \psi_1 \overline{\psi_2}(m \overline{a}) \sum_{\substack{n_1,n_2 \\ \substack{ \\ }}}\beta_{n_1}\overline{\beta_{n_2}} \psi_1(n_1)\overline{\psi_2}(n_2).
		\end{align*}
		
		We first estimate $\mathscr{S}_1$. We begin by discarding the terms with $(q_1,q_2)\geq Q_0$ and $(n_1,n_2)\geq N_0$ from $\mathscr{S}_1$ for  $Q_0=N_0=P^{1/2}$. This produces an error term which is $O\left(P^{-1/3}XN R^{-1} (\log X)^{O_C(1)}\right)$. Indeed, the contribution of $(n_1,n_2)=\nu$ is crudely
		\begin{align*}
		&\ll \sum_{r\sim R}\sum_{\substack{M/2\leq m\leq 3M\\(r,am)=1}}\sum_{\substack{q_1,q_2\sim Q\\(am,q_1q_2)=1\\(q_1q_2r,\nu)\mid a}}\tau(q_1)^C\tau(q_2)^C\sum_{\substack{n_1\sim N\\n_1\equiv a\overline{m}\pmod{q_1r}\\\nu\mid n_1}}\,\,\sum_{\substack{n_2\sim N\\n_2\equiv a\overline{m}\pmod{q_2r}\\\nu\mid n_2}}\tau(n_1)^{C}\tau(n_2)^C
		\\&\ll \sum_{r\sim R}\sum_{\substack{M/2\leq m\leq 3M\\(r,am)=1}}\sum_{\substack{n_1\sim N\\n_1\equiv a\overline{m}\pmod{r}\\\nu\mid n_1}}\,\,\sum_{\substack{n_2\sim N\\n_2\equiv a\overline{m}\pmod{r}\\\nu\mid n_2}}(\tau(n_1)\tau(n_2))^{C}(\tau(mn_1-a)\tau(mn_2-a))^{C+1}\\
		&\ll_a \sum_{r\sim R}\sum_{\substack{M/2\leq m\leq 3M\\(r,am)=1}}\frac{N^2}{r^2}(\log X)^{O_C(1)}\tau(\nu)^{4C+2}\nu^{-2}\\
		&\ll \frac{MN^2}{R}(\log X)^{O_C(1)}\nu^{1/100-2},
		\end{align*}
		where for the third line we used the trivial inequality $x_1\cdots x_4\leq \sum_{i\leq 4}x_i^4$ and a standard upper bound for $\sum_{n\leq x, n\equiv \alpha\pmod s}\tau(n)^{B}$ arising from Shiu's bound~\cite[Theorem 1]{shiu}. This summed over $\nu\geq P^{1/2}$ produces $\ll P^{-1/3}XN R^{-1} (\log X)^{O_C(1)}$. The contribution of $(q_1,q_2)\geq P^{1/2}$ is bounded similarly (cf.~\cite[p.219]{bfi}).
		
		Therefore, we obtain
		\begin{align} \label{eq_S_1}
		\mathscr{S}_1=\mathscr{S}_1'+O\left(P^{-1/3} XN R^{-1} (\log X)^{O_C(1)}\right),
		\end{align}
		where
		\begin{align*}
		\mathscr{S}_1'\coloneqq \sum_{\nu\leq P^{1/2}}\sum_{\substack{q_0\leq P^{1/2}\\ (q_0,a\nu)=1}}\mathscr{S}_1(\nu,q_0)
		\end{align*}
		and
		\begin{align*}
		\mathscr{S}_1(\nu,q_0)\coloneqq \sum_{\substack{r\sim R, \, m \\ (r,a\nu)=1}} f_0\bigl(\frac{m}{M}\bigr) \sum_{\substack{q_1, q_2 \\ (a\nu,q_1q_2)=1\\(q_1,q_2)=1}}\gamma_{q_0q_1}\overline{\gamma_{q_0q_2}} \sum_{\substack{n_1,n_2 \\ \substack{(n_1,n_2)=1 \\  m\nu n_i\equiv a(q_i r)}}}\beta_{\nu n_1}\overline{\beta_{\nu n_2}}.
		\end{align*}
		Changing the order of summation, we have 
		\begin{align*}
		\mathscr{S}_1(\nu,q_0)=\sum_{\substack{r\sim R\\ \substack{(r,a\nu)=1}}}\sum_{\substack{q_1, q_2 \\ \substack{(a\nu,q_1q_2)=1\\ (q_1,q_2)=1}}}\gamma_{q_0q_1}\overline{\gamma_{q_0q_2}}\sum_{\substack{n_1,n_2 \\ \substack{(n_1,n_2)=1 \\ n_1\equiv n_2 (r)\\(n_i,q_0q_ir)=1}}}\beta_{\nu n_1}\overline{\beta_{\nu n_2}} \sum_{\substack{m\\ m \nu n_i\equiv a(q_0q_i r)}}f_0\bigl(\frac{m}{M}\bigr).
		\end{align*}
		
		We approximate the sum over $m$ by $M\widehat{f_0}(0)/(q_0q_1q_2r)$ so that  
		\begin{align*}
		\mathscr{S}_1'=\mathscr{X}+\sum_{\nu\leq P^{1/2}}\sum_{\substack{q_0\leq P^{1/2}\\ (q_0,a\nu)=1}}\mathscr{E}(q_0,\nu),
		\end{align*}
		where $\mathscr{E}(q_0,\nu)$ is as in~\eqref{eq_mscrEdef} and
		\begin{align*}
		\mathscr{X}&\coloneqq M\widehat{f_0}(0)\sum_{\nu\leq P^{1/2}}\sum_{\substack{q_0\leq P^{1/2}\\ (q_0,a\nu)=1}} \sum_{\substack{r\sim R\\(r,a\nu)=1}}\frac{1}{r}\sum_{\substack{q_1, q_2 \\ \substack{(a\nu,q_1q_2)=1\\ (q_1,q_2)=1}}}\frac{\gamma_{q_0q_1}\overline{\gamma_{q_0q_2}}}{q_0q_1q_2}\sum_{\substack{n_1,n_2 \\ \substack{(n_1,n_2)=1 \\ n_1\equiv n_2 (q_0r)\\(n_i,q_0q_ir)=1}}}\beta_{\nu n_1}\overline{\beta_{\nu n_2}}.
		\end{align*}
		
		We complete the sum over $\nu\leq P^{1/2}$, again introducing and admissible error. Then the above becomes
		\begin{align*}
		M\widehat{f_0}(0)\sum_{\substack{q_0\leq P^{1/2}\\ (q_0,a)=1}}\sum_{\substack{r\sim R\\(r,a)=1}}\frac{1}{r}\sum_{\substack{q_1, q_2 \\ \substack{(a,q_1q_2)=1\\ (q_1,q_2)=1}}}\frac{\gamma_{q_0q_1}\overline{\gamma_{q_0q_2}}}{q_0q_1q_2\varphi(q_0r)}\sum_{\psi\pmod {q_0r}} \sum_{\substack{n_1,n_2  \\ (n_i,q_i)=1 \\ }}\psi(n_1)\beta_{ n_1}\overline{\psi(n_2)\beta_{n_2}}.
		\end{align*}
		The main term here comes from characters with small conductor, for us those that have conductor $\leq P$. Large conductor characters are handled by the Cauchy--Schwarz inequality and Lemma~\ref{lem_BDHrough}; this gives us
		\begin{align*}
		&\Bigl|\sum_{\substack{r\sim R\\(r,a\nu)=1}}\frac{1}{\varphi(q_0r)}\sum_{\substack{\psi\pmod {q_0r}\\ \text{ cond}(\psi)>P}} \sum_{\substack{n_1,n_2  \\ (n_i,q_i)=1 \\ }}\psi(n_1)\beta_{ n_1}\overline{\psi(n_2)\beta_{n_2}}\Bigr|\\
		&\leq \Bigl(\sum_{\substack{r\leq  Rq_0 \\}}\frac{1}{\varphi(r)}\sum_{\substack{\psi\pmod {r}\\ \text{ cond}(\psi)>P}} \bigl|\sum_{\substack{n  \\ (n,q_1)=1 \\ }}\psi(n)\beta_{n}\bigl|^2\Bigr)^{1/2}\\\
		&\times \Bigl(\sum_{\substack{r\leq  Rq_0 \\}}\frac{1}{\varphi(r)}\sum_{\substack{\psi\pmod {r}\\ \text{ cond}(\psi)>P}} \bigl|\sum_{\substack{n  \\ (n,q_2)=1 \\ }}\psi(n)\beta_{n}\bigl|^2\Bigr)^{1/2}\\
		&\ll \Big(Rq_0+\frac{N}{P}\Big)N (\log X)^{O_C(1)}.
		\end{align*}
		
		Thus we have
		\begin{align*}
		\mathscr{X}&=M\widehat{f_0}(0)\sum_{\substack{q_0\leq P^{1/2}\\ (q_0,a)=1}} \sum_{\substack{r\sim R\\(r,a)=1}}\frac{1}{r}\sum_{\substack{q_1, q_2 \\ \substack{(a,q_1q_2)=1\\ (q_1,q_2)=1}}}\frac{\gamma_{q_0q_1}\overline{\gamma_{q_0q_2}}}{q_0q_1q_2\varphi(q_0r)}\sum_{\substack{\psi\pmod {q_0r}\\ \cond(\psi)\leq P}} \sum_{\substack{n_1,n_2  \\  (n_i,q_i)=1}}\psi(n_1)\beta_{ n_1}\overline{\psi(n_2)\beta_{ n_2}}\\
		&+O\Bigl(\frac{MN (\log X)^{O_C(1)}}{R}\bigl(PR+\frac{N}{P^{1/3}}\bigr) \Bigr).
		\end{align*}
		By~\eqref{eq_congtodisp_cond01} and~\eqref{eq_congtodisp_cond02}, this is acceptable. After completing the sum over $q_0$ again, the main term here equals up to an admissible error
		\begin{align*}
		\mathcal{X}\coloneqq M\widehat{f_0}(0) \sum_{\substack{r\\(r,a)=1}}\frac{1}{r}\sum_{\substack{q_1, q_2 \\ \substack{(a,q_1q_2)=1\\ }}}\frac{\gamma_{q_1}\overline{\gamma_{q_2}}}{[q_1,q_2] \varphi((q_1,q_2)r)}\sum_{\substack{\psi\pmod {(q_1,q_2)r}\\ \cond(\psi)\leq P}}  \sum_{\substack{n_1,n_2 \\(n_1,n_2)=1 \\ }}\psi(n_1)\beta_{ n_1}\overline{\psi(n_2)\beta_{ n_2}}.
		\end{align*}
		To complete the proof of Lemma~\ref{lem_congtodisp}, we show that $\mathscr{S}_2$ and $\mathscr{S}_3$ are also well approximated by $\mathcal{X}$.
		
		We now consider $\mathscr{S}_2$ and follow a strategy that generalises both~\cite[Section 5.3.2]{drappeau} and~\cite[Section 5]{bfi}. Let $\mathfrak{c}$ denote the conductor of $\psi$ in the definition of $\mathscr{S}_2$. In the relevant range of summation, $m$ is coprime to $q_1r$ and we have
		\begin{align*}
		\sum_{\substack{\psi(q_2r)\\ \text{cond}(\psi)\leq P}}\psi(m\overline{a}n_2)&=\sum_{\substack{\mathfrak{c}\leq P \\ \psi(\mathfrak{c})^*}} \psi(m\overline{a}n_2)1_{\substack{(n_2,q_2r)=1\\ \mathfrak{c}|q_2r}}1_{(m,q_2)=1}\\
		&=\sum_{\substack{\mathfrak{c}\leq P \\ \psi(\mathfrak{c})^*}}\psi(m\overline{a}n_2)1_{\substack{(n_2,q_2r)=1\\ \mathfrak{c}|q_2r}} \sum_{\substack{e|(m,q_2), \, (e,q_1r)=1}}\mu(e).
		\end{align*}
		
		Since $e>P^{1/2}$ implies $(m,q_2)>P^{1/2}$, the contribution of $e>P^{1/2}$ to  $\mathscr{S}_2$ gives an admissible error. We denote the remaining terms with $e\leq P^{1/2}$ by $\mathscr{S}_2'$ and sort $r$ into residue classes modulo $\mathfrak{c}$ to get 
		\begin{align*}
		\mathscr{S}_2'&=\sum_{\substack{\mathfrak{c}\leq P\\ \psi(\mathfrak{c})^*}}\psi(\overline{a})\sum_{\substack{q_1,q_2 \\  \substack{(a,q_1q_2)=1\\ }}}\gamma_{ q_1}\gamma_{q_2}\sum_{\substack{e|q_2 , \, (e,q_1)=1\\ e\leq P^{1/2}}}\mu(e)\psi(e)\sum_{\substack{n_1,n_2\\ (n_i,q_i)=1}}\beta_{n_1}\overline{\beta_{n_2}}\psi(n_2) \\
		&\times \sum_{\substack{c(\mathfrak{c})\\ c\equiv 0( \mathfrak{c}/(\mathfrak{c},q_2))}}\mathscr{T}(n_1,n_2,q_1,q_2,\psi,c,e),
		\end{align*}
		where
		\begin{align*}
		\mathscr{T}(n_1,n_2,q_1,q_2,\psi,c,e)\coloneqq \sum_{\substack{r\sim R\\ \substack{(r,aen_1n_2)=1\\ r\equiv c(\mathfrak{c})}}}\frac{1}{\varphi( q_2 r)}\sum_{\substack{emn_1\equiv a (q_1 r) \\ }}\psi(m)f_0\bigl(\frac{em}{M}\bigr).
		\end{align*}
		
		Set $W=[q_1r,\mathfrak{c}]$, $H=WeX^{\epsilon'}/M$ so that by Poisson summation (Lemma~\ref{lem_Poiss})
		\begin{align}\label{eq_S2_1}
		\sum_{\substack{emn_1\equiv a (q_1 r) \\ }}\psi(m)f_0\bigl(\frac{em}{M}\bigr)=\frac{M}{eW}\sum_{\substack{b(W)\\ b\equiv a\overline{e n_1}(q_1r)}} \psi(b)\sum_{|h|\leq H}\widehat{f_0}\bigl(\frac{hM}{eW}\bigr) e\bigl(\frac{-bh}{W}\bigr)+O_{\epsilon'}(X^{-100}).
		\end{align}
		The error term is negligible. We furthermore have
		\begin{align*}
		\sum_{\substack{b(W)\\ b\equiv a\overline{e n_1}(q_1r)}} \psi(b)=\psi( a\overline{e n_1})1_{\mathfrak{c}|q_1r},
		\end{align*}
		and so the term $h=0$ on the right-hand side of~\eqref{eq_S2_1} contributes to $\mathscr{S}_2'$
		\begin{align}\label{eq_S2_2}
		M\widehat{f_0}(0)\sum_{\substack{\mathfrak{c}\leq P\\ \psi(\mathfrak{c})^*}}\sum_{\substack{q_1,q_2 \\  \substack{(a,q_1q_2)=1\\ }}}\sum_{\substack{r\sim R\\ \substack{(r,an_1n_2)=1\\ }}}1_{\mathfrak{c}|r(q_1,q_2)}\frac{\gamma_{ q_1}\gamma_{q_2}}{\varphi( q_2 r)q_1r} \sum_{\substack{e|q_2 , \, (e,rq_1)=1\\ e\leq P^{1/2}}}\frac{\mu(e)}{e} \sum_{\substack{n_1,n_2\\ (n_i,q_i)=1}}\beta_{n_1}\overline{\psi(n_1)\beta_{n_2}}\psi(n_2). 
		\end{align}
		
		Observing that
		\begin{align*}
		\frac{1}{\varphi(q_2 r)}\sum_{\substack{e|q_2 , \, (e,rq_1)=1}}\frac{\mu(e)}{e}&=\frac{1}{q_2 r}\prod_{\substack{p|q_2r\\}}\Bigl(1-\frac{1}{p}\Bigr)^{-1}\prod_{\substack{p|q_2\\ p\nmid rq_1}}\Bigl(1-\frac{1}{p}\Bigr)\\
		&=\frac{1}{q_2 r}\prod_{\substack{p|(q_1,q_2)r\\}}\Bigl(1-\frac{1}{p}\Bigr)^{-1},
		\end{align*}
		the expression in~\eqref{eq_S2_2} equals to $\mathcal{X}$ up to admissible error after completing the sum over $e$ again, as we may. 
		
		To estimate the contribution of terms with $h\neq 0$ to the right-hand side of~\eqref{eq_S2_1}, we start by writing $W=q_1 r \mathfrak{c}'\widetilde{\mathfrak{c}}$ with $\mathfrak{c}'=(\mathfrak{c}/(\mathfrak{c},q_1r),(q_1r)^\infty)$ and  observe
		\begin{align}\nonumber
		\sum_{\substack{b(W)\\ b\equiv a\overline{e n_1}(q_1r)}} \psi(b)e\bigl(\frac{-bh}{W}\bigr)&=\sum_{\substack{b_1(q_1r \mathfrak{c}')\\ b_1\equiv a\overline{e n_1}(q_1r)}}\psi^{(\mathfrak{c}/\widetilde{\mathfrak{c}})}(b_1)e\bigl(\frac{-b_1h}{q_1 r\mathfrak{c}'}\bigr)\sum_{b_2(\widetilde{\mathfrak{c}})}\psi^{(\widetilde{\mathfrak{c}})}(b_2)e\bigl(\frac{-b_2h}{\widetilde{\mathfrak{c}}}\bigr)\\
		&=e\bigl(-ah\frac{\overline{e n_1}}{q_1 r \mathfrak{c}'}\bigr)\sum_{b'(\mathfrak{c}')}\psi^{(\mathfrak{c}/\widetilde{\mathfrak{c}})}(a\overline{e n_1}+b' q_1 r ) e\bigl(\frac{-b'h}{\mathfrak{c}'}\bigr)\sum_{b_2(\widetilde{\mathfrak{c}})}\psi^{(\widetilde{\mathfrak{c}})}(b_2)e\bigl(\frac{-b_2h}{\widetilde{\mathfrak{c}}}\bigr),\label{eq_S2_3}
		\end{align}
		where we factorised $\psi=\psi^{(\mathfrak{c}/\widetilde{\mathfrak{c}})}\psi^{(\widetilde{\mathfrak{c}})}$. Given $q_1$ and $\mathfrak{c}$, the condition $r\equiv c (\mathfrak{c})$ determines in particular also $\mathfrak{c}'$ and $\widetilde{\mathfrak{c}}$. Thus, after changing the order of summation and trivially summing over $b'$ and $b_2$ that come from~\eqref{eq_S2_3}, the contribution of terms with $h\neq 0$ in $\eqref{eq_S2_1}$ to $\mathscr{T}(n_1,n_2,q_1,q_2,\psi,c,e)$ can be estimated by
		\begin{align*}
		\ll \frac{M}{eq_1}\sum_{0<|h|\leq H}  \Bigl|\sum_{\substack{r\sim R\\ \substack{(r,ae n_1n_2)=1\\ r\equiv c (\mathfrak{c})}}}\widehat{f_0}\bigl(\frac{hM}{eq_1 r \mathfrak{c}'\widetilde{\mathfrak{c}} }\bigr)\frac{1}{r \varphi(q_2r)}e\bigl(-ah\frac{\overline{e n_1}}{q_1 r \mathfrak{c}'}\bigr) \Bigr|
		\end{align*}
		
		This is almost the same object that is considered in~\cite[Section 5]{bfi}. Similarly as there, a routine calculation using reciprocity, Weil's Kloosterman sum bound, and partial summation lets us bound it by
		\begin{align*}
		\ll \frac{X^{2\epsilon'}\mathfrak{c}}{eq_1 R}\Bigl(N^{1/2}+\frac{R}{N}\Bigr).
		\end{align*}
		This contributes to $\mathscr{S}_2'$ at most
		\begin{align*}
		\frac{Q N^2 X^{2\epsilon'}P^4}{R}\Bigl(N^{1/2}+\frac{R}{N}\Bigr),
		\end{align*}
		which is admissible by~\eqref{eq_congtodisp_cond01},~\eqref{eq_congtodisp_cond03}, and~\eqref{eq_congtodisp_cond04}.
		
		We complete the proof of Lemma~\ref{lem_congtodisp} by extracting $\mathcal{X}$ out of $\mathscr{S}_3$. Our strategy is closely related to the one in~\cite[Section 5.3.1]{drappeau}. Set now $W=r[q_1,q_2]$, $H=WX^{\epsilon'}/M$, and apply Poisson summation (Lemma~\ref{lem_Poiss}) to get
		\begin{align}\nonumber
		\mathscr{S}_3&=\sum_{\substack{r\sim R, \\ \substack{(r,a)=1\\ }}}\sum_{\substack{q_1, q_2 \\ \substack{ (a,q_1q_2)=1 \\ }}}\frac{\gamma_{q_1}\overline{\gamma_{q_2}}}{\varphi(q_1r)\varphi(q_2r)}\sum_{\substack{\psi_i(q_ir)\\ \text{cond}(\psi_i)\leq P}} \psi_1 \overline{\psi_2}( \overline{a}) \sum_{\substack{n_1,n_2 \\ \substack{ \\ }}}\beta_{n_1}\overline{\beta_{n_2}} \psi_1(n_1)\overline{\psi_2}(n_2)\sum_{b(W)^*}\psi_1\overline{\psi_2}(b) \\
		&\times \Bigg(\frac{1}{W}\sum_{|h|\leq H} \widehat{f_0}\bigl(\frac{h}{H}\bigr)e\bigl(\frac{-bh}{W}\bigr)+O_{\epsilon'}(X^{-100}) \Bigg) \label{eq_S_3}.
		\end{align}
		The error term is negligible. 
		
		We continue with the observation that
		\begin{align*}
		\sum_{b(W)^*}\psi_1\overline{\psi_2}(b) = \varphi(W)1_{\text{cond}(\psi_1\overline{\psi_2})=1}.
		\end{align*}
		Since
		\begin{align*}
		\sum_{\substack{\psi_i(q_ir)\\ \text{cond}(\psi_i)\leq P}} \psi_1 \overline{\psi_2}( \overline{a}) \psi_1(n_1)\overline{\psi_2}(n_2)1_{\text{cond}(\psi_1\overline{\psi_2})=1}= \sum_{\substack{\psi((q_1,q_2)r)\\\textnormal{cond}(\psi)\leq P}}\psi(n_1\overline{n_2})
		\end{align*}
		and 
		\begin{align*}
		\frac{\varphi([q_1,q_2]r)}{\varphi(q_1r)\varphi(q_2r)}=\frac{1}{\varphi((q_1,q_2)r)},
		\end{align*}
		the $h=0$ contribution to $\mathscr{S}_3$ in~\eqref{eq_S_3} equals $\mathcal{X}$. 
		
		The $h\neq 0$ contribution can be estimated using
	    the bound (see Lemma \ref{lem_ccalc}))
		\begin{align*}
		\bigl|\sum_{b(W)^*}\psi_1\overline{\psi_2}(b)e\bigl(\frac{-bh}{W}\bigr)\bigr|\ll P (h,W).
		\end{align*}
		This gives a contribution of
		\begin{align*}
		&\ll M P (\log X)\sum_{\substack{r\sim R, \\ \substack{(r,a)=1\\ }}}\sum_{\substack{q_1, q_2 \\ \substack{ (a,q_1q_2)=1 \\ }}}\frac{\bigl|\gamma_{q_1}\overline{\gamma_{q_2}}\bigr|}{W\varphi(q_1r)\varphi(q_2r)}\sum_{\substack{\psi_i(q_ir)\\ \text{cond}(\psi_i)\leq P}}  \sum_{\substack{n_1,n_2 \\ \substack{ \\ }}}\bigl|\beta_{n_1}\overline{\beta_{n_2}}\bigr| \sum_{0<|h|\leq H} (h,W)\\
		&\ll X^{\epsilon'} P (\log X) \sum_{\substack{r\sim R, \\ \substack{(r,a)=1\\ }}}\tau(r)^3\sum_{\substack{q_1, q_2 \\ \substack{ (a,q_1q_2)=1 \\ }}}\frac{\bigl|\tau(q_1)^2\gamma_{q_1}\tau(q_2)^2\overline{\gamma_{q_2}}\bigr|}{\varphi(q_1r)\varphi(q_2r)}  \sum_{\substack{n_1,n_2 \\ \substack{ \\ }}}\bigl|\beta_{n_1}\overline{\beta_{n_2}}\bigr| \\
		&\ll X^{\epsilon'} P (\log X)^{O_C(1)} N^2 R^{-1}.
		\end{align*}
		This is admissible by~\eqref{eq_BFI_Rough_1} and~\eqref{eq_congtodisp_cond01}.
	\end{proof}
	
	Next we translate the remaining term of the previous lemma, $\mathcal{E}(q_0,\nu)$, into sums of Kloosterman sums.
	
	\begin{lemma}[Reduction to Kloosterman sums]\label{lem_redtokloo} Assume Convention~\ref{def_BFIrough_assumptions} and that
		\begin{align}\label{eq_redtoklo_c0}
		\beta_n \textnormal{ is supported on squarefree } n \textnormal{ only. }
		\end{align}  Assume further
		\begin{align}
		Q^2 R N &\leq X^{2-\epsilon} \label{eq_redtoklo_c2}.
		\end{align}
		Let $\mathscr{E}(q_0,\nu)$ be given by~\eqref{eq_mscrEdef}.
		Then we have 
		\begin{align*}
		|\mathscr{E}(q_0,\nu)|&\ll M\nu^5 W(q_0,\nu)+\frac{X^{1-\epsilon'}N}{R},
		\end{align*}
		where
		\begin{align}\label{eq_Wdef}
		W(q_0,\nu)&\coloneqq \max_{l_1,l_2,l_3,l_4.l_5(\nu)}\Bigg|\frac{1}{q_0}\sum_{\substack{\substack{r\sim R\\ r\equiv l_3}\\(r,a\nu)=1}}\frac{1}{r}\sum_{\substack{q_i\\ \substack{(a\nu,q_1q_2)=1\\ \substack{(q_1,q_2)=1}\\ q_i\equiv l_i(\nu)}}}\frac{\gamma_{q_0q_1}\overline{\gamma_{q_0q_2}}}{q_1q_2}\sum_{\substack{n_1,n_2 \\ \substack{(n_1,n_2)=1 \\ n_1\equiv n_2 (q_0r)\\\substack{(n_i,q_0q_ir)=1\\ n_1\equiv l_4(\nu)}}}}\beta_{\nu n_1}\overline{\beta_{\nu n_2}}\\
		&\nonumber \times \sum_{\substack{0<|h|\leq H \\ h\equiv l_5(\nu)}}\widehat{f_0}\bigl(\frac{hM}{q_0q_1q_2r} \bigr)e\left(ah\frac{n_2-n_1}{q_0 r}\frac{\overline{\nu q_1n_2}}{q_2n_1}\right)\Bigg|.
		\end{align}
	\end{lemma}
	
	\begin{proof}
		Given $n_i, q_i, r$ in the range of summation in ~\eqref{eq_mscrEdef}, let $\ell$ be such that
		\begin{align}\begin{split}\label{e14}
		\ell \nu n_1&\equiv a(q_0q_1r)\\
		\ell \nu n_2&\equiv a(q_0q_2r).
		\end{split}
		\end{align} 
		Let 
		\begin{align}\label{e21}
		H\coloneqq X^{\epsilon'} Q^2 R/M
		\end{align}
		and apply Poisson summation (Lemma~\ref{lem_Poiss}) to get
		\begin{align}\label{e10b}\begin{split}
		\sum_{\substack{m\\ m \nu n_i\equiv a(q_0q_i r)}}f_0\bigl(\frac{m}{M}\bigr)-\frac{M\widehat{f_0}(0)}{q_0q_1q_2r}&=\sum_{\substack{m\\ m \equiv \ell(q_0q_1q_2r)}}f_0\bigl(\frac{m}{M}\bigr)-\frac{M\widehat{f_0}(0)}{q_0q_1q_2r}\\
		&=\frac{M}{q_0q_1q_2 r}\sum_{1\leq |h| \leq H}\widehat{f_0}\left(\frac{hM}{q_0q_1q_2r}\right)e\left(\frac{-\ell h}{q_0q_1q_2r}\right)+O\left(X^{-100}\right).
		\end{split}
		\end{align}
		The error term is small enough.
		
		At this point, since we cannot make the assumption $(A_4)$ in~\cite{bfi} (roughness of $\beta_n$) to reduce to the case $\nu=1$, we instead follow the argument after~\cite[(3.18)]{fouvry-iwaniec} almost verbatim. Write $t=(n_2-n_1)/(q_0r)$. The congruence conditions~\eqref{e14} are equivalent to
		\begin{align*}
		\ell\nu n_1&=a+q_0q_1r u_1\\
		\ell\nu n_2&=a+q_0q_2r u_2
		\end{align*}
		for some integers $u_1, u_2$. Thus $q_2u_2n_1-q_1u_1n_2=at$ and
		\begin{align*}
		u_1\equiv -at \overline{q_1n_2} (q_2n_1).
		\end{align*}
		Furthermore, 
		\begin{align*}
		u_1\equiv -a \overline{r q_0 q_1} (\nu)
		\end{align*}
		and so, using that by~\eqref{eq_redtoklo_c0} we can assume $(n_1,\nu)=1$,
		\begin{align*}
		\frac{\ell h}{q_0 q_1 q_2 r}&=\frac{ah}{\nu q_0 q_1 q_2 r n_1}+\frac{u_1h}{\nu n_1 q_2}\\
		&\equiv -ah\frac{\overline{r q_0 q_1 q_2 n_1}}{\nu} -a t h\frac{\overline{q_1 n_2 \nu }}{q_2n_1}+\frac{ah}{\nu q_0 q_1 q_2 r n} \pmod 1,
		\end{align*}
		so
		\begin{align}\label{e11b}
		e\left(\frac{-\ell h}{q_0q_1q_2r}\right)=e\left(ah\frac{\overline{r q_0q_1q_2 n_1}}{\nu}\right)e\left(ah\frac{n_2-n_1}{q_0r}\frac{\overline{\nu q_1n_2}}{q_2n_1}\right)+O\left(\frac{q_0|ah|}{\nu NQ^2 R}\right).
		\end{align}
		The contribution of the error terms present in ~\eqref{e11b} to $\mathscr{E}(\nu,q_0)$ is crudely bounded using the divisor bound and $|\widehat{f_0}(y)|\ll 1$ by
		\begin{align*}
		\ll_a& \frac{q_0 X^{2\epsilon'}}{X} \frac{N^2 Q^2}{q_0^3},
		\end{align*}
		which is admissible by~\eqref{eq_redtoklo_c2}. The term we still have to consider is
		\begin{align*}
		&\frac{M}{q_0}\sum_{\substack{r\sim R\\(r,a\nu)=1}}\frac{1}{r}\sum_{\substack{q_1, q_2 \\ \substack{(a\nu,q_1q_2)=1\\ (q_1,q_2)=1}}}\frac{\gamma_{q_0q_1}\overline{\gamma_{q_0q_2}}}{q_1q_2}\sum_{\substack{n_1,n_2 \\ \substack{(n_1,n_2)=1 \\ n_1\equiv n_2 (q_0r)\\(n_i,q_0q_ir)=1}}}\beta_{\nu n_1}\overline{\beta_{\nu n_2}}\\
		&\times \sum_{0<|h|\leq H}\widehat{f_0}\left(\frac{hM}{q_0q_1q_2r}\right)e\left(ah\frac{\overline{r q_0q_1q_2 n_1}}{\nu}\right)e\left(ah\frac{n_2-n_1}{q_0r}\frac{\overline{\nu q_1n_2}}{q_2n_1}\right).
		\end{align*}
		The lemma follows after restricting $r$, $q_1$, $q_2$, $n_1$, and $h$ into fixed residue classes $\pmod \nu$ and taking the maximum over those.
	\end{proof}

	\subsection{Estimates for $W(q_0,\nu)$}
	In this section we estimate $W(q_0,\nu)$ and as a consequence deduce the remaining cases \ref{(T.2)}, \ref{(T.3)}, and \ref{(T.5)} of Proposition~\ref{theorem_BFI_rough}.
	
	The following lemma gives us case \ref{(T.2)}.
	
	\begin{lemma}[Variant of Lemma 8.1 in~\cite{MaynardII}]\label{lem_MII8.1} Assume Convention~\ref{def_BFIrough_assumptions} and let $W(q_0,\nu)$ be as in~\eqref{eq_Wdef} with $q_0,\nu\leq X^{\epsilon'}$. Let further $\gamma=\lambda \star \rho$ with $\lambda_s$ and $\rho_t$ being coefficient sequences of order $C$ supported on $t\sim T$, $s\sim S$ (so $ST\asymp Q$). Assume that
		\begin{align}
		N^2 T^2 S &<X^{1-\epsilon}, \\
		N^2 T^3 S^4 R&< X^{2- \epsilon},\\
		N T^2 S^5 R &< X^{2-\epsilon}.
		\end{align}
		Then we have
		\begin{align*}
		W(q_0,\nu)\ll \frac{N^2}{RX^{\epsilon'}}.
		\end{align*}
		In particular, case \ref{(T.2)} of Proposition~\ref{theorem_BFI_rough} holds.
	\end{lemma}
	This result is closely related to Maynard's~\cite[Lemma 8.1]{MaynardII}, but differs in the following aspects:
	\begin{itemize}
		\item The variables $s_i,t_i$ (appearing after Cauchy--Schwarz) and $n_1$, $r$, and $h$ are restricted to a congruence class.
		\item There are factors $q_0$ and $\nu$ in the coefficients.
		\item There is no lower bound for $|n_1-n_2|$ and no Fouvry-style $\mathcal{N}$ condition.
		\item The exponential phase is of slightly different shape and has an additional $\nu$ term.
		\item The condition $(q_2 s_2, r q_0)=1$ is missing.
		\item There is an additional $q_0$ dependence in $\lambda$ and $\lambda'$.
		\item The variable $H'$ is slightly larger, in Maynard's statement $X^{o(1)}$ replaces $X^{\epsilon'}$.
		\item There is the $\widehat{f_0}\left(\frac{hM}{q_0q_1q_2r}\right)$ term.
		\item Our coefficients are only divisor bounded, not $1$-bounded.
	\end{itemize}
	These changes are harmless and require no considerable modifications of the original argument.
	
	\begin{proof}[Proof of Lemma~\ref{lem_MII8.1}] 
		First we remove the term $\widehat{f_0}\left(\frac{hM}{q_0q_1q_2r}\right)$ by summation by parts. We have $|\widehat{f_0}^{(j)}(\xi)|\ll_{j,k} |\xi|^{-k}$ and recall that $f$ is supported on $[1/2,5/2]$. Thus
		\begin{align*}
		\frac{\partial^{k_1+k_2+k_3+k_4}}{\partial r^{k_1}\partial {q_1}^{k_2} \partial {q_2}^{k_3} \partial h^{k_4}}\widehat{f_0}\left(\frac{hM}{q_0q_1q_2r}\right)\ll_{k_1,k_2,k_3,k_4} r^{-k_1}{q_1}^{-k_2}{q_2}^{-k_3}h^{-k_4}  
		\end{align*}
		and so
		\begin{align*}
		W(q_0,\nu)\ll W'(q_0,\nu),
		\end{align*}
		where
		\begin{align*}
		W'(q_0,\nu)&= \sup_{\substack{\substack{H'\leq 2H\\ R'\leq 2R}\\ Q_1,Q_2\leq 2Q}}\Bigg|\sup_{l_1,l_2,l_3,l_4.l_5(\nu)} \sum_{\substack{\substack{R\leq r \leq R'\\ r\equiv l_3(\nu)}\\(r,a\nu)=1}}\frac{1}{r}\sum_{\substack{q_1\leq Q_1, q_2\leq Q_2 \\ \substack{(a\nu,q_1q_2)=1\\ \substack{(q_1,q_2)=1}\\ q_i\equiv l_i(\nu)}}}\frac{\gamma_{q_0q_1}\overline{\gamma_{q_0q_2}}}{q_1q_2}\sum_{\substack{n_1,n_2 \\ \substack{(n_1,n_2)=1 \\ n_1\equiv n_2 (q_0r)\\\substack{(n_i,q_0q_ir)=1\\ n_1\equiv l_4(\nu)}}}}\beta_{\nu n_1}\overline{\beta_{\nu n_2}} \\
		&\times \sum_{\substack{0<|h|\leq H' \\ h\equiv l_5(\nu)}}e\left(ah\frac{n_2-n_1}{q_0 r}\frac{\overline{\nu q_1n_2}}{q_2n_1}\right)\Bigg|.
		\end{align*}
		For the remainder of the proof we consider the fixed choice of $H', R', Q_1, Q_2, l_1, l_2, l_3, l_4$ that corresponds to the largest output.
		
		Essential for Maynard's improvements over~\cite[Theorem 1]{bfi} is the factorisation $\gamma=\lambda\star \rho$. We apply it and furthermore absorb the congruence condition $q_i\equiv l_i(\nu)$ and the factor $q_0$ in $\gamma_{q_0q_i}$ into new coefficients. More precisely, for any fixed $q_0',q_0''$, $l_i',l_i''$ with
		\begin{align*}
		q_0&=q_0'q_0'',\\
		l_1&\equiv l_1' l_1''(\nu),\\
		l_2&\equiv l_2' l_2''(\nu),
		\end{align*}
		write
		\begin{align}\label{eq_lambda'}\begin{split}
		\lambda'_{s_1}&= X^{-\epsilon'/100}\lambda_{q_0' s_1}1_{s_1\equiv l_1'(\nu)}\\
		\lambda''_{s_2}&= X^{-\epsilon'/100}\lambda_{q_0' s_2}1_{s_2\equiv l_2'(\nu)}\\
		\rho'_{t_1}&=X^{-\epsilon'/100}\rho_{q_0'' t_1} 1_{t_1\equiv l_1'(\nu)}\\
		\rho''_{t_2}&=X^{-\epsilon'/100}\rho_{q_0'' t_2} 1_{t_2\equiv l_2''(\nu)}\\
		\beta'_{n_1}&=X^{-\epsilon'/100}\beta_{\nu n_1} 1_{n_1\equiv l_4(\nu)}\\
		\beta''_{n_2}&=X^{-\epsilon'/100}\beta_{\nu n_2} ,
		\end{split}
		\end{align}
		where the $X^{-\epsilon'/100}$ term ensures that the new coefficients are $1$-bounded, if $X$ is sufficiently large.
		
		Set
		\begin{align*}
		N'&=N/\nu
		\end{align*}
		and observe that
		\begin{align*}
		\frac{N'^2}{R X^{2\epsilon'}}\ll \frac{N^2}{R \nu^2  \tau(q_0) X^{11\epsilon'/10}}.
		\end{align*}
		Thus, for some choice of $q_0',q_0'',l_i,l_i''$ in~\eqref{eq_lambda'} it suffices to show
		\begin{align}\label{eq_lemMII8.1_1}
		W_1\ll  \frac{N'^2}{R X^{2\epsilon'}},
		\end{align}
		where
		\begin{align*}
		W_1\coloneqq \sum_{\substack{\substack{R\leq r \leq R'\\ r\equiv l_3(\nu)}\\(r,a\nu)=1}}\frac{1}{r}\sum_{\substack{s_1t_1\leq Q_1, s_2t_2\leq Q_2 \\ \substack{(a\nu,s_1t_1s_2t_2)=1\\ \substack{(s_1t_1,s_2t_2)=1}\\ }}}\frac{\lambda'_{s_1}\rho'_{t_1}\overline{\lambda''_{s_2}\rho''_{t_2}}}{ s_1t_1 s_2t_2}\sum_{\substack{n_1,n_2 \\ \substack{(n_1,n_2)=1 \\ n_1\equiv n_2 (q_0r)\\\substack{(n_i,q_0q_ir)=1\\ }}}}\beta'_{n_1}\overline{\beta''_{n_2}} \sum_{\substack{0<|h|\leq H' \\ h\equiv l_5(\nu)}}e\left(ah\frac{n_2-n_1}{q_0 r}\frac{\overline{\nu q_1n_2}}{q_2n_1}\right).
		\end{align*}
		Here the sequences are supported on
		\begin{align*}
		n_i&\sim N'\\
		s_i&\sim S'\\
		t_i&\sim T'
		\end{align*}
		for $N'=N/\nu$ and for some
		\begin{align*}
		S/q_0\leq S' &\leq S\\
		T/q_0\leq T'&\leq T.
		\end{align*}
		
		We now follow the steps in~\cite[proof of Lemma 8.1]{MaynardII} and write 
		\begin{align*}
		f=(n_1-n_2)/(q_0r)
		\end{align*}
		for some $1\leq |f|\leq 2N'/q_0R$ (note that $n_1\neq n_2$ by $(n_1,n_2)=1$). Thus the exponential becomes
		\begin{align*}
		e\Bigl(\frac{ahf \overline{\nu q_1 s_1 n_2}}{q_2s_2n_1}\Bigr),
		\end{align*}
		which is the same as in~\cite[Lemma 8.1]{MaynardII} after the $f$ substitution, except of the additional $\overline{\nu}$. We get
		\begin{align*}
		W_1&=\sum_{1\leq |f| \leq 2N/R} \sum_{\substack{t_i \\ (t_1t_2,a)=1}} \sum_{\substack{s_2\\ \substack{(t_2 s_2, a t_1)=1}}} \primesum_{\substack{n_i \\ n_1\equiv n_2 (q_0 f)}} \frac{\beta'_{n_1}\rho'_{t_1}\overline{\beta''_{n_2}\lambda''_{s_2}\rho''_{t_2}}q_0f}{t_1 t_2 s_2 (n_1-n_2)}\\
		&\times \sum_{\substack{s_1 \\ \substack{(s_1,an_1 s_2t_2)=1\\ s_1t_1\leq Q_1}}}\frac{\lambda'_{s_1}}{s_1} \sum_{\substack{1\leq |h|\leq H'\\ h\equiv l_2(\nu)}}e\Bigl(\frac{ahf \overline{\nu t_1 s_1 n_2}}{t_2s_2n_1}\Bigr),
		\end{align*}
		where $\sum'$ encodes the following summation conditions
		\begin{align*}
		(n_1,n_2)&=1\\
		(n_i,s_it_i)&=1\\
		(n_1-n_2)/(fq_0)&\equiv l_1(\nu)\\
		((n_1-n_2)/(fq_0),a\nu)&=1\\
		Rq_0f\leq n_1-n_2&\leq R' q_0f.
		\end{align*}
		
		We can remove the condition $ s_1 t_1\leq Q_1$ by Fourier analysis and apply the Cauchy--Schwarz inequality, verbatim as in~\cite[proof of Lemma 8.1]{MaynardII}, to get
		\begin{align*}
		W_1\ll \frac{N ' (\log X)^3}{R T' S'^{3/2}} \sup_\theta \bigl| W_1'(\theta) \bigr|^{1/2},
		\end{align*}
		where
		\begin{align*}
		W_1'(\theta)&\coloneqq \sum_{1\leq |f| \leq 2N'/R} \sum_{\substack{n_1,n_2 \sim N' \\ n_1\equiv n_2 (q_0 f)}}\sum_{t_i\sim T_i} \sum_{\substack{s_2\sim S'\\ (n_2t_1,n_1t_2s_2)=1}}  \\
		&\times \Bigg|\sum_{\substack{s_1\sim S'\\ (s_1,an_1 t_2s_2)=1}}\frac{S'\lambda'_{s_1} e(s_1\theta)}{s_1}\sum_{\substack{1\leq |h|\leq H'\\ h\equiv l_2(\nu)}} e\Bigl(\frac{ahf \overline{\nu t_1 s_1 n_2}}{t_2s_2n_1}\Bigr)\Bigg|^2,
		\end{align*}
		which is almost the same as Maynard's $\mathscr{W}_3$ with the only differences being the congruence condition on $h$ and the term $\overline{\nu}$ in the exponential. 
		
		In Maynard's proof there is no summation with cancellation over the $h$ variable, so we can drop the congruence condition on $h$ (the two variables $b$ and $c$ that are summed with cancellation arise in our notation from  $n_2t_1$ and $n_1t_2s_2$). Finally, the additional term $\overline{\nu}$ can be incorporated in the variable Maynard calls $z$. As a consequence, we replace his $b_{z,y}$ by
		\begin{align*}
		b'_{z,y}&=\sum_{s_1,s_2\sim S'}\sum_{\substack{1\leq |h_1|,|h_2|\leq H' \\ \substack{s_1s_2\nu=z\\af(h_1s_2-h_1s_2)=y}}}\sum_{f\sim F}1\\
		&=1_{\nu|z} b_{z/\nu,y}
		\end{align*} 
		and this changes $Z$ to $Z'=\nu Z$. By Maynard's proof, this increase, the slightly larger choice of $H'$, and the fact that we replace the bound $NM\gg X$ by
		\begin{align*}
		N'M\gg X^{1-\epsilon'} 
		\end{align*}
		are accounted for if $\epsilon$ is sufficiently large in terms of $\epsilon'$. Thus, observing the upper bounds $S'\leq S$, $T'\leq T$, we obtain~\eqref{eq_lemMII8.1_1} and so the lemma in the required range.
		
		By Lemmas~\ref{lem_initred},~\ref{lem_congtodisp},~\ref{lem_redtokloo}, this proves case \ref{(T.2)} of Proposition~\ref{theorem_BFI_rough}, after choosing the $\epsilon'$ there small enough.	\end{proof}

	We continue with the case \ref{(T.3)}.
	
	\begin{lemma}[Variant of~\cite{bfi} Section 9]\label{lem_BFIIsec9} Assume Convention~\ref{def_BFIrough_assumptions} and let $W(q_0,\nu)$ as in~\eqref{eq_Wdef} with $q_0,\nu\leq X^{\epsilon'}$. Assume that
		\begin{align*}
		N^2Q^2R^{-1}&\leq X^{1-\epsilon}\\
		N^5Q^2&\leq X^{2-\epsilon}\\
		N^4Q^3&\leq X^{2-\epsilon}.
		\end{align*}
		Then 
		\begin{align*}
		W(q_0,\nu)\ll \frac{N^2}{RX^{\epsilon'}}.
		\end{align*}
		In particular, case \ref{(T.3)} of Proposition~\ref{theorem_BFI_rough} holds.
	\end{lemma}
	This result is closely related to Bombieri--Friedlander--Iwaniec's work~\cite[Section 9]{bfi} with the following harmless differences.
	\begin{itemize}
		\item There is an additional congruence condition on the $h$ and $r$ summation.
		\item There is an additional term $\overline{\nu}$ in the exponential.
		\item There is a $\nu$ factor in the $\beta$ coefficients.
		\item There is the restriction $(r,\nu)=1$.
	\end{itemize}
	
	\begin{proof}
		We follow the strategy of proof in~\cite[Section 9]{bfi} with the corrections given in~\cite{bficor}. 
		
		We remove the condition $(r,a\nu)=1$ with the help of Möbius inversion,
		\begin{align*}
		1_{(r,a\nu)=1}=\sum_{\delta|(r,a\nu)}\mu(\delta).
		\end{align*}
		Afterwards we change the variable $r$ into $k$ (compare~\cite[eq. (9.3), (9.4)]{bfi}) with
		\begin{align*}
		n_2-n_1&=\delta q_0 r k\\
		(\delta q_0 k ,n_1n_2)&=1\\
		n_2&\equiv n_1  (\delta q_0 k)\\
		n_2-n_1 &\equiv \delta l_3 q_0 k (\nu).
		\end{align*}
		
		To remove the coupling of $n_2$ and $k$ from the congruence condition, we restrict $k$ into a fixed class $k\equiv l(\nu)$. We can now, similarly to the start of the proof of Lemma~\ref{lem_MII8.1}, absorb several of our additional conditions into new coefficients. More precisely, let
		\begin{align*}
		\gamma'_{q_0q_1}&=\gamma_{q_0q_1}1_{q_1\equiv l_1(\nu)}1_{(q_1,\nu)=1}\\
		\gamma''_{q_0q_2}&=\gamma_{q_0q_2}1_{q_1\equiv l_2(\nu)}1_{(q_2,\nu)=1}\\
		\beta'_{n_1}&=\beta_{\nu n_1} 1_{n_1\equiv l_4(\nu)}\\
		\beta''_{n_2}&=\beta_{\nu n_2}1_{n_2\equiv l_3 q_0 l+l_4(\nu)} .
		\end{align*}
		Taking the maximum over $l_i(\nu)$, removing $\widehat{f_0}$ just as in the proof of Lemma~\ref{lem_MII8.1}, applying the triangle inequality, and using the lower bounds $r\geq R$, $q_i\geq Q/q_0$, it suffices to show
		\begin{align*}
		W_2\ll \frac{N^2 Q^2}{\nu q_0^2 X^{\epsilon'}},
		\end{align*}
		where
		\begin{align*}
		W_2=\sum_{\delta|a\nu}\sum_{1\leq |k| \leq K}\sum_{\substack{q_i \\ \substack{(a\nu,q_1q_2)=1\\ \substack{(q_1,q_2)=1}\\ }}}\bigl|\gamma'_{q_0q_1}\overline{\gamma''_{q_0q_2}}\bigr| \Bigg|\sum_{\substack{n_1,n_2 \\ \substack{(n_1,n_2)=1 \\ n_1\equiv n_2 (\delta q_0k)\\\substack{(n_i,q_iq_0k)=1\\ q_0|k|R<|n_2-n_1|\leq q_0|k|R'}}}}\beta'_{n_1}\overline{\beta''_{n_2}} \sum_{\substack{0<|h|\leq H' \\ h\equiv l_5(\nu)}}e\left(ahk\frac{\overline{\nu q_1n_2}}{q_2n_1}\right)\Bigg|,
		\end{align*}
		$K_0=N/q_0R$, $R\leq R'\leq 2R$, and $\gamma'_{q}, \gamma''_{q}$, $\beta'_{n}$ $\beta''_{n}$ are supported on $q\sim Q, n\sim N'=N/\nu$ respectively.

		At this point we can open the congruence condition $n_2\equiv n_1 (\delta q_0 k)$, $(n_1n_2,\delta q_0 k)$ with the help of characters and the range condition on $|n_2-n_1|$ with Fourier analysis, just as in~\cite[eq. (9.7)]{bficor} and~\cite[eq. (9.8)]{bfi} to get
		\begin{align*}
		W_2&\ll (\log 2N) \sum_{\delta|a\nu}\sum_{\substack{\substack{1\leq |k| \leq K_0\\ }\\}}\sum_{\psi(k\delta q_0)}\frac{1}{\varphi(k\delta q_0)}\sum_{\substack{q_i \\ (q_1,q_2)=1}}|\gamma'_{q_0q_1}\gamma''_{q_0q_2}| \sum_{\substack{n_1 \\ (n_1,q_1)=1}}|\beta'_{n_1}|\\
		&\times\sum_{1\leq |h|\leq H} \left|\sum_{\substack{n_2 \\ (n_2,n_1q_1)=1}}\beta(h,n_2)\chi \psi(n_2) e\left(ahk\frac{\overline{\nu q_1n_2}}{q_2n_1}\right)\right|
		\end{align*}
		for some $|\beta(h,n_2)|=|\beta''_{n_2}|$. This is essentially the same object as in~\cite[eq. (9.12)]{bficor}, as the fact that $q_1, q_2$ and $n_1,n_2$ are associated to different sequences is irrelevant. We can absorb the factor $\nu$ into the $n_2$ variable, only increasing this instance of $N$ to $\nu N$. Afterwards the proof in the remainder of~\cite[section 9]{bfi} goes through with the simplifications that as $\beta_n$ are supposed to be divisor-bounded, we do not actually need the assumption that they are supported on squarefree $n$ only for this step, and condition~\cite[(A.5)]{bfi} holds obviously (unless $\sum_{n}|\beta_n|^2\ll N^{1-\varepsilon/10}$ in which case Proposition~\ref{theorem_BFI_rough} follows trivially from Cauchy--Schwarz).
		
		By Lemma~\ref{lem_initred}, Lemma~\ref{lem_congtodisp}, and~\ref{lem_redtokloo}, this proves case \ref{(T.3)} of Proposition~\ref{theorem_BFI_rough}.
		
	\end{proof}

	We end this section with proving the case \ref{(T.5)}. While it is based on the work of Bombieri--Friedlander--Iwaniec~\cite[Section 13]{bfi}, the fact that Lemma~\ref{lem_redtokloo} introduces a congruence condition on the $q_i$ makes it necessary to use Drappeau's extension of the arguments on the spectral side given in~\cite[Theorem 2.1]{drappeau} that is built precisely to handle this situation.
	
	\begin{lemma}[\cite{bfi} Section 13 meets~\cite{drappeau} Theorem 2.1]\label{lem_T5}
		Let $W(q_0,\nu)$ as in~\eqref{eq_Wdef} with $\gamma_q=f_Q(q/Q)$ as in Lemma~\ref{lem_initred}(5). Let further $\nu, q_0\leq X^{\epsilon'}$. Assume that
		\begin{align}\label{eq_lem_T5_1}
		N^3 R &\leq X^{1-\epsilon}\\
		Q^2 R&\leq X. \label{eq_lem_T5_2}
		\end{align}
		Then we have
		\begin{align*}
		W(q_0,\nu)\ll \frac{N^2}{RX^{\epsilon'}}.
		\end{align*}
		In particular, case \ref{(T.5)} of Proposition~\ref{theorem_BFI_rough} holds.
	\end{lemma}

	This result is closely related to~\cite[Section 13]{bfi} with the following differences.
	\begin{itemize}
		\item There is an additional congruence condition on the $h$ and $r$ summation.
		\item There is an additional term $\overline{\nu}$ in the exponential.
		\item There is a $\nu$ factor in the $\beta$ coefficients.
		\item There is the restriction $(r,\nu)=1$.
		\item There is a congruence condition on the $q_i$.
	\end{itemize}
	Apart from the congruence condition on the $q_i$, these modifications are again harmless.
	
	\begin{proof}
		
		Similarly as in~\cite[Section 13]{bfi} we write
		\begin{align*}
		\widehat{f_Q}\Bigl(\frac{hM}{q_0 q_1 q_2 r}\Bigr)=\frac{q_0 q_1 q_2}{M}\int_{-\infty}^\infty f_Q\Bigl(\frac{\xi q_0 q_1 q_2}{M}\Bigr)e(\xi h/r)d\xi
		\end{align*}
		and detect the condition $(a, q_1q_2)$ by M\"{o}bius inversion so that 
		\begin{align*}
		W(q_0,\nu)&= M^{-1}\Bigg|\sup_{\substack{l_1,l_2,l_3,l_4.l_5(\nu)\\ (l_1l_2,\nu)=1}}\sum_{\substack{\delta_1\delta_2|a\\ \substack{(\delta_1,\delta_2)=1\\ (\delta_1\delta_2,\nu)=1}}}\int_{-\infty}^\infty \sum_{\substack{\substack{r\sim R\\ r\equiv l_3}\\(r,a\nu)=1}}\frac{1}{r}\sum_{\substack{q_i\\ \substack{(q_1,q_2)=1\\ q_i\equiv \overline{\delta_1}l_i(\nu)}}}\gamma_{\delta_1q_0q_1}\overline{\gamma_{\delta_2 q_0q_2}}f_Q\Bigl(\frac{\xi\delta_1\delta_2 q_0 q_1 q_2}{M}\Bigr)\\
		&\times \sum_{\substack{n_1,n_2 \\ \substack{(n_1,n_2)=1 \\ n_1\equiv n_2 (q_0r)\\\substack{(n_i,q_0q_ir)=1\\ n_1\equiv l_4(\nu)}}}}\beta_{\nu n_1}\overline{\beta_{\nu n_2}}\sum_{\substack{0<|h|\leq H \\ h\equiv l_5(\nu)}}e(\xi h/r)e\left(ah\frac{n_2-n_1}{q_0 r}\frac{\overline{\nu q_1n_2}}{q_2n_1}\right)d\xi\Bigg|.
		\end{align*}
		
		Taking the supremum of $\delta_1,\delta_2, \xi$ and specialising to the maximal choice of admissible $l_i$, it suffices to show
		\begin{align}\label{eq_lem_T5_3}
		W_3&\ll \frac{N^2Q^2}{R X^{2\epsilon'}},
		\end{align} 
		where
		\begin{align*}
		W_3&=\sum_{\substack{\substack{r\sim R\\ r\equiv l_3}\\(r,a\nu)=1}}\frac{1}{r}\sum_{\substack{q_i\\ \substack{(q_1,q_2)=1\\  q_i\equiv \overline{\delta_1}l_i(\nu)}}}\gamma_{\delta_1q_0q_1}\overline{\gamma_{\delta_2 q_0q_2}}f_Q(\xi\delta_1\delta_2 q_0 q_1 q_2)\sum_{\substack{n_1,n_2 \\ \substack{(n_1,n_2)=1 \\ n_1\equiv n_2 (q_0r)\\\substack{(n_i,q_0q_ir)=1\\ n_1\equiv l_4(\nu)}}}}\beta_{\nu n_1}\overline{\beta_{\nu n_2}} \\
		&\times \sum_{\substack{0<|h|\leq H \\ h\equiv l_5(\nu)}}e(\xi h/r)e\left(ah\frac{n_2-n_1}{q_0 r}\frac{\overline{\nu \delta_1 q_1n_2}}{\delta_2 q_2n_1}\right).
		\end{align*}
		
		This is a mixture of Drappeau's $\mathcal{R}_1''$ (defined after~\cite[eq. (5.25)]{drappeau}) and~\cite[eq. (13.2)]{bfi}. (Note an inaccuracy in~\cite[eq. (13.2)]{bfi}: Similarly as in Remark~\ref{re_bfi1}, the term $\overline{\delta_1}$ cannot be made to $1/\delta_1$. Here this is without consequence.). We write $W_3$ in the form of~\cite[eq. (2.3)]{drappeau} with
		\begin{align*}
		&\bm{c}\leftarrow q_2 &&\bm{d}\leftarrow q_1 &&\bm{n}\leftarrow ah (n_2-n_1)/(q_0r) && \bm{r}\leftarrow \nu \delta_1 n_2 &\bm{s}\leftarrow \delta_2 n_2 && \bm{q}\leftarrow \nu\\
		&\bm{C}\leftarrow Q/(q_0\delta_2) &&\bm{D}\leftarrow Q/(q_0\delta_1) &&\bm{N}\leftarrow |a|HN/(q_0R\nu ) && \bm{R}\leftarrow \delta_1 N &\bm{S}\leftarrow \delta_2 N/\nu.
		\end{align*} 
		The new coefficients are given by
		\begin{align*}
		b_{\bm{n},\bm{r},\bm{s}}=\sum_{\substack{r\sim R \\ r\equiv l_3(\nu)\\ (r,a\nu)=1 }}\frac{1}{r}\sum_{\substack{n_1,n_2\\(n_1,n_2)=1\\ n_1\equiv n_2 (q_0r) \\ n_1\equiv l_4(\nu)}}\beta_{\nu n_1}\overline{\beta_{\nu n_2}}\sum_{\substack{\bm{r}=\nu \delta_1 n_2\\ \bm{s}=\delta_2n}} \sum_{\substack{0<|h|\leq H\\ h\equiv l_5(\nu)\\ \bm{n}=ah(n_1-n_2)/(q_0r)}}e(\xi h/r).
		\end{align*}
		
		By~\cite[Theorem 2.1]{drappeau},  (with the correction in~\cite{bficor} that also applies here and means that there should be no $\bm{S}^{-1}$ factor)  we get
		\begin{align*}
		W_3\ll X^{O(\epsilon')}\bm{q}^{3/2}\Bigl(\bm{q}\bm{C}\bm{S}(\bm{R}\bm{S}+\bm{N})(\bm{C}+\bm{R}\bm{D})+\bm{C}^2\bm{D}\bm{S}\sqrt{\bm{R}\bm{S}+\bm{N}}\sqrt{\bm{R}}+\bm{D}^2\bm{N}\bm{R} \Bigr)^{1/2}\|b_{\bm{n},\bm{r},\bm{s}}\|_2.
		\end{align*}
		Similarly as in~\cite[eq. (13.3)]{bfi} we have
		\begin{align*}
		\|b_{\bm{n},\bm{r},\bm{s}}\|_2^2\ll X^{\epsilon'}N^2HR^{-2}.
		\end{align*}
		Plugging in our choice of variables and using that we have $\nu\leq X^{\epsilon'}$ and that by the assumption $Q^2R\leq X$ we have $Q\leq \sqrt{X}$,  we get
		\begin{align*}
		W_3&\ll X^{O(\epsilon')}\Bigl(Q^2 N^4+Q^3 N^{5/2} \Bigr)^{1/2} NQ (RM)^{-1/2}\\
		&\ll X^{O(\epsilon')} \frac{N^2 Q^2}{R} \frac{R^{1/2}}{NQM^{1/2}} \Bigl(Q N^2+Q^{3/2} N^{5/4} \Bigr) \\
		&= \frac{N^2 Q^2}{R} X^{1/2+O(\epsilon')} \Bigl( N^{3/2}R^{1/2}+Q^{1/2} R^{1/2} N^{3/4} \Bigr).
		\end{align*}
		By~\eqref{eq_lem_T5_1} and~\eqref{eq_lem_T5_2} this is sufficient for~\eqref{eq_lem_T5_3}.
		
		By Lemmas~\ref{lem_initred},~\ref{lem_congtodisp} and~\ref{lem_redtokloo}, this proves case \ref{(T.5)} of Proposition~\ref{theorem_BFI_rough}.
	\end{proof}

		\section{An application -- Proof of Theorem~\ref{MT3} }
		
		We now prove Theorem~\ref{MT3} that extends~\cite[Theorem 1.1]{MaynardII}. This extension is possible as our application of the  dispersion method does not require the coefficients after the combinatorial dissection to be rough, and we work with $\mathfrak{u}_P$ that takes into account the contribution of all low conductor characters.
		
		We shall prove the following proposition, which directly implies Theorem~\ref{MT3}.
		\begin{prop}\label{MT3b}
Let $k\geq 1$, $a\in \mathbb{Z}\setminus \{0\}$ and $\varepsilon>0$ be fixed, with $\varepsilon$ sufficiently small. Let $N\geq 2$ and $P\leq N^{\varepsilon}$. Let $|\lambda_d|\leq \tau_k(d)$ be any triply well-factorable sequence. 
Then it holds that
\begin{align}\label{eq_theorem3}
	\sum_{\substack{d\leq N^{3/5-\varepsilon} \\(d,a)=1}}\lambda_d\left(\sum_{\substack{n\leq N}}\frac{\mu(n)}{\varphi(d)}\sum_{\substack{\psi(d)\\ \textup{cond}(\psi)>P}}\psi(\overline{a}n)\right) \ll  NP^{-1/200}.
	\end{align}
	\end{prop}
	
	Applying this with $d\mapsto \lambda_d\rho(d,P)$ in place of $d\mapsto \lambda_d$ gives Theorem~\ref{MT3}.
		
	With the notation~\eqref{eq_uPdef}, the statement~\eqref{eq_theorem3} of Proposition~\ref{MT3b} that we want to show becomes 
		\begin{align}\label{eq_MT3reduced}
			\sum_{\substack{d\leq N^{3/5-\epsilon}\\(d,a)=1}}\lambda_d\left(\sum_{\substack{n\leq N}}\mu(n)\mathfrak{u}_P(n\overline{a};d)\right) \ll  NP^{-1/200}.
		\end{align}
	  We now prove two variants of Maynard's central results that quickly lead to this using the methods of the previous sections.
				 
		\begin{lemma}[Variant of Proposition 5.1 in~\cite{MaynardII}]
			Let $C\geq 1$ be fixed. Assume Convention~\ref{def_BFIrough_assumptions}  and let $\lambda_q$ be any triply well-factorable sequence of order $C$ and level $Q\leq X^{3/5-\epsilon}$, and assume 
			\begin{align*}
			 N \leq X^{2/5}.
			\end{align*}
		Then for any $P\leq X^{\epsilon'}$ we have
		\begin{align*}
			\sum_{q\leq Q}\lambda_q \sum_{\substack{n,m \\ }}\alpha_m \beta_n \mathfrak{u}_P(mn\overline{a}:q)\ll X (\log X)^{O_C(1)}P^{-1/13}.
		\end{align*}
			\end{lemma}
	
		 \begin{proof}
		 	This follows from the case \ref{(T.2)} of Proposition~\ref{theorem_BFI_rough} by using triply well-factorability with $Q=RQ_1Q_2$, where (as in~\cite[Proof of Proposition 5.1]{MaynardII})
		 	\begin{align*}
		 		R&=NX^{-\epsilon}\\
		 		Q_1&=QX^{-2/5+\epsilon}\\
		 		Q_2&=N^{-1}X^{2/5 }.
		 	\end{align*}
		 	\end{proof}

		\begin{lemma}[Variant of Proposition 5.2 in~\cite{MaynardII}]
Let $N_1,N_2\geq X^\epsilon$ and $N_1N_2M\asymp X$ and \begin{align*}
	Q\leq \left(\frac{X}{M}\right)^{2/3-\epsilon}. 
\end{align*}			
Let $\alpha_m$ be a sequence with $|\alpha_m|\leq m^{\epsilon'}$. Then for $P\leq X^{\epsilon'}$
\begin{align*}
	\sum_{q\sim Q}\Bigg|\sum_{\substack{n_1\sim N_1, n_2\sim N_2\\ m\sim M \\ n_1n_2m\leq X}}\alpha_m \mathfrak{u}_P(mn_1n_2\overline{a};q) \Bigg|\ll X^{1-\epsilon'}.
\end{align*}
		\end{lemma}
	
	\begin{proof}
		By a sufficiently fine dissection of the range of the variables  $n_1,n_2,m$ and the introduction of smooth weights $f_{N_1},f_{N_2}$ as in  Lemma~\ref{lem_initred}(4), it suffices to show
		\begin{align*}
			\mathscr{K}\ll \frac{X^{1-\epsilon/2}}{Q},
		\end{align*}
	where
	\begin{align*}
		\mathscr{K}=\sup_{\substack{(a,q)=1\\ q\sim Q}}\left|\sum_{m\sim M}\alpha_m \sum_{n_1,n_2}f_{N_1}(n/N_1)f_{N_2}(n/N_2) \mathfrak{u}_P(mn_1n_2\overline{a};q)\right|.
	\end{align*}
Compare~\cite[Lemma 7.1]{MaynardII}. Note that the derivatives of the smooth weights there are slightly smaller. This only changes the error term in the application of Poisson summation in an inconsequential manner. Let the supremum occur at $a$ and $q$ and write $\mathscr{K}=|\mathscr{K}_2-\mathscr{K}_1|$ with
\begin{align*}
	\mathscr{K}_1&\coloneqq \frac{1}{\varphi(q)}\sum_{\substack{\psi(q)\\\text{cond}(\psi)\leq P}}\sum_{m\sim M}\alpha_m \sum_{n_1,n_2}f_{N_1}(n/N_1)f_{N_2}(n/N_2) \psi(mn_1n_2\overline{a})\\
	\mathscr{K}_2&\coloneqq \sum_{m\sim M}\alpha_m \sum_{n_1,n_2}f_{N_1}(n/N_1)f_{N_2}(n/N_2)1_{mn_1n_2\equiv a(q)}.
\end{align*}
Here $\mathscr{K}_2$ is the same Maynard considers and he shows, assuming $N_1\leq N_2$, \begin{align*}
\mathscr{K}_2=\mathscr{K}_{\text{MT}}+O\Bigl(\frac{X^{1+o(1)}}{QN_1}+X^{o(1)}MQ^{1/2}\Bigr),
\end{align*}
where
\begin{align*}
	\mathscr{K}_{\text{MT}}\coloneqq N_1N_2 \widehat{f_{N_1}}(0)\widehat{f_{N_2}}(0)\frac{\varphi(q)^2}{q^2}\sum_{\substack{m\sim M \\ (m,q)=1}}\alpha_m.
\end{align*}

This main term can be extracted out of $\mathscr{K}_1$ similarly as in the proof of Lemma~\ref{lem_T4} (case \ref{(T.4)}). By Poisson summation (Lemma~\ref{lem_Poiss}) we have any character $\psi$ with modulus $q$ that
\begin{align*}
	\sum_{n_1}f_{N_1}(n/N_1)\psi(n_1)=\frac{N_1}{q}\sum_{|h|\leq \frac{qX^{\epsilon'}}{N_1}}\widehat{f_{N_1}}\bigl(\frac{hn_1}{q}\bigr)\sum_{\substack{c(q)}}\psi(c)e\bigl(\frac{-ch}{q}\bigr)+O_{\epsilon'}(X^{-100}).
\end{align*}
For $h=0$, and $\psi$ being the principal character the sum over $c$ equals $\varphi(q)$. In all other cases, we can estimate it by $O(\sqrt{P}(h,q))$, by Lemma \ref{lem_ccalc}. Thus, 
\begin{align*}
	\mathscr{K}_1=\mathscr{K}_{\text{MT}}+O\left(\frac{X^{1+2\epsilon'} P^{3/2}}{QN_1}\right).
\end{align*}
This gives the lemma by the condition $N_i\geq X^{\epsilon}$, $Q\leq (X/M)^{2/3-\epsilon}$.
	\end{proof}
	
\begin{proof}[Proof of Proposition~\ref{MT3b}]
	As we have the same arithmetic information available, the proof can be done with Heath-Brown's identity as in~\cite[Proof of Theorem 1.1]{MaynardII}. The only difference is the necessity for a finer-than-dyadic dissection and a trivial bound for ranges not covered completely, for which we gave the details in the proof of Proposition~\ref{thm_bfi2}.
\end{proof}

\part{The Montgomery--Vaughan result with sieve weights}
\label{Part:IV}
\section{Proof of Key Proposition~\ref{prop2}}
Our goal in this section is to prove Key Proposition~\ref{prop2} assuming a few lemmas (Lemmas~\ref{lem_Fmult},~\ref{le_Fsimple},~\ref{lem_SS1},~\ref{lem_Feval} and~\ref{lem_SS2}) relating to singular series and ``modified Gau\ss{} sums'' that will be proved in the next section. 

\subsection{Initial reduction}

We start by noting that it suffices to consider only
\begin{align}\label{eq_nmini}
m\geq NP^{-c_1}.
\end{align}
We furthermore can replace $\Lambda$ in the definition of $\nu_i$ by 
\begin{align}\label{eq_L_0def}
1_{n\in \mathbb{P}}\log n=:\Lambda_0(n),
\end{align}
as the contribution of prime powers to $\nu_1*\nu_2(m)$ can be trivially bounded by $\ll_\varepsilon N^{1/2+\varepsilon}$.

To prove Key Proposition~\ref{prop2}, we apply the circle method to the convolution 
$$\nu_1*\nu_2(m)=\int_{0}^1\widehat{\nu_1}\widehat{\nu_2}(\alpha)e(-\alpha m) \,d\alpha.$$ 
We use the same major and minor arc splitting as in Section~\ref{sec: prop1a}.
The minor arcs can be handled just as in Section~\ref{sub: minor} (taking $\lambda_2(d)=1_{d=1}$ there), so we conclude that for all $m\leq N$ apart from $O(NP^{-c_1})$ exceptions we have
\begin{align*}
\nu_1*\nu_2(m)=\int_{\mathfrak{M}} \widehat{\nu_1}\widehat{\nu_2}(\alpha)e(-\alpha m) \,d\alpha+O(NP^{-c_1}).
\end{align*}

In contrast to the proof of Propositions~\ref{prop1a} and~\ref{prop1b}, we need to asymptotically estimate the summands over the moduli $q$ in the expression of the major arc integral as
\begin{align*}
\int_{\mathfrak{M}} \widehat{\nu_1}\widehat{\nu_2}(\alpha)e(-\alpha m)=\sum_{q\leq P^{c_0}}\sum_{a(q)^{*}}e_q(-am) \int_{|\eta|\leq Q^{-1}}\widehat{\nu_1}\widehat{\nu_2}(a/q+\eta) e(-\eta m) \,d\eta.
\end{align*}
By the conditions of Key Proposition~\ref{prop2}, we may write
\begin{align*}
\nu_i(n)=\Lambda_0(n)\sum_{d\mid n+2}\lambda_i(d),
\end{align*}
where $\lambda_i(d)$ is the product of the two sieve weights in the definition of the admissible pre-sieve $\nu_i$, i.e. one weight $\lambda_i^{\dagger}$ for the primes $3\leq p\leq P$, $p\nmid \widetilde{r}$ and the sieve of Eratosthenes weight $\mu$ for the primes $p\geq 3, p\mid\widetilde{r}$. We recall that $\widetilde{r}$ denotes the exceptional modulus, if it exists, and $\widetilde{r}=1$ otherwise. 

Recall that by Definition~\ref{def_PtildPdag} we may write the just mentioned sifting ranges as
\begin{align*}
\widetilde{\mathcal{P}}=\prod_{\substack{p\geq 3  \\p\mid \widetilde{r}}}p,\quad \mathcal{P}^{\dagger}\coloneqq \prod_{\substack{3 \leq  p\leq P\\p\nmid \widetilde{r}}}p.    
\end{align*}
We now analogously for $n\geq 1$ write
\begin{align}\label{e22}
 \widetilde{n}=\prod_{\substack{p\mid n\\p\in \widetilde{\mathcal{P}}}}p^{v_p(n)},\quad n^{\dagger}=\prod_{\substack{p\mid n\\p\in \mathcal{P}^{\dagger}}}p^{v_p(n)}.   
\end{align}
Note that this only gives a factorisation $n=\widetilde{n}n^{\dagger}$ if every prime divisor of $n$ divides $\widetilde{\mathcal{P}}\mathcal{P}^{\dagger}$ (which happens if and only if $n$ is $P$ smooth and odd). 

We need furthermore what we call \emph{modified Gau\ss{} sums}, namely given a character $\chi$ to the modulus $q$, we set
\begin{align}\label{e37}
c_\chi(a,j)\coloneqq \asum_{\substack{b\pmod q\\ (b+2,\textnormal{rad}(q))=(j,\textnormal{rad}(q))}}\chi(b)e_q(ab),\,\,\textnormal{ where }\,\, \textnormal{rad}(q)\coloneqq \prod_{p|q}p.
\end{align}
We will also encounter the sieve-theoretic sums
\begin{align} \label{eq_Sdef}
S_i(q,j,e)\coloneqq \frac{1}{\varphi(e)}\sum_{\substack{c,d \\ c\mid j,(d,q e)=1 }}\frac{\lambda_i(cde)}{\varphi(d)}
\end{align}
and their restriction to divisors of $\mathcal{P}^\dagger$ given by
\begin{align*}
S_i^\dagger(q,j,e)\coloneqq \frac{1}{\varphi(e^\dagger)}\sum_{\substack{c,d |\mathcal{P}^\dagger \\ c\mid j ,(d,q e)=1 }}\frac{\lambda_i^\dagger(cde^{\dagger})}{\varphi(d)}.
\end{align*}
Note that $S_i^\dagger$ is invariant under replacing any of its arguments by their $^\dagger$ components. Both modified Gau\ss{} sums and the sieve-related sums will be studied in more detail later, in Subsections~\ref{sec: gauss} and~\ref{sec: sieve}, respectively.

Finally, define a modified Euler function that accounts for the fact that the sifted primes only occupy $p-2$ residue classes modulo a given prime $p>2$, as
\begin{align*}
\varphi_2(n)=n\prod_{p\mid n}(1-2/p)
\end{align*}
and 
\begin{align*}
\varphi_2(2^s)=\varphi(2^s).
\end{align*}
With this notation, we are ready to state a technical decomposition of $\widehat{\nu}_i(\alpha)$ on the major arcs.

\begin{lemma}[A splitting of the major arc contribution] \label{le_majorarcs} Let $1\leq a\leq q\leq P^{c_0}$, $(a,q)=1$. We have
\begin{align*}
&\widehat{\nu_i}(a/q+\eta)\\
&= \frac{V(\widetilde{\mathcal{P}})}{\varphi_2(\widetilde{q})\varphi(q/\widetilde{q})}\sum_{\substack{j^{\dagger}|\textnormal{rad}(q^\dagger) \\ l|q}}\sum_{\substack{e|\mathcal{P}^{\dagger}\widetilde{\mathcal{P}}\\ (e,q)=1\\ el\leq P^{2c_0}}}\frac{\mu(\widetilde{e})S_i^\dagger(q,j,e)}{\varphi_2(\widetilde{e})} \sum_{\substack{\psi(e)^*\\ \chi(l)^*}}\overline{\psi}(-2)c_{\overline{\chi} \chi^{(q)}_{0}}(a,j^{\dagger}) \\
&\times \sum_{n\leq N} \Lambda_0(n)\chi\psi(n)e(\eta n)+O(NP^{-c_0/10}),
\end{align*}
where $\chi^{(q)}_{0}$ denotes the principal character $\pmod q$.
\end{lemma}

\begin{proof}
Clearly it suffices to consider the case $i=1$. 
We begin by translating $\widehat{\nu_1}(a/q+\eta)$ to the language of Dirichlet characters. For $(n,q)=1$, we have the Fourier expansion
\begin{align*}
e_q(an)&=\frac{1}{\varphi(q)}\sum_{\substack{b(q)\\ \chi(q)}} e_q(ab)\overline{\chi}(b) \chi(n). 
\end{align*}
By orthogonality relations, for $n,d$ odd we have
\begin{align*}
1_{d|n+2}=\frac{1}{\varphi(d)}\sum_{\psi(d)}\overline{\psi}(-2)\psi(n).
\end{align*}
Recalling that we are only sieving for odd primes, the above together result in 
\begin{align*}
\widehat{\nu_1}(a/q+\eta)&=\frac{1}{\varphi(q)}\sum_{d}\frac{\lambda_1(d)}{\varphi(d)}  \sum_{\substack{\psi(d) \\ \substack{ \chi (q)\\ b(q)}}} \overline{\psi}(-2) e_q(ab) \overline{\chi}(b) \sum_{n\leq N} \Lambda_0(n)\chi\psi(n)e(\eta n)+O(qN^{c}),
\end{align*}
where the error term accounts for those $n$ that are not coprime to $2q$ and $c=1/100$, say. 

Note that $d$ is always odd and squarefree in the support of $\lambda_1(d)$, and factorise $d=d_1d_2$, where $d_1\mid q$ and $(d_2,q)=1$. Split the sum over the characters $\psi$ accordingly to reach
\begin{align*}
&\widehat{\nu_1}(a/q+\eta)\\
&=\frac{1}{\varphi(q)}\sum_{\substack{d_1, d_2\\ d_1|q, (d_2,q)=1}}\frac{\lambda_1(d_1 d_2)}{\varphi(d_1)\varphi(d_2)}\sum_{\substack{\psi_1(d_1)\\\psi_2(d_2)}} \sum_{\substack{\chi(q)\\ b(q)}}\overline{\psi_1\psi_2}(-2)\overline{\chi}(b)e_q(ab)\sum_{n\leq N} \Lambda_0(n)\chi\psi_1\psi_2(n)e(\eta n)\\
&+O(qN^{c}).
\end{align*}

To simplify, note that for any character $\chi'\pmod q$ we have
\begin{align*}
\sum_{\substack{\psi_1(d_1), \chi (q) \\ \psi_1 \chi= \chi'}} \overline{\psi_1}(-2) \overline{\chi}(b) &= \overline{\chi'}(b) \sum_{\psi_1(d_1)} \overline{\psi_1}(-2)\psi_1(b) \\
&=\overline{\chi'}(b) \varphi(d_1) 1_{d_1|b+2}.
\end{align*}
Summing over different $\chi'\pmod q$ gives us
\begin{align*}
&\widehat{\nu_1}(a/q+\eta)\\
&=\frac{1}{\varphi(q)}\sum_{d_1\mid q, (d_2,q)=1}\frac{\lambda_1(d_1 d_2)}{\varphi(d_2)}\sum_{\substack{\psi_2(d_2)\\ \substack{\chi(q)\\ b(q)}}}\overline{\psi_2}(-2)\overline{\chi}(b)1_{d_1|b+2}e_q(ab)\sum_{n\leq N} \Lambda_0(n)\chi\psi_2(n)e(\eta n)\\
&+O(qN^{c}).
\end{align*}

Each pair of characters $\psi_2$ and $\chi$ in above summation is induced by a unique pair of primitive characters to some moduli $e|d_2$ and $l|q$, respectively. We sort the characters based on those, write $d_2=de$ (with $d,e$ necessarily coprime) and  trivially estimate the error from replacing $\psi_2,\chi$ with their primitive parts in the $n$ sum to obtain
\begin{align*}
&\widehat{\nu_1}(a/q+\eta)\\
&=\frac{1}{\varphi(q)}\sum_{l|q}\sum_{\substack{e\\(e,q)=1}}\sum_{d_1\mid q,(d,qe)=1}\frac{\lambda_1(d_1 d e)}{\varphi(e)\varphi(d)} \sum_{\substack{\psi(e)^*\\ \chi(l)^*}}\overline{\psi}(-2)\sum_{b(q)}\overline{\chi}\chi^{(q)}_{0}(b)1_{d_1|b+2}e_q(ab) \\
&\times\sum_{n\leq N} \Lambda_0(n)\chi\psi(n)e(\eta n)+O(qN^{c}).
\end{align*}

As $\lambda_1$ is supported on squarefree integers only, for each fixed $b$ the summation condition for $d_1$, that is $d_1|q, d_1|b+2$, can be replaced by $d_1|(b+2,\textnormal{rad}(q))$. Writing $j=(b+2,\textnormal{rad}(q))$ and sorting according to $j|\textnormal{rad}(q)$ gives us
\begin{align*}
&\widehat{\nu_1}(a/q+\eta)\\
&=\frac{1}{\varphi(q)}\sum_{\substack{j|\textnormal{rad}(q)\\l |q}}\sum_{\substack{e\\(e,q)=1}}\sum_{d_1\mid j,(d_2,q)=1}\frac{\lambda_1(d_1 d e)}{\varphi(e)\varphi(d)} \sum_{\substack{\psi(e)^*\\ \chi(l)^*}}\overline{\psi}(-2)\sum_{\substack{b(q)\\(b+2,\textnormal{rad}(q))=j}}\overline{\chi}\chi^{(q)}_{0}(b)e_q(ab) \\
&\times \sum_{n\leq N} \Lambda_0(n)\chi\psi(n)e(\eta n)+O(qDN^{c})\\
&=\frac{1}{\varphi(q)}\sum_{\substack{j|\textnormal{rad}(q)\\l |q}}\sum_{\substack{e\\(e,q)=1}} S_1(q,j,e) \sum_{\substack{\psi(e)^*\\ \chi(l)^*}}\overline{\psi}(-2) c_{\overline{\chi}\chi^{(q)}_{0}}(a,j) \sum_{n\leq N} \Lambda_0(n)\chi\psi(n)e(\eta n)+O(qN^{c}),
\end{align*}
with the notation from~\eqref{e37} and~\eqref{eq_Sdef}.

We next bound the contribution of terms with $el>P^{2c_0}$. By the choice of sieves, $S(q,j,e)$ vanishes if $e>D_0P^{2c_0}$. By trivial estimates and combining $\chi,\psi$ into one character, we have
\begin{align*}
&\frac{1}{\varphi(q)}\sum_{\substack{j|\textnormal{rad}(q)\\l |q}}\sum_{\substack{e\\(e,q)=1\\ el>P^{2c_0}}} S_1(q,j,e) \sum_{\substack{\psi(e)^*\\ \chi(l)^*}}\overline{\psi}(-2) c_{\overline{\chi}\chi^{(q)}_{0}}(a,j) \sum_{n\leq N} \Lambda_0(n)\chi\psi(n)e(\eta n)\\
&\ll_{\varepsilon} q^{1/2+\varepsilon}\sum_{\substack{P^{2c_0}\leq e \leq D_0P^{3c_0}\\ \chi(e)^*}}\frac{1}{\varphi(e)}\Big|\sum_{n\leq N} \Lambda_0(n)\chi(n)e(\eta n)\Big|.
\end{align*}
After reinserting again higher prime powers and dealing with the exponential phase as in $\eqref{e8}$, we can apply Lemma~\ref{lem_BDHrough} and Cauchy--Schwarz to give the upper bound 
\begin{align*}
Nq^{1/2+\varepsilon}(\log N)^{O(1)}P^{-2c_0/3}\ll N P^{-c_0/10}.
\end{align*}

To finish the proof of the lemma, it remains to separate the contribution of the two sieves in the definition of $\lambda_i$. The sieve weights $\lambda_i$ are composed of two sieve weights, $\lambda_1^{\dagger}$ and $\mu$, with the ranges $\mathcal{P}^{\dagger}$ and $\widetilde{\mathcal{P}}$ as given in Definition~\eqref{def_PtildPdag}. Note that $e$ must be odd and squarefree whenever $S_1(q,j,e)\neq 0$, and write $e=e^\dagger \widetilde{e}$, where $e^\dagger, \widetilde{e}$ are as in~\eqref{e22}, and similarly for $j$ and all other variables that occur. For $j\mid \textnormal{rad}(q)$, we get
\begin{align*}
S_1(q,j,e)=\frac{1}{\varphi(e^\dagger)\varphi(\widetilde{e})}\sum_{\substack{d_1^\dagger, d^\dagger |\mathcal{P} \\ d_1^\dagger|j^\dagger , (d^\dagger,e^\dagger q)=1}}\frac{\lambda_1^\dagger(d_1^\dagger e^\dagger d^\dagger)}{\varphi(d^\dagger)}\sum_{\substack{\widetilde{d}_1,\widetilde{d}|\widetilde{\mathcal{P}}\\ \widetilde{d}_1|\widetilde{j}, (\widetilde{d},\widetilde{e}q)=1}} \frac{\mu(\widetilde{d}_1 \widetilde{e}\widetilde{d})}{\varphi(\widetilde{d})}.
\end{align*}
The sum over $\widetilde{d}_1$ vanishes unless $\widetilde{j}=1$. If $\widetilde{j}=1$, then
\begin{align*}
\frac{1}{\varphi(\widetilde{e})}\sum_{\substack{\widetilde{d}_1,\widetilde{d}|\widetilde{\mathcal{P}}\\ \widetilde{d}_1|\widetilde{j}, (\widetilde{d},\widetilde{e}q)=1}} \frac{\mu(\widetilde{d}_1 \widetilde{e}\widetilde{d})}{\varphi(\widetilde{d})}=\frac{\mu(\widetilde{e})}{\varphi(\widetilde{e})}\prod_{\substack{p\mid \widetilde{\mathcal{P}}\\p\nmid \widetilde{e}q}}\left(1-\frac{1}{p-1}\right)=\frac{\mu(\widetilde{e})\varphi(\widetilde{q})V(\widetilde{\mathcal{P}})}{\varphi_2(\widetilde{q})\varphi_2(\widetilde{e})}.
\end{align*}

Having simplified $S_1(q,j,e)$, we return to studying $\widehat{\nu_1}$. We arrive at
\begin{align*}
&\widehat{\nu_1}(a/q+\eta)\\
&= \frac{V(\widetilde{\mathcal{P}})}{\varphi_2(\widetilde{q})\varphi(q/\widetilde{q})}\sum_{\substack{j^{\dagger}|\textnormal{rad}(q^\dagger)\\ l|q}}\sum_{\substack{e|\mathcal{P}^{\dagger}\widetilde{\mathcal{P}}\\ (e,q)=1\\ el\leq P^{2c_0}}}\frac{\mu(\widetilde{e})S_1^\dagger(q^\dagger,j^{\dagger},e^\dagger)}{\varphi_2(\widetilde{e})} \sum_{\substack{\psi(e)^*\\ \chi(l)^*}}\overline{\psi}(-2)c_{\overline{\chi} \chi^{(q)}_{0}}(a,j^{\dagger}) \\
&\times \sum_{n\leq N} \Lambda_0(n)\chi\psi(n)e(\eta n)+O( N P^{-c_0/10}),
\end{align*}
which we can lastly simplify by recalling that $S_1^\dagger(q^\dagger,j^{\dagger},e^\dagger)=S_1^\dagger(q,j,e)$.
\end{proof}

We expect the contribution of $\sum_{n\leq N} \Lambda_0(n)\chi\psi(n)e(\eta n)$ in Lemma~\ref{le_majorarcs} to be small for $|\eta|\leq Q^{-1}$ for any given pair of primitive characters, except when $\chi\psi=\chi_{0}^{(1)}$ is the trivial character or when $\chi\psi=\widetilde{\chi}$ is the exceptional character. We define the exponential sums
\begin{align*}
T(\eta)&\coloneqq\sum_{n\leq N}e(\eta n) \\
\widetilde{T}(\eta)&\coloneqq -\sum_{n\leq N}n^{\widetilde{\beta}-1}e(\eta n)\\
W(\chi\psi,\eta)&\coloneqq \sum_{n\leq N} \Lambda_0(n)\chi\psi(n)e(\eta n)-1_{\chi\psi=\chi_{0}^{(1)}}T(\eta)-1_{\chi\psi=\widetilde{\chi}} \widetilde{T}(\eta)
\end{align*}
corresponding to the expected contribution of the trivial character, the Siegel character, and the remainder to $\sum_{n\leq N} \Lambda_0(n)\chi\psi(n)e(\eta n)$. 

In the decomposition given by Lemma~\ref{le_majorarcs}, the cases $\chi\psi=\chi^{(1)}_{0}$ and $\chi\psi=\widetilde{\chi}$ can respectively only happen if $e=l=1$ and $el=\widetilde{r}$, since $(e,l)=1$ and the $\psi$, $\chi$ are primitive. Let $2^t||\widetilde{r}$. As the exceptional character is quadratic and primitive, we note that $t\in \{0,2,3\}$ and that  $\widetilde{r}/2^t=\widetilde{\mathcal{P}}$. As $e$ is odd and $(e,q)=1$, the statement $el=\widetilde{r}$ is equivalent to $l=2^t \textnormal{rad}(\widetilde{q})$ and $e=\widetilde{r}/(2^t  \textnormal{rad}(\widetilde{q}))$. In particular, this can happen only if $2^t|q$ and $e^\dagger=1$. These considerations motivate the following definitions. Write
\begin{align}\label{eq_q'def}
q'\coloneqq 2^t \textnormal{rad}(\widetilde{q})
\end{align} 
and set
\begin{align*}
\nu_1^T(a,q,\eta)&\coloneqq \frac{V(\widetilde{\mathcal{P}})}{\varphi_2(\widetilde{q})\varphi(q/\widetilde{q})}\sum_{\substack{j|q^\dagger \\ }} S_1^\dagger(q,j,1) c_{\chi^{(q)}_{0}}(a,j)T(\eta), \\
\nu_1^{\widetilde{T}}(a,q,\eta)&\coloneqq 1_{2^t|q}\frac{\overline{\widetilde{\chi}^{(\widetilde{r}/q')}}(-2)\mu(\widetilde{r}/q')V(\widetilde{\mathcal{P}})}{\varphi_2(\widetilde{q})\varphi_2(\widetilde{r}/q')\varphi(q/\widetilde{q})}\sum_{\substack{j|q^\dagger \\ }} S_1^\dagger(q,j,1) c_{\overline{\widetilde{\chi}^{(q')}\chi^{(q)}_{0}}}(a,j)\widetilde{T}(\eta),
\end{align*}
where, following the notation given in~\eqref{e23}, we denote by $\widetilde{\chi}^{(\widetilde{r}/q')}$ and $\widetilde{\chi}^{(q')}$ the $\widetilde{r}/q'$ and $q'$ component of the exceptional character, respectively. We define further the part of $\nu_1$ that remains after subtracting these terms as
\begin{align*}
\nu_1^W(a,q,\eta)&\coloneqq \widehat{\nu_1}(a/q+\eta)-\nu_1^T(a,q,\eta)-\nu_1^{\widetilde{T}}(a,q,\eta)\\
&=\frac{V(\widetilde{\mathcal{P}})}{\varphi_2(\widetilde{q})\varphi(q/\widetilde{q})}\sum_{\substack{j|q^\dagger \\ l|q}}\sum_{\substack{e|\mathcal{P}^{\dagger}\widetilde{\mathcal{P}}\\ (e,q)=1\\el\leq P^{2c_0}}}\frac{\mu(\widetilde{e})S_1^\dagger(q,j,e)}{\varphi_2(\widetilde{e})} \sum_{\substack{\psi(e)^*\\ \chi(l)^*}}\overline{\psi}(-2)c_{\overline{\chi} \chi^{(q)}_{0}}(a,j)W(\chi\psi,\eta).
\end{align*}

We use this decomposition and Lemmas~\ref{le_fourier},~\ref{le_majorarcs} to get for $m\leq N$ outside a set of size $O(NP^{-c_1})$ that 
\begin{align}
\nonumber &\int_{\mathfrak{M}} \widehat{\nu_1}\widehat{\nu_2}(\alpha)e(-\alpha m)\, d\alpha\\
\begin{split}
&=\sum_{\substack{q\leq P^{c_0}\\ a(q)^*}}e_q(-am) \int_{|\eta|\leq Q^{-1}}\bigl(\nu_1^T(a,q,\eta)+\nu_1^{\widetilde{T}}(a,q,\eta)+\nu_1^W(a,q,\eta)\bigr)\\
&\times \bigl(\nu_2^T(a,q,\eta)+\nu_2^{\widetilde{T}}(a,q,\eta)+\nu_2^W(a,q,\eta)\bigr)e(-\eta m)\,d\eta+O(NP^{-c_1}). \label{eq_initredend}
\end{split}
\end{align}

We split this up into two parts, the main term that consists out of all those pairs that do not contain the superscript $W$, and the remaining error term. It will make matters easier if we enlarge the range of $q$ for the main terms, in the case the exceptional zero exists, from $\{q: q\leq P^{c_0}\}$ to $\{q=2^s\widetilde{q}q^\dagger: \widetilde{q}|\widetilde{r}, 2^s q^\dagger \leq P^{c_0}\}=:\mathcal{Q}$, say. To justify this, note that $c_{\chi_0^{(q)}}(a,j)$ respectively $c_{\widetilde{\chi}^{(q')}\chi_0^{(q)}}(a,j)$ vanish if $\widetilde{q}$ has a square prime divisor and $(j,\widetilde{q})=1$. Furthermore in the added non-overlapping intervals we have $q>P^{c_0}$ and the estimate
\begin{align*}
\max_{q\in \mathcal{Q}, q>P^{c_0}}\{|\nu_i^T(a,q,\eta)|,|\nu_2^{\widetilde{T}}(a,q,\eta)|\}\leq N (\log N)^{O(1)}P^{-c_0/3}
\end{align*}
that follows from the trivial estimates $|T(\eta)|,|\widetilde{T}(\eta)|\leq N$. Thus, by Lemma~\ref{le_fourier} for all $m$ outside a sufficiently small exceptional set, this larger range main term approximates the previous one with an admissible error term.  In the same way also the error term of Lemma~\ref{le_majorarcs} can be handled.

Now, to finish the proof of Key Proposition~\ref{prop2}, it suffices to show that there exist functions $\mathcal{M}(m)$ and $\mathcal{E}(m)$ as proposed, such that
\begin{align}\label{e34}\begin{split}
 &\sum_{\substack{q\in \mathcal{Q}\\ a(q)^*}}e_q(-am) \int_{|\eta|\leq Q^{-1}}\bigl(\nu_1^T(a,q,\eta)+\nu_1^{\widetilde{T}}(a,q,\eta)\bigr)\cdot \bigl(\nu_2^T(a,q,\eta)+\nu_2^{\widetilde{T}}(a,q,\eta)\bigr)e(-\eta m)\,d\eta\\
 &= mV(P)^2 \mathfrak{S}(m) \left(\mathcal{M}(m)\Bigl(1+O\bigl(e^{- c\frac{\log D_0}{\log P}}\bigr)\Bigr)+O\bigl(e^{100\sqrt{\log N}}P^{-c}\bigr)\right)
 \end{split}
\end{align}
and for $i\in \{1,2\}$  
\begin{align}
\nonumber &\Big|\sum_{\substack{q\leq P^{c_0}\\ a(q)^*}}e_q(-am) \int_{|\eta|\leq Q^{-1}}\nu_i^W(a,q,\eta)
\cdot \bigl(\nu_{i+1}^T(a,q,\eta)+\nu_{i+1}^{\widetilde{T}}(a,q,\eta)+\nu_{i+1}^W(a,q,\eta)\bigr)e(-\eta m)\,d\eta\Big|\\
&\ll mV(P)^2 \mathfrak{S}(m)e^{-c\frac{\log N}{\log P}}\mathcal{E}(m),   \label{e35} 
\end{align}
where the index $i+1$ is taken $\pmod 2$.

\subsection{Main term} \label{sec:majarcMT}

In this subsection, we give the desired asymptotic evaluation~\eqref{e34} of the contribution of products of all terms in~\eqref{eq_initredend} that do not contain a superscript $W$. To do this, we first need to introduce some auxiliary quantities and some lemmas about them whose proofs are postponed to the next section.

Let $\chi_1$ and $\chi_2$ be characters to the modulus $q$ and define
\begin{align}\label{eq_Fdef}
F(\chi_1,\chi_2,j_1,j_2,m)\coloneqq \asum_{a(q)}c_{\chi_1}(a,j_1)c_{\chi_2}(a,j_2)e_q(-am).
\end{align}

We recall the notation in~\eqref{e23}. The following lemmas tell us that $F$ is multiplicative and that it vanishes if the modulus has a square divisor that is not shared by both conductors of $\chi_1$, $\chi_2$.

\begin{lemma}[$F$ is multiplicative]\label{lem_Fmult}
Let $\chi_1$ and $\chi_2$ be characters to the modulus $q$ and factorise $\chi_i$ as in~\eqref{e23}. Then we have 
\begin{align*}
F(\chi_1,\chi_2,j_1,j_2,m)=\prod_{p\mid q}F(\chi^{(p^{\alpha_1(p)})}_{1},\chi^{(p^{\alpha_2(p)})}_{2},j_{1},j_{2},m).
\end{align*}
\end{lemma}

\begin{lemma}\label{le_Fsimple}
 If $p^{\alpha}\mid \mid q$ for some $\alpha>1$, we have 
 \begin{align}\label{eq_Festi1}
 F(\chi_1,\chi_2,j_1,j_2,m)=0\,\, \textnormal{unless}\,\, \alpha_1^*(p)=\alpha_2^*(p)=\alpha.
\end{align}
\end{lemma}

Now, with $F$ as in~\eqref{eq_Fdef}, we can write
\begin{align*}
&\sum_{a(q)^*}\nu_1^T(a,q,\eta)\nu_2^T(a,q,\eta)e_q(-am)\\
&= \frac{V(\widetilde{\mathcal{P}})^2 T(\eta)^2}{\varphi_2(\widetilde{q})^2\varphi(q/\widetilde{q})^2}\sum_{j_i|q^\dagger}S_1^\dagger(q,j_1,1)S_2^\dagger(q,j_2,1)\sum_{a(q)^*}c_{\chi^{(q)}_{0}}(a,j_1)c_{\chi^{(q)}_{0}}(a,j_2)e_q(-am)\\
&=\frac{V(\widetilde{\mathcal{P}})^2 T(\eta)^2}{\varphi_2(\widetilde{q})^2\varphi(q/\widetilde{q})^2}\sum_{j_i|q^\dagger}S_1^\dagger(q,j_1,1)S_2^\dagger(q,j_2,1)F(\chi^{(q)}_{0},\chi^{(q)}_{0},j_1,j_2,m).
\end{align*}
Lemma~\ref{le_Fsimple} gives that $F(\chi^{(q)}_{0},\chi^{(q)}_{0},j_1,j_2,m)=0$ unless $q$ is squarefree. We can factorise it by Lemma~\ref{lem_Fmult} as
\begin{align*}
F(\chi^{(q)}_{0},\chi^{(q)}_{0},j_1,j_2,m)=F(\chi_{0}^{(q^{\dagger})},\chi_{0}^{(q^{\dagger})},j_1,j_2,m)F(\chi^{(q/q^{\dagger})}_{0},\chi^{(q/q^{\dagger})}_{0},1,1,m),
\end{align*}
where we used that $j_i$ divides $q^\dagger$ and so is coprime to $q/q^\dagger$.

We arrive at 
\begin{align*}
&\sum_{a(q)^*}\nu_1^T(a,q,\eta)\nu_2^T(a,q,\eta)e_q(-am)\\
&= \frac{V(\widetilde{\mathcal{P}})^2 T(\eta)^2F(\chi^{(q/q^{\dagger})}_{0},\chi^{(q/q^{\dagger})}_{0},1,1,m)}{\varphi_2(\widetilde{q})^2\varphi(q/\widetilde{q})^2}\sum_{j_i|q^\dagger}S_1^\dagger(q,j_1,1)S_2^\dagger(q,j_2,1)F(\chi_{0}^{(q^{\dagger})},\chi_{0}^{(q^{\dagger})},j_1,j_2,m)\\
&=\frac{V(\widetilde{\mathcal{P}})^2 T(\eta)^2 F(\chi^{(q/q^{\dagger})}_{0},\chi^{(q/q^{\dagger})}_{0},1,1,m) }{\varphi_2(\widetilde{q})^2\varphi(q/\widetilde{q})^2}\mathcal{F}(q^\dagger,m),
\end{align*}
say. We follow the same steps for $\nu_1^T \nu_2^{\widetilde{T}}$ and observe that the $q^\dagger$ component of $\widetilde{\chi}^{(\widetilde{q})}\chi^{(q)}_{0}$ is $\chi^{(q^\dagger)}_{0}$. We get
\begin{align*}
&\sum_{a(q)^*}\nu_1^T(a,q,\eta)\nu_2^{\widetilde{T}}(a,q,\eta)e_q(-am)\\
&= 1_{2^t|q}\frac{\overline{\widetilde{\chi}^{(\widetilde{r}/q')}}(-2)\mu(\widetilde{r}/q')V(\widetilde{\mathcal{P}})^2 T(\eta)\widetilde{T}(\eta) F(\chi^{(q/q^{\dagger})}_{0},\widetilde{\chi}^{(q')}\chi^{(q/q^{\dagger})}_{0},1,1,m)}{\varphi_2(\widetilde{q})^2\varphi_2(\widetilde{r}/q')\varphi(q/\widetilde{q})^2}\mathcal{F}(q^\dagger,m),
\end{align*}
and as $F(\chi_1,\chi_2,1,1,m)=F(\chi_2,\chi_1,1,1,m)$ we also have the same formula for
\begin{align*}
\sum_{a(q)^*}\nu_1^{\widetilde{T}}(a,q,\eta)\nu_2^T(a,q,\eta)e_q(-am).
\end{align*}

We recall that $\widetilde{r}/q'$ is odd and squarefree, so $\mu(\widetilde{r}/q')^2=1$ and, since $\widetilde{\chi}$ is a quadratic character, $\overline{\widetilde{\chi}^{(\widetilde{r}/q')}}^2(-2)=1$. We get
\begin{align*}
&\sum_{a(q)^*}\nu_1^{\widetilde{T}}(a,q,\eta)\nu_2^{\widetilde{T}}(a,q,\eta)e_q(-am)\\
&=1_{2^t|q}\frac{V(\widetilde{\mathcal{P}})^2 \widetilde{T}(\eta)^2F(\widetilde{\chi}^{(q')}\chi^{(q/q^{\dagger})}_{0},\widetilde{\chi}^{(q')}\chi^{(q/q^{\dagger})}_{0},1,1,m)}{\varphi_2(\widetilde{q})^2\varphi_2(\widetilde{r}/q')^2\varphi(q/\widetilde{q})^2}\mathcal{F}(q^{\dagger},m).
\end{align*}
Combining these, we see that the left-hand side of~\eqref{e34} is 
\begin{align}\begin{split} &V(\widetilde{\mathcal{P}})^2\sum_{q\in \mathcal{Q}}\mathcal{F}(q^\dagger,m) \int_{|\eta|\leq Q^{-1}}\Bigl(\frac{T(\eta)^2F(\chi^{(q/q^{\dagger})}_{0},\chi^{(q/q^{\dagger})}_{0},1,1,m)}{\varphi_2(\widetilde{q})^2\varphi(q/\widetilde{q})^2} \\
&+1_{2^t|q}\frac{2 \overline{\widetilde{\chi}^{(\widetilde{r}/q')}}(-2)\mu(\widetilde{r}/q')T(\eta)\widetilde{T}(\eta) F(\chi^{(q/q^{\dagger})}_{0},\widetilde{\chi}^{(q')}\chi^{(q/q^{\dagger})}_{0},1,1,m)}{\varphi_2(\widetilde{q})^2\varphi_2(\widetilde{r}/q')\varphi(q/\widetilde{q})^2}  \\
&+1_{2^t|q}\frac{\widetilde{T}(\eta)^2F(\widetilde{\chi}^{(q')}\chi^{(q/q^{\dagger})}_{0},\widetilde{\chi}^{(q')}\chi^{(q/q^{\dagger})}_{0},1,1,m)}{\varphi_2(\widetilde{q})^2\varphi_2(\widetilde{r}/q')^2\varphi(q/\widetilde{q})^2}\Bigr)e(-\eta m)\,d\eta.
\label{e25}
\end{split}
\end{align}

We now apply the factorisation $q=2^s\widetilde{q}q^\dagger$ to get
\begin{align*}
\frac{F(\chi^{(q/q^{\dagger})}_{0},\chi^{(q/q^{\dagger})}_{0},1,1,m)}{\varphi_2(\widetilde{q})^2\varphi(q/\widetilde{q})^2}=\frac{F(\chi_{0,2^s\widetilde{q}},\chi_{0,2^s\widetilde{q}},1,1,m)}{\varphi_2(2^s\widetilde{q})^2} \frac{1}{\varphi(q^\dagger)^2}=:\mathcal{L}_1(2^s\widetilde{q},m)\frac{1}{\varphi(q^\dagger)^2}
\end{align*}
and
\begin{align*}
&\frac{2 \overline{\widetilde{\chi}^{(\widetilde{r}/q')}}(-2)\mu(\widetilde{r}/q')F(\chi^{(q/q^{\dagger})}_{0},\widetilde{\chi}^{(q')}\chi^{(q/q^{\dagger})}_{0},1,1,m)}{\varphi_2(\widetilde{q})^2\varphi_2(\widetilde{r}/q')\varphi(q/\widetilde{q})^2}\\
&=\frac{2 \overline{\widetilde{\chi}^{(\widetilde{r}/q')}}(-2)\mu(\widetilde{r}/q') F(\chi_{0,2^s\widetilde{q}},\widetilde{\chi}^{(q')}\chi_{0,2^s\widetilde{q}},1,1,m)}{\varphi_2(2^s\widetilde{q})\varphi_2(\widetilde{r})}\frac{1}{\varphi(q^\dagger)^2}\\
&=:\mathcal{L}_2(2^s\widetilde{q},m)\frac{1}{\varphi(q^\dagger)^2}
\end{align*}
and
\begin{align*}
\frac{F(\widetilde{\chi}^{(q')}\chi^{(q/q^{\dagger})}_{0},\widetilde{\chi}^{(q')}\chi^{(q/q^{\dagger})}_{0},1,1,m)}{\varphi_2(\widetilde{q})^2\varphi_2(\widetilde{r}/q')^2\varphi(q/\widetilde{q})^2}
&=\frac{F(\widetilde{\chi}^{(q')}\chi_{0,2^s\widetilde{q}},\widetilde{\chi}^{(q')}\chi_{0,2^s\widetilde{q}},1,1,m)}{\varphi_2(\widetilde{r})^2}\frac{1}{\varphi(q^\dagger)^2}\\
&=:\mathcal{L}_3(2^s\widetilde{q},m)\frac{1}{\varphi(q^\dagger)^2},
\end{align*}
say. 

By Lemma~\ref{le_Fsimple} $\mathcal{L}_1(2^s\widetilde{q})$ vanishes if $s>1$, $\mathcal{L}_2(2^s\widetilde{q})$ vanishes unless $s=t=0$ (as $t\neq 1$), and $\mathcal{L}_3(2^s\widetilde{q})$ vanishes unless $s=t$. Therefore, with the just introduced notation and the fact that $t\leq 3$,~\eqref{e25} becomes
\begin{align} \nonumber&V(\widetilde{\mathcal{P}})^2\sum_{q\in \mathcal{Q}}\frac{\mathcal{F}(q^\dagger,m)}{\varphi(q^\dagger)^2}\int_{|\eta|\leq Q^{-1}}\Bigl(1_{s\leq 1}T(\eta)^2\mathcal{L}_1(2^s\widetilde{q},m)\\
&+1_{s=t=0}T(\eta)\widetilde{T}(\eta)\mathcal{L}_2(2^s\widetilde{q},m)+1_{s=t}\widetilde{T}(\eta)^2\mathcal{L}_3(2^s\widetilde{q},m)\Bigr)e(-\eta m)\,d\eta \nonumber \\
\begin{split}
&= \label{eq_after25}V(\widetilde{\mathcal{P}})^2\int_{|\eta|\leq Q^{-1}}\sum_{\substack{s\leq 3\\ \widetilde{q}|\widetilde{\mathcal{P}}}}\Bigl(1_{s\leq 1}T(\eta)^2\mathcal{L}_1(2^s\widetilde{q},m)\\
&+1_{s=t=0}T(\eta)\widetilde{T}(\eta)\mathcal{L}_2(2^s\widetilde{q},m)+1_{s=t}\widetilde{T}(\eta)^2\mathcal{L}_3(2^s\widetilde{q},m)\Bigr)e(\eta m)\, d\eta \sum_{q^\dagger\leq P^{c_0}/2^s}\frac{\mathcal{F}(q^\dagger,m)}{\varphi(q^\dagger)^2},  
\end{split}
\end{align}
At this point it was helpful that the integration range is independent of $q$.

We now evaluate the sum over $q^\dagger$ in~\eqref{eq_after25}. It is in this evaluation that the singular series appears. The relevant local factors are 
\begin{align}
\label{eq_sigmadef}\sigma(p,m)&\coloneqq F(\chi^{(p)}_{0},\chi^{(p)}_{0},1,1,m)
\end{align}
and the singular series is given by
\begin{align}\label{eq_SSdef}
\mathfrak{S}(m)\coloneqq \prod_p \Bigl(1+\frac{\sigma(p,m)}{\varphi_2(p)^2}\Bigr).
\end{align}
As we sieve separately for the exceptional modulus and $2$, we also require notation for the partial singular series given by
\begin{align*}
\mathfrak{S}(m,2\widetilde{\mathcal{P}})\coloneqq \prod_{p|2\widetilde{\mathcal{P}}} \Bigl(1+\frac{\sigma(p,m)}{\varphi_2(p)^2}\Bigr).
\end{align*}

With this notation, we postulate the following evaluation for the sum over $q^\dagger$ in~\eqref{eq_after25}. It will be proved in Section~\ref{sec: singular} and can be seen as a combination of the typical completion of the singular series in the circle method with a fundamental lemma of sieve theory.
\begin{lemma}\label{lem_SS1} Let the notation be as before and assume $P'\leq P^{c_0} $. We have 
	\begin{align*}
	\sum_{q^\dagger\leq P'}\frac{\mathcal{F}(q^\dagger,m)}{\varphi(q^\dagger)^2}=V(\mathcal{P}^\dagger)\frac{\mathfrak{S}(m)}{\mathfrak{S}(m,2\widetilde{\mathcal{P}})}\left(1+O\bigl(e^{-c\frac{\log D_0}{\log P}}+e^{100\sqrt{\log N}}(P')^{-1}\bigr)\right).
	\end{align*}
\end{lemma}
We recall that $V(\mathcal{P}^\dagger)V(\widetilde{\mathcal{P}})=V(\mathcal{P})$ so that an application of Lemma~\ref{lem_SS1} tells us that~\eqref{eq_after25} is
\begin{align}\label{eq_after252}\nonumber
&=V(\mathcal{P})^2\frac{\mathfrak{S}(m)}{\mathfrak{S}(m,2\widetilde{\mathcal{P}})}\int_{|\eta|\leq Q^{-1}}\sum_{\substack{s\leq 3\\ \widetilde{q}|\widetilde{\mathcal{P}}}}\Bigl(1_{s\leq 1}T(\eta)^2\mathcal{L}_1(2^s\widetilde{q},m)+1_{s=t=0}T(\eta)\widetilde{T}(\eta)\mathcal{L}_2(2^s\widetilde{q},m)\\
&+1_{s=t}\widetilde{T}(\eta)^2\mathcal{L}_3(2^s\widetilde{q},m)\Bigr)e(\eta m)\, d\eta\times \Bigl(1+O\bigl(e^{-c\frac{\log D_0}{\log P}}+e^{100\sqrt{\log N}} P^{-c_0}\bigr) \Bigr)
\end{align}

We can complete the integrals over $\eta$ to obtain
\begin{align}
\int_{|\eta|\leq Q^{-1}}T(\eta)^2e(-\eta m) \,d\eta &=\int_0^1 T(\eta)^2e(-\eta m) \,d\eta+O(Q)= m+O(Q)\label{e24a} \\
\int_{|\eta|\leq Q^{-1}}\widetilde{T}(\eta)T(\eta) e(-\eta m)\,d\eta&=\int_0^1 \widetilde{T}(\eta)T(\eta)e(-\eta m) \,d\eta+O(Q)=: \widetilde{J}(m)+O(Q)\label{e24b}\\
\int_{|\eta|\leq Q^{-1}}\widetilde{T}(\eta)^2e(-\eta m)\,d\eta&=\int_0^1 \widetilde{T}(\eta)^2e(-\eta m) \,d\eta+O(Q)=: \widetilde{I}(m)+O(Q)\label{e24c}.
\end{align}

To prove~\eqref{e24a}, note that by the geometric sum formula we have $|T(\eta)|\ll Q$ for $\eta \in [-1/2,1/2]$, $|\eta|\geq 1/Q$, and then claim follows by crudely apply the triangle inequality to the integral over $[-1/2,1/2]\setminus [-1/Q,1/Q]$. To prove~\eqref{e24b} and~\eqref{e24c}, we argue essentially in the same way; partial summation gives $|\widetilde{T}(\eta)|\ll Q$ for $\eta \in [-1/2,1/2]\setminus [-1/Q,1/Q],$ and then we can again apply a crude pointwise bound. 

We can now define the main term function $\mathcal{M}(m)$ that appears in Proposition~\ref{prop2} and in~\eqref{e34} as 
\begin{align*}
\mathcal{M}(m)\coloneqq \frac{1}{m\mathfrak{S}(m,2\widetilde{\mathcal{P}})}\sum_{\substack{s\leq 3\\ \widetilde{q}|\widetilde{\mathcal{P}}}}\Bigl(1_{s\leq 1}m\mathcal{L}_1(2^s\widetilde{q},m)+1_{s=t=0}\widetilde{J}(m)\mathcal{L}_2(\widetilde{q},m)+1_{s=t}\widetilde{I}(m)\mathcal{L}_3(2^t\widetilde{q},m)\Bigr).
\end{align*}
With this definition,~\eqref{e24a},~\eqref{e24b},~\eqref{e24c}, the trivial estimates $|\widetilde{J}(m)|,|\widetilde{I}(m)|\leq N$, and recalling that $Q=NP^{-c_0}$, we see that~\eqref{eq_after252} is
\begin{align}\begin{split} &V(\mathcal{P})\mathfrak{S}(m)\Bigl( m\mathcal{M}(m)\bigl(1+O(e^{-c\frac{\log D_0}{\log P}})\bigr) \\
	&+O\Bigl(\frac{Ne^{O(\sqrt{\log N})}}{P^{c_0}\mathfrak{S}(m,2\widetilde{\mathcal{P}})}\Bigl(\bigl| \sum_{\substack{s\leq 1\\ \widetilde{q}|\widetilde{\mathcal{P}}}}\mathcal{L}_1(2^s\widetilde{q},m)\bigr|+\bigl|\sum_{\substack{ \widetilde{q}|\widetilde{\mathcal{P}}}}\mathcal{L}_2(\widetilde{q},m)\bigr|+\bigl|\sum_{\substack{ \widetilde{q}|\widetilde{\mathcal{P}}}}\mathcal{L}_3(2^t\widetilde{q},m)\bigr|\Bigr)\Bigr)\Bigr)\label{e27_MT}
	\end{split}
\end{align}
To complete our treatment of the main term, it remains to show that $\mathcal{M}(m)$ is as required in Key Proposition~\ref{prop2} (see~\eqref{eq_prop2_1} and~\eqref{eq_prop2_2}) and to estimate the error term in~\eqref{e27_MT}. Both are achieved if we understand the relation of the sums over $\mathcal{L}_i$ to $\mathfrak{S}(m,2\widetilde{P})$.

The functions $\mathcal{L}_i$ are can be expressed in terms of the local contribution of the exceptional modulus and the prime $2$ to the binary additive problem. More precisely, if we define
\begin{align}
\label{eq_sigma'def}\sigma'(p,m)&\coloneqq F(\chi^{(p)}_{0},\widetilde{\chi}^{(p)},1,1,m) 
\\
\label{eq_sigmatildedef}\widetilde{\sigma}(p,m)&\coloneqq \begin{cases}F(\widetilde{\chi}^{(p)},\widetilde{\chi}^{(p)},1,1,m) & \text{ if } p\neq 2, \\
F(\widetilde{\chi}^{(2^t)},\widetilde{\chi}^{(2^t)},1,1,m)& \text{ if } p= 2,
\end{cases}
\end{align}
then by multiplicativity (Lemma~\ref{lem_Fmult}) and the vanishing at higher prime powers (Lemma~\ref{le_Fsimple}) we have 
\begin{align*}
\mathcal{L}_1(2^s \widetilde{q})&=1_{s\leq 1}|\mu( \widetilde{q})|\prod_{p\mid 2^s\widetilde{q}}\frac{\sigma(p,m)}{\varphi_2(p)^2}\\
|\mathcal{L}_2( \widetilde{q})|&=|\mu(\widetilde{q})|\frac{1}{\varphi_2(\widetilde{r})}\prod_{p|\widetilde{q}}\frac{|\sigma'(p,m)|}{\varphi_2(p)}\\
\mathcal{L}_3(2^t \widetilde{q})&=|\mu(\widetilde{q})|\frac{1}{\varphi_2(\widetilde{r})^2}\prod_{p|2^t\widetilde{q}}\widetilde{\sigma}(p,m).
\end{align*}

Therefore, we get
\begin{align*}
\sum_{\substack{s\leq 1\\ \widetilde{q}|\widetilde{\mathcal{P}}}}\mathcal{L}_1(2^s\widetilde{q})&=\prod_{p\mid 2\widetilde{\mathcal{P}}}\Bigl(1+\frac{\sigma(p,m)}{\varphi_2(p)^2}\Bigr)=\mathfrak{S}(m,2\widetilde{\mathcal{P}})\\
\bigl|\sum_{\substack{\widetilde{q}|\widetilde{\mathcal{P}}}}\mathcal{L}_2(\widetilde{q})\bigr|&\leq\prod_{p|\widetilde{\mathcal{P}}}\frac{\bigl(1+\frac{|\sigma'(p,m)|}{p-2}\bigr)}{p-2}\\
\sum_{\substack{ \widetilde{q}|\widetilde{\mathcal{P}}}}\mathcal{L}_3(2^t\widetilde{q})&=\frac{\widetilde{\sigma}(2,m)}{\varphi(2^t)^2}\prod_{p|\widetilde{\mathcal{P}}}\frac{1+\widetilde{\sigma}(p,m)}{(p-2)^2}.
\end{align*}

The following lemma gives explicitly the value of the three versions of $\sigma$. It is proved in Section~\ref{sec: gauss} and also confirms that the singular series matches the one in Key Proposition~\ref{prop2}.
\begin{lemma}[Evaluation of main term local densities]
	\label{lem_Feval} Let $p>2$ be a prime. We have 
	\begin{align}
	\label{eq_Feval1} \sigma(p,m)&=\begin{cases}
	p-4 &\,\,\textnormal{if}\,\, p\mid m\,\, \textnormal{or}\,\, p\mid m+4 \\
	2p-4 &\,\,\textnormal{if}\,\, p\mid m+2 \\
	-4 &\,\,\textnormal{else},
	\end{cases}\\
		\label{eq_Feval5} \sigma'(p,m)&=\begin{cases}
	-p(\widetilde{\chi}_p(m)+\widetilde{\chi}_p(-2)-\widetilde{\chi}_p(m+2))-\widetilde{\chi}_p(-2)-\widetilde{\chi}_p(2)&\,\, \textnormal{if}\,\, p\mid m+2\\
	-p(\widetilde{\chi}_p(m)+\widetilde{\chi}_p(2)-\widetilde{\chi}_p(m+2))-\widetilde{\chi}_p(-2)-\widetilde{\chi}_p(2)&\,\, \textnormal{if}\,\, p\mid m+4\\
	-p(\widetilde{\chi}_p(m)-\widetilde{\chi}_p(m+2))-\widetilde{\chi}_p(-2)-\widetilde{\chi}_p(2)&\,\, \textnormal{else,}
	\end{cases}\\
	\label{eq_Feval4} \widetilde{\sigma}(p,m)&=\begin{cases}(p^2-3p)\widetilde{\chi}_p(-1)-1 &\,\, \textnormal{if}\,\, p\mid m\\
	p(-\widetilde{\chi}_p(-1)-2\widetilde{\chi}_p(2)-1)-1 &\,\, \textnormal{if}\,\, p\mid m+4 \\
	p\widetilde{\chi}_p(-1)(-1-2\widetilde{\chi}_p(2m+4))-1 &\,\, \textnormal{else.}
	\end{cases} 
	\end{align}
	We have furthermore
	\begin{align}
	\label{eq_Fevalp=2a}\sigma(2,m)&=\begin{cases}
	1 &\,\,\textnormal{if}\,\, 2\mid m\,\, \\
	-1 &\,\,\textnormal{else},
	\end{cases}\\
		\label{eq_Fevalp=2b}|\widetilde{\sigma}(2,m)|&\leq \begin{cases}
	2^{2t-1} &\,\,\textnormal{if}\,\, 2\mid m\,\, \\
	0 &\,\,\textnormal{else.}
	\end{cases}
	\end{align}
\end{lemma}
With this evaluation, a short calculation shows that for $p\mid m$, $p\neq 2$
\begin{align}\label{eq_sigmatildsigma}
\Bigl|\frac{1+\widetilde{\sigma}(p,m)}{(p-2)^2}\Bigr|=1+\frac{\sigma(p,m)}{(p-2)^2}
\end{align}
and, applying a straightforward estimate for $p>7$ and a character table for quadratic characters for primes up to $7$, for $p\nmid m$, $p\neq 2$
\begin{align*}
\Bigl|\frac{1+\widetilde{\sigma}(p,m)}{(p-2)^2}\Bigr|\Big(1+\frac{\sigma(p,m)}{(p-2)^2}\Big)^{-1}\leq \frac{21}{25}.
\end{align*}
Similarly we get for $2\mid m$
\begin{align*}
\Bigl|\frac{\widetilde{\sigma}(2,m)}{\varphi(2^t)^2}\Bigr|\leq 2 = 1+\sigma(p,m).
\end{align*}

Together with the simple bound $\widetilde{J}(m)\leq m$, the evaluation for $\sigma'(p,m)$ in~\eqref{eq_Feval5}, and the inequality $\widetilde{I}(m)\leq m^{\widetilde{\beta}}$ we get consequently
\begin{align*}
\mathcal{M}(m)\geq 1-m^{\widetilde{\beta}-1}\prod_{\substack{p|\widetilde{r}\\ p\nmid m}}\frac{21}{25}+O(\widetilde{r}^{-0.99}).
\end{align*}
Here the second and third term are not present in the absence of an exceptional modulus and we have $\mathcal{M}(m)=\frac{\mathfrak{S}(m,2)}{\mathfrak{S}(m,2)}=1$ in that case. This confirms the proposed behaviour of the function $\mathcal{M}(m)$.

Finally, we consider the error term in~\eqref{e27_MT}. By the same considerations as for $\mathcal{M}(m)$, the sums over $s, \widetilde{q}$ are $O(\mathfrak{S}(m,2\widetilde{\mathcal{P}}))$ and so the error is admissible, recalling that we consider $m\geq NP^{-c_1}$. 

\subsection{Error term}
In this subsection we establish~\eqref{e35}. By definition,
\begin{align*}
&\sum_{a(q)^*}\nu_1^W(a,q,\eta)\nu_2^W(a,q,\eta)e_q(-am)\\
&=\frac{V(\widetilde{\mathcal{P}})^2 }{\varphi_2(\widetilde{q})^2\varphi(q/\widetilde{q})^2}\sum_{\substack{j_i|q^\dagger \\ l_i|q}}\sum_{\substack{e_i|\mathcal{P}^{\dagger}\widetilde{\mathcal{P}}\\ (e_i,q)=1\\eq\leq P^{2c_0}}}\frac{\mu(\widetilde{e}_1)\mu(\widetilde{e}_2)S_1^\dagger(q,j_1,e_1)S_2^\dagger(q,j_2,e_2)}{\varphi_2(\widetilde{e}_1)\varphi_2(\widetilde{e}_2)} \\
&\times \sum_{\substack{\psi_i(e_i)^*\\ \chi(l_i)^*}}\overline{\psi_1}(-2)\overline{\psi_2}(-2)F(\chi_1\chi^{(q)}_{0},\chi_2\chi^{(q)}_{0},j_1,j_2,m)W(\chi_1\psi_1,\eta)W(\chi_2\psi_2,\eta).
\end{align*}
 For a primitive character $\xi$, put 
\begin{align*}
\mathscr{W}(\xi)\coloneqq \Bigl(\int_{|\eta|\leq Q^{-1}} |W(\xi,\eta)|^2 \,d\eta \Bigr)^{1/2}.
\end{align*}
This is a larger range of integration than in the analogous integral in~\cite[formula (6.5)]{mv}, but the proof there would go still go through with this wider choice of major arcs and, just as for the main term, the fact that we chose the range of integration independent of $q$ again simplifies matters for us. 

An application of the Cauchy--Schwarz inequality gives
\begin{align}
&\sum_{\substack{q\leq P^{c_0}\\ a(q)^*}}e_q(-am) \int_{|\eta|\leq Q^{-1}}\nu_1^W(a,q,\eta)\nu_2^W(a,q,\eta)e(-\eta m)\,d\eta \nonumber\\
\begin{split}\label{e31}
\leq & V(\widetilde{\mathcal{P}})^2 \sum_{q\leq P^{c_0}} \frac{1}{\varphi_2(\widetilde{q})^2\varphi(q/\widetilde{q})^2} \sum_{\substack{j_i|q^\dagger \\ l_i|q}}\sum_{\substack{e_i|\mathcal{P}^{\dagger}\widetilde{\mathcal{P}}\\ (e_i,q)=1\\ e_il_i\leq P^{2c_0}}} \frac{|S_1^\dagger(q,j_1,e_1)S_2^\dagger(q,j_2,e_2)|}{\varphi_2(\widetilde{e}_1)\varphi_2(\widetilde{e}_2)}\\
&\times \sum_{\substack{\psi_i(e_i)^*\\ \chi(l_i)^*}} |F(\chi_1\chi^{(q)}_{0},\chi_2\chi^{(q)}_{0},j_1,j_2,m)|\, \mathscr{W}(\chi_1\psi_1)\mathscr{W}(\chi_2\psi_2).
\end{split}
\end{align}

Just as in the previous subsection we can use Lemma~\ref{lem_Fmult} to split $F$ into factors over divisors of $\mathcal{P}^{\dagger}$ and $\widetilde{\mathcal{P}}$ and write $q=2^s\widetilde{q}q^\dagger$. We furthermore discard the range condition $2^s\widetilde{q}q^\dagger<P^{c_0}$ to see that~\eqref{e31} is
\begin{align}\nonumber
&\leq V(\widetilde{\mathcal{P}})^2\sum_{2^s\widetilde{q}q^\dagger}\sum_{\substack{j_i|q^\dagger \\ l_i|q}}\sum_{\substack{e_i|\mathcal{P}^{\dagger}\widetilde{\mathcal{P}}\\ (e_i,q)=1\\ e_il_i\leq P^{2c_0}}}\sum_{\substack{\psi_i(e_i)^*\\ \chi(l_i)^*}} \frac{|F(\chi^{(l_1^\dagger)}_{1}\chi^{(q^\dagger)}_{0},\chi^{(l_2^\dagger)}_{1}\chi^{(q^\dagger)}_{0},j_1,j_2,m)||S_1^\dagger(q,j_1,e_1)S_2^\dagger(q,j_2,e_2)|}{\varphi(q^\dagger)^2}\\
&\times \frac{|F(\chi^{(\widetilde{l_1}2^s)}_{1}\chi^{(\widetilde{q}2^s)}_{0},\chi^{(\widetilde{l_2}2^s)}_{2}\chi^{(\widetilde{q}2^s)}_{0},1,1,m)|}{\varphi_2(\widetilde{q}2^s)^2\varphi_2(\widetilde{e_1})\varphi_2(\widetilde{e_2})} \mathscr{W}(\chi_1\psi_1)\mathscr{W}(\chi_2\psi_2).\label{eq_32}
\end{align}

Each $\xi_i\coloneqq \chi_i\psi_i$ is a primitive character to the modulus $g_i\coloneqq e_i l_i\leq P^{2c_0}$. We sort after fixed pairs $\xi_1,\xi_2$ to see that~\eqref{eq_32} is
\begin{align}
\leq V(\widetilde{\mathcal{P}})^2\sum_{\substack{g_i\leq P^{2c_0} \\ \xi_i(g_i)^*}}\widetilde{\mathcal{G}}(\xi_1,\xi_2,m)\mathcal{G}^\dagger(\xi_1,\xi_2,m) \mathscr{W}(\xi_1)\mathscr{W}(x_2),\label{eq_33}
\end{align}
where 
\begin{align}\label{eq_Gtilddef}
\widetilde{\mathcal{G}}(\xi_1,\xi_2,m)&\coloneqq \sum_{\substack{e_i l_i=2^{g_i(2)}\widetilde{g}_i\\ e_i|\widetilde{\mathcal{P}}}} \frac{1}{\varphi_2(e_1)\varphi_2(e_2)} \sum_{\substack{q|(2\widetilde{\mathcal{P}})^\infty\\  l_i|q \\ (q,e_1e_2)=1}}\frac{|F(\xi^{(l_1)}_{1}\chi^{(q)}_{0},\xi^{(l_2)}_{2}\chi^{(q)}_{0},1,1,m)|}{\varphi_2(q)^2}, \\
\label{eq_Gdagdef}\mathcal{G}^\dagger(\xi_1,\xi_2,m)&\coloneqq \sum_{\substack{e_i l_i=g_i^\dagger\\ e_i|\mathcal{P^\dagger}}} \sum_{\substack{q|(\mathcal{P}^{\dagger})^\infty\\  \substack{ l_i|q \\ (q,e_1e_2)=1}}}\sum_{j_i|q} \frac{ |S_1^\dagger(q,j_1,e_1)S_2^\dagger(q,j_2,e_2)F(\xi^{(l_1)}_{1}\chi^{(q)}_{0},\xi^{(l_2)}_{2}\chi^{(q)}_{0},j_1,j_2,m)|}{\varphi(q)^2}.
\end{align}
Here $2^{g_i(2)}$ denotes the $2$ component of $g_i$. 
The product $\mathcal{G}^\dagger(\xi_1,\xi_2,m)\widetilde{\mathcal{G}}(\xi_1,\xi_2,m)$ can be seen as a sieve-weighted pseudo-singular series induced by the pair $\xi_1,\xi_2$. We can bound it by the following Lemma that is proved in Subsection~\ref{sec: singular}. It is our analogue of~\cite[Lemma 5.5]{mv}. 
\begin{lemma}\label{lem_SS2}
Let $\xi_i$ be primitive characters.  We have
\begin{align}\label{eq_lemSS2_1}
\widetilde{\mathcal{G}}(\xi_1,\xi_2,m)\mathcal{G}^\dagger(\xi_1,\xi_2,m)\ll V(\mathcal{P}^\dagger)^2 \mathfrak{S}(m).
\end{align}
\end{lemma}
Applying this lemma, we find that~\eqref{eq_33} is
\begin{align*}
\ll V(\mathcal{P})^2 \mathfrak{S}(m) \mathscr{W}^2,
\end{align*}
where
\begin{align}\label{eq_W}
\mathscr{W}\coloneqq \sum_{\substack{g\leq P^{2c_0} \\ \xi(g)^*}}\mathscr{W}(\xi).
\end{align}
Since
\begin{align*}
\int_0^1 |T(\eta)|^2\, d\eta \ll N,\quad 
\int_0^1 |\widetilde{T}(\eta)|^2\, d\eta \ll N,
\end{align*}
and we can also apply Lemma~\ref{lem_SS2} if one of the characters is primitive modulo $1$ or the exceptional character, we similarly estimate
\begin{align*}
\sum_{\substack{q\leq P^{c_0}\\ a(q)^*}}e_q(-am) \int_{|\eta|\leq Q^{-1}}\nu_1^T(a,q,\eta)\nu_2^W(a,q,\eta)e(-\eta m)\,d\eta&\ll V(\mathcal{P})^2 \mathfrak{S}(m) N^{1/2}\mathscr{W}\\
\sum_{\substack{q\leq P^{c_0}\\ a(q)^*}}e_q(-am) \int_{|\eta|\leq Q^{-1}}\nu_1^{\widetilde{T}}(a,q,\eta)\nu_2^W(a,q,\eta)e(-\eta m)\,d\eta&\ll V(\mathcal{P})^2 \mathfrak{S}(m) N^{1/2}\mathscr{W}.
\end{align*}
Of course the same holds for the contribution of $\nu_1^W$ with $\nu_2^T$ or $\nu_2^{\widetilde{T}}$. 

It remains to bound $\mathscr{W}$. For this we use Gallagher's prime number theorem.
\begin{lemma}[Gallagher's prime number theorem] \label{lem_GPNT} Let $N\geq R\geq 2$. Then 
\begin{align}\label{eq_gallagher}
\sum_{q\leq R}\sum_{\chi(q)^*}\max_{\substack{h<x\leq N}}\frac{1}{h+N/R}\bigl|\,\,\,\,\psum_{x-h\leq n\leq x} \chi(n)\Lambda_0(n)\bigr|\ll e^{-c\frac{\log N}{\log R}},
\end{align}
where the $\#$ in summation denotes that if $q=1$ we need to subtract from the sum the term
\begin{align*}
\sum_{x-h\leq n\leq x}1,
\end{align*}
and if there is an exceptional zero $\widetilde{\beta}$ of level $R$, then for  $\chi=\widetilde{\chi}$ being the exceptional character we need to subtract
\begin{align*}
-\sum_{x-h\leq n\leq x}n^{\widetilde{\beta}-1},
\end{align*}
and the bound on the right of~\eqref{eq_gallagher} may be improved by a factor of 
\begin{align*}
(1-\widetilde{\beta})(\log R).
\end{align*}
\end{lemma}
\begin{proof}
 This is ~\cite[Lemma 4.3]{mv}.
\end{proof}

We follow the strategy in~\cite[Section 7]{mv}. By Gallagher's Lemma~\cite[Lemma 4.2]{mv} and the fact that we accounted for the possible exceptional zero, we have 
\begin{align*}
\mathscr{W}(\xi)&\ll \Bigl(\int_0^{2N}\frac{1}{Q} \bigl|\psum_{\substack{ n\leq N \\ |n-x|\leq Q}} \xi(n) \Lambda_0(n) \bigr|^2\, dx \Bigr)^{1/2}\\
&\ll N^{1/2} \max_{x\leq 2N} \frac{1}{Q} \Bigl| \psum_{\substack{n\leq N \\ |n-x|\leq Q}} \xi(n) \Lambda_0(n) \Bigr|.
\end{align*}

Applying Lemma~\ref{lem_GPNT} with $R=P^{2c_0}$ (and recalling~\eqref{eq_W}) gives us 
\begin{align*}
\mathscr{W}&\ll N^{1/2}\sum_{\substack{g\leq P^{2c_0} D_0\\ \xi(g)^*}} \max_{x\leq 2N} \frac{1}{Q} \Bigl| \psum_{\substack{n\leq N \\ |n-x|\leq Q}} \xi(n) \Lambda_0(n) \Bigr|\\
&\ll N^{1/2}e^{-c\frac{\log N}{\log P^{c_0}}}
\end{align*}
if there is no exceptional modulus. If the exceptional zero does exist we get the improved estimate
\begin{align*}
\mathscr{W}&\ll N^{1/2}(1-\widetilde{\beta})(\log P^{c_0}) e^{-c\frac{\log N}{\log  P^{c_0}}}. 
\end{align*}
As $c_0$ is fixed, this concludes the proof of~\eqref{e35} and also Proposition~\ref{prop2}, under the assumption of Lemmas~\ref{lem_Fmult},~\ref{le_Fsimple},~\ref{lem_SS1},~\ref{lem_Feval} and~\ref{lem_SS2} whose proofs are postponed to the next section.

\section{Auxiliary results}

The goal of this section is to prove several auxiliary results that were used in the previous section in the proof of Key Proposition~\ref{prop2}.

\subsection{Gau\ss{} Sums with gcd condition}\label{sec: gauss}
In this subsection we consider the modified Gau\ss{}  sums $c_{\chi}(a,j)$ and the related function $F(\chi_1,\chi_2,j_1,j_2,m)$ defined in~\eqref{e37} and~\eqref{eq_Fdef} respectively. 

We now introduce some more notation to handle the remaining statements. We write $c_{\chi^{(q)}_{0}}(a,j)=c_q(a,j)$ if $\chi$ is principal. As usual, we also use $c_{\chi}(a)\coloneqq \sum_{b(q)^*}\chi(b)e_q(ab)$ to denote the Gau\ss{} sum without a gcd condition and write $c_{\chi^{(q)}_{0}}(a)=c_q(a)$ (which is a Ramanujan sum). We use the placeholder '$-$' in place of $j_i$ in $F(\chi_1,\chi_2,j_1,j_2,m)$ to denote that the corresponding Gau\ss{} sum has no gcd restriction. For example, this means that
\begin{align*}
	F(\chi_1,\chi_2,-,j_2,m)\coloneqq \sum_{a(q)^*}c_{\chi_1}(a)c_{\chi_2}(a,j_2)e_q(-am).
\end{align*}
We follow the usual notation and write 
\begin{align*}
	\tau(\chi)\coloneqq c_\chi(1).
\end{align*}

We start by proving the stated multiplicativity of $F$.
\begin{proof}[Proof of Lemma~\ref{lem_Fmult}]
Let $r,s$ be coprime integers with $q=rs$. It suffices to show that 
\begin{align*}
    F(\chi_1,\chi_2,j_1,j_2,m)=F(\chi^{(r)}_{1},\chi^{(r)}_{2},j_1,j_2,m)F(\chi^{(s)}_{1},\chi^{(s)}_{2},j_1,j_2,m)
\end{align*} We observe
\begin{align*}
\chi_i(b_i)1_{(b_i+2,q')=(j_i,q')}=\chi_{i}^{(r)}(b_i)1_{(b_i+2,r')=(j_i,r')}\chi_{i}^{(s)}(b_i)1_{(b_i+2,s')=(j_i,s')}.
\end{align*}
The lemma now follows from the Chinese remainder theorem since
\begin{align*}
F(\chi_1,\chi_2,j_1,j_2,m)=\sum_{b_1,b_2(q)}\chi_1(b_1)1_{(b_1+2,q')=(j_1,q')}\chi_2(b_2)1_{(b_2+2,q')=(j_2,q')}c_q(b_1+b_2-m),
\end{align*}
all appearing functions are suitably periodic, and $c_q(b_1+b_2-m)$ is multiplicative in $q$.
\end{proof}

By the principle of inclusion-exclusion we can rearrange the gcd condition. In particular, if $\chi_i$ are characters whose modulus is a power of $p$, to calculate $F(\chi_1,\chi_2,j_1,j_2,m)$ we observe that
\begin{align}\label{eq_c10p}
c_{\chi_i}(a,1)=c_{\chi_i}(a)-c_{\chi_i}(a,p)
\end{align}
and so
\begin{align}\label{eq_Fred1}
F(\chi_1,\chi_2,1,1,m)=F(\chi_1,\chi_2,-,-,m)-F(\chi_1,\chi_2,-,p,m)-F(\chi_1,\chi_2,p,-,m)+F(\chi_1,\chi_2,p,p)
\end{align}
and 
\begin{align}\label{eq_Fred2}
F(\chi_1,\chi_2,1,p,m)=F(\chi_1,\chi_2,-,p,m)-F(\chi_1,\chi_2,p,p,m).
\end{align}
To make use of these decompositions, we recall Lemma \ref{lem_ccalc}. We can now prove the formulas for $\sigma(p,m)$, $\sigma'(p,m)$, $\widetilde{\sigma}(p,m)$ that were used for the main term evaluation in Subsection~\ref{sec:majarcMT}. We also prove the following closely related result.
\begin{lemma}\label{lem_Feval2}
	Let $p\neq 2$. We have
	\begin{align}
	\label{eq_Feval2} F(\chi^{(p)}_{0},\chi^{(p)}_{0},p,1,m)&=\begin{cases}
	-p+2 &\,\,\textnormal{if}\,\, p\mid m+2\,\, \textnormal{or}\,\, p\mid m+4 \\
	2 &\,\,\textnormal{else},
	\end{cases}  \\
	\label{eq_Feval3} F(\chi^{(p)}_{0},\chi^{(p)}_{0},p,p,m)&=\begin{cases}
	p-1 &\,\,\,\,\,\textnormal{if}\,\,  p\mid m+4 \\
	-1 &\,\,\,\,\,\textnormal{else} .
	\end{cases}
	\end{align}
\end{lemma}

\begin{proof}[Proof of Lemma~\ref{lem_Feval} and Lemma~\ref{lem_Feval2}]
For arbitrary characters $\chi_1,\chi_2$ to the modulus $p$ by~\eqref{eq_Fred1} we have
\begin{align*}
F(\chi_1,\chi_2,1,1,m)&=\sum_{\substack{a(p)^*\\b_i(p)}}\chi_1(b_1)\chi_2(b_2)e_p(a(b_1+b_2-m))\\
&-\chi_1(-2)\sum_{\substack{a(p)^*\\b_2(p)}}\chi_2(b_2)e_p(a(-2+b_2-m))\\
&-\chi_1(-2)\sum_{\substack{a(p)^*\\b_1(p)}}\chi_1(b_1)e_p(a(b_1-2-m))\\
&+\chi_1\chi_2(-2)\sum_{a(p)^*}e_p(a(-4-m)) \\
&=\tau(\chi_1)\tau(\chi_2)c_{\overline{\chi_1 \chi_2}}(-m)-\chi_1(-2)\tau(\chi_2)c_{\overline{\chi_2}}(-m-2)\\
&-\chi_2(-2)\tau(\chi_1)c_{\overline{\chi_1}}(-m-2)+\chi_1\chi_2(-2)c_p(-m-4).
\end{align*}

If $\chi_1=\chi_2=\chi^{(p)}_{0}$ then $\tau(\chi_i)=-1$ and~\eqref{eq_Feval1} follows from Lemma~\ref{lem_ccalc}. For the case $\chi_1=\chi_2=\chi$ for some quadratic character $\chi$, we have
\begin{align*}
F(\chi,\chi,1,1,m)=\tau(\chi)^2 c_{p} (-m)-2\chi(-2)\tau(\chi)c_{\chi}(-m-2)+c_p(-m-4)
\end{align*}
Note that $\chi$ is primitive and so by Lemma~\ref{lem_ccalc} we get that $c_{\chi}(-m-2)=0$ if $p\mid m+2$. For the case $p\nmid m+2$ we get
\begin{align*}
\chi(-2)\tau(\chi)c_{\chi}(-m-2)=\chi(2(m+2))\tau(\chi)^2.
\end{align*}
Since $\chi$ is quadratic, we have  $\tau(\chi)^2=\chi(-1)p$, so that~\eqref{eq_Feval4} follows. The remaining results can be deduced in similar fashion, starting with~\eqref{eq_Fred2} to show~\eqref{eq_Feval2} and~\eqref{eq_Feval3}.

The case of $p=2$, that is~\eqref{eq_Fevalp=2a} and~\eqref{eq_Fevalp=2b} is as in~\cite[Section 5]{mv}, as the gcd condition has no effect.
\end{proof}

 We continue with an estimate that is used later in Section~\ref{sec: singular} in the proof of Lemma~\ref{lem_SS2} to bound the contribution of all characters induced by a fixed pair of primitive ones.
\begin{lemma} \label{lem_Festi}
Let $\chi_i$ be characters to the modulus $p^\alpha$, $p\neq 2$, $\alpha>0$. We have
\begin{align}\label{eq_lemFeststatement1}
|F(\chi_1,\chi_2,1,1,m)|&\leq \begin{cases}p^{2\alpha}-3p^{2\alpha-1}+1 & \text{ if } \chi_1=\overline{\chi_2} \text{ and } p^{\alpha}\mid m,  \\
p^{2\alpha-1/2}+3p^{2\alpha-1} & else.
\end{cases}
\end{align}
and 
\begin{align}\label{eq_lemFeststatement2}
|F(\chi_1,\chi_2,j_1,j_2,m)|&\leq 2p^{2 \alpha-1} \quad \text{ if } p|j_1 j_2.
\end{align}
Furthermore, for $p=2$ we have
\begin{align}\label{eq_lemFeststatementp=2}
|F(\chi_1,\chi_2,1,1,m)|&\leq 2^{2\alpha}.
\end{align}
\end{lemma}
\begin{proof}

For $i\in \{1,2 \}$ we denote by $p^{\alpha_i^*}$ the conductor of $\chi_i$. Let further $p^{\alpha'}$ be the conductor of $\chi_1\overline{\chi}_2$ and $p^{\alpha_m}=(p^\alpha,m)$. With the help of the identities~\eqref{eq_Fred1} and~\eqref{eq_Fred2}, we can reduce the lemma to the study of
\begin{align*}
&F(\chi_1,\chi_2,-,-,m),\\
 &F(\chi_1,\chi_2,-,p,m)= F(\chi_2,\chi_1,p,-,m)\,\,\textnormal{ and }\\
&F(\chi_1,\chi_2,p,p,m).
\end{align*}

We start by observing that (see also~\cite[proof of Lemma 5.5]{mv})
\begin{align*}
F(\chi_1,\chi_2,-,-,m)=\tau(\chi_1)\tau(\chi_2)c_{\overline{\chi_1 \chi_2}}(-m).
\end{align*}
We apply Lemma~\ref{lem_ccalc} to $c_{\overline{\chi_1 \chi_2}}(-m)$. If $\chi_1=\overline{\chi}_2$ (i.e. $\alpha'=0$) we have
\begin{align*}
c_{\overline{\chi_1 \chi_2}}(-m)=c_{p^\alpha}(m)=\begin{cases} p^{\alpha}(1-\frac{1}{p}) & \text{ if } p^\alpha|m \\
-p^{\alpha-1}  & \text{ if } p^{\alpha-1}||m \\
0  & \text{ else. }
\end{cases}
\end{align*}
If $\alpha'\neq 0$, then $c_{\overline{\chi_1\chi_2}}(-m)$ vanishes, unless $\alpha-\alpha_n-\alpha'=0$. We get
\begin{align*}
|c_{\overline{\chi_1\chi_2}}(-m)|&\leq p^{(\alpha-\alpha_m)/2}\frac{\varphi(p^{\alpha})}{\varphi(p^{\alpha-\alpha_m})} \\
&\leq p^{\alpha-1/2}.
\end{align*}

For $\tau(\chi_1)$, we have the following well-known results
\begin{align}
|\tau(\chi_1)|&=\begin{cases}p^{\alpha/2} &\text{ if } \alpha_1^*=\alpha \\
1 &\text{ if } \alpha_1^*=0, \alpha=1 \\
0 & \text{ else,}
\end{cases} \label{eq_lemFestitau1}\\
\tau(\chi_1)\tau(\overline{\chi_1})&=\chi(-1)|\tau(\chi_1)|^2. \label{eq_lemFestitau2}
\end{align}
We get
\begin{alignat}{2}
\label{eq_lemFesti_1} F(\chi_1,\chi_2,-,-,m)&=\chi_1(-1)p^{2\alpha}(1-\frac{1}{p}) &&\,\,\text{ if } \chi_1=\overline{\chi_2}, \,\,\alpha_1^*=\alpha \text{ and } p^\alpha\mid m\\
\label{eq_lemFesti_2}|F(\chi_1,\chi_2,-,-,m)|&\leq p^{2\alpha-1/2} &&\,\,\text{ else. }
\end{alignat}

We continue with the next case of $F$ and open the divisibility condition in the definition of $F(\chi_1,\chi_2,-,p,m)$ by characters. This gives us
\begin{align}
F(\chi_1,\chi_2,-,p,m)&=\frac{1}{\varphi(p)}\sum_{\psi(p)}\overline{\psi}(-2)\sum_{b_1(p^\alpha)}\chi_1(b_1)\sum_{b_2(p^\alpha)}\chi_2\psi(b_2)\sum_{a(p^\alpha)^*}e_{p^\alpha}(a(b_1+b_2-m)) \nonumber\\
&=\frac{1}{\varphi(p)}\sum_{\psi(p)}\overline{\psi}(-2)\tau(\chi_1)\tau(\chi_2\psi)c_{\overline{\chi_1\chi_2\psi}}(-m). \label{eq_lemFesti_3}
\end{align}

We consider first the case $\alpha'\leq 1$. In that case the conductor of $\chi_1\chi_2\psi$ is always at most $p$. Thus, 
\begin{align*}
c_{\overline{\chi_1\chi_2\psi}}(-m)&=\sum_{b(p^\alpha)}\overline{\chi_1\chi_2\psi}(b)e_{p^\alpha}(-bm)\\
&=\sum_{\substack{b'(p)^* \\ b''(p^{\alpha-1})}} \overline{\chi_1\chi_2\psi}(b''p+b')e_{p^\alpha}(-(b''p+b')m) \\
&=\sum_{\substack{b'(p)^*}}\overline{(\chi_1\chi_2)^*\psi}(b')e_{p^\alpha}(-b'm)\sum_{b''(p^{\alpha-1})}e_{p^{\alpha-1}}(-b'' m)\\
&= 1_{p^{\alpha-1}\mid m}p^{\alpha-1}\sum_{\substack{b'(p)^*}}\overline{(\chi_1\chi_2)^*\psi}(b')e_{p^{\alpha}}(-b'm),
\end{align*}
where $(\chi_1\chi_2)^*$ is the primitive character inducing $\chi_1\chi_2$. Plugging this in~\eqref{eq_lemFesti_3} gives
\begin{align*}
&F(\chi_1,\chi_2,-,p,m)\\
&=\frac{1_{p^{\alpha-1}\mid m}p^{\alpha-1}\tau(\chi_1)}{\varphi(p)}\sum_{\psi(p)}\overline{\psi}(-2)\sum_{b_2(p^{\alpha})}\chi_2\psi(b_2)e_{p^{\alpha}}(b_2) \sum_{\substack{b'(p)^*}}\overline{(\chi_1\chi_2)^*\psi}(b')e_{p}(-b'm/p^{\alpha-1})\\
&=\frac{1_{p^{\alpha-1}|m}p^{\alpha-1}\tau(\chi_1)}{\varphi(p)}\sum_{b_2(p^{\alpha})}\chi_2(b_2)e_{p^{\alpha}}(b_2)\sum_{\substack{b'(p)^*}}\overline{(\chi_1\chi_2)^*}(b')e_{p^{\alpha}}(-b'm)\sum_{\psi(p)}\psi(\overline{-2b'}b_2 )\\
&=1_{p^{\alpha-1}|m}p^{\alpha-1}\tau(\chi_1)\sum_{b_2(p^{\alpha})}\chi_2(b_2)e_{p^{\alpha}}(b_2) \overline{(\chi_1\chi_2)}(\overline{-2}b_2)e_{p^{\alpha}}(\overline{2}b_2m)\\
&=1_{p^{\alpha-1}|m}p^{\alpha-1}\tau(\chi_1)\chi_1\chi_2(-2)c_{\overline{\chi_1}}(\overline{2}m+1).
\end{align*}

If $p^\alpha|m$ then  $c_{\overline{\chi_1}}(\overline{2}m+1)=c_{\overline{\chi_1}}(1)=\tau(\overline{\chi_1})$ and we get
\begin{align}
\label{eq_lemFesti_4} F(\chi_1,\chi_2,-,p,m)=\chi(-1)p^{2\alpha-1} \quad  \text{ if } \chi_1=\overline{\chi_2} \text{ and } p^\alpha|m.
\end{align}
Furthermore, if $\alpha_1\geq 1$, then $\tau(\chi_1)=0$ unless $\alpha_1^*=\alpha$. By Lemma~\ref{lem_ccalc} and~\eqref{eq_lemFestitau2}, we have
\begin{align}\label{eq_lemFesti_5}
|\tau(\chi_1)c_{\overline{\chi_1}}(\overline{2}m+1)|\leq p^\alpha.
\end{align}

If $\alpha_1^*=0$, then $\tau(\chi_1)=-1$ and $|c_{\overline{\chi_1}}(\overline{2}m+1)|\leq p^\alpha$, so again~\eqref{eq_lemFesti_5} follows. We get (for the moment still restricted to $\alpha'\leq 1$) the estimate
\begin{align}\label{eq_lemFesti_6}
|F(\chi_1,\chi_2,-,p,m)|\leq p^{2\alpha-1}.
\end{align}

For the case $\alpha'\geq 2$ we observe that for any $\psi$ the conductor of $\chi_1\chi_2\psi$ in~\eqref{eq_lemFesti_3} is $p^{\alpha'}$. From Lemma~\ref{lem_ccalc} we get that $c_{\overline{\chi_1\chi_2\psi}}(-m)=0$ unless $\alpha_m= \alpha-\alpha'$ and so we can estimate 
\begin{align*}
|c_{\overline{\chi_1\chi_2\psi}}(-m)| &\leq p^{\alpha-\alpha'/2}\\
 &\leq p^{\alpha-1}.
\end{align*}
Using this bound, a summation with absolute values over $\psi$ in~\eqref{eq_lemFesti_3} shows that~\eqref{eq_lemFesti_6} holds for any value of $\alpha'$.

We now consider $F(\chi_1,\chi_2,p,p,m)$. We first show that
\begin{align}\label{eq_lemFesti_7}
|F(\chi_1,\chi_2,p,p,m)|\leq 1  &\text{ if } \chi_1=\overline{\chi_2} \text{ and } p^\alpha\mid m.
\end{align}
For $\alpha=1$ we have
\begin{align}\label{eq_lemFesti_8}
F(\chi_1,\chi_2,p,p,m)=\chi_1\chi_2(-2)c_p(m-4)
\end{align}
and~\eqref{eq_lemFesti_7} follows. Furthermore, we claim
\begin{align}\label{eq_lemFesti_9}
F(\chi_1,\chi_2,p,p,m)=0, \quad \text{ if } \alpha'\leq 1 \text{ and } \alpha\geq 2.
\end{align}
This will give~\eqref{eq_lemFesti_7} for the remaining $\alpha$. 

To show~\eqref{eq_lemFesti_9}, we note that the choice of $\alpha'$ implies that $\chi_2=\overline{\chi_1}\chi'$ for some character $\chi'$ with conductor at most $p$. Thus
\begin{align*}
F(\chi_1,\chi_2,p,p,m)=\sum_{\substack{b_i(p^\alpha)^*\\ b_i\equiv -2(p)}}\chi_1(b_1\overline{b_2})\chi'(b_2)\sum_{a(p^\alpha)^*}e_{p^\alpha}(a(b_1+b_2-m)).
\end{align*}
We substitute $b_2=b_1r$ so that this equals
\begin{align*}
\chi'(-2)\sum_{\substack{b_1,r(p^\alpha)^*\\ b_1\equiv -2(p), r\equiv 1(p)}}\chi_1(\overline{r})\sum_{a(p^\alpha)^*}e_{p^\alpha}(a(b_1+b_1\overline{r}-m)),
\end{align*}
where we used that $\chi'(rb_1)=\chi'(-2)$, since $rb_1\equiv -2(p)$ and the conductor of $\chi'$ is at most $p$. We consider the sum over $b_1$ and get
\begin{align*}
\sum_{\substack{b_1(p^\alpha),\\ b_1\equiv -2(p)}}e_{p^\alpha}(a(b_1+b_1\overline{r}))&=e_{p^\alpha}(-2a(1+\overline{r}))\sum_{\substack{b'(p^{\alpha-1})}} e_{p^{\alpha-1}}(b'a(1+\overline{r}))\\
&=0,
\end{align*}
since $p\nmid a(1+\overline{r})$. From~\eqref{eq_lemFesti_8} and~\eqref{eq_lemFesti_9} we also get the always valid estimate
\begin{align}\label{eq_lemFesti_10}
|F(\chi_1,\chi_2,p,p,m)|\leq p-1.
\end{align}

We gather the results and get from~\eqref{eq_Fred1},~\eqref{eq_lemFesti_1},~\eqref{eq_lemFesti_4}, and~\eqref{eq_lemFesti_7} the case $\chi_1=\overline{\chi_2},  p^\alpha|m$ of~\eqref{eq_lemFeststatement1}.The remaining cases of the lemma follow from a combination of the decompositions~\eqref{eq_Fred1} or~\eqref{eq_Fred2} with~\eqref{eq_lemFesti_2},~\eqref{eq_lemFesti_6}, and~\eqref{eq_lemFesti_10}.

For the statement in~\eqref{eq_lemFeststatementp=2} the gcd conditions do not play a role and so it follows directly from the considerations in~\cite[Section 5]{mv}.
\end{proof}

With the tools used in the previous proof, we can now also quickly deduce Lemma~\ref{le_Fsimple} and complete our considerations of Gau\ss{} sums with a gcd condition.
\begin{proof}[Proof of Lemma~\ref{le_Fsimple}] The statement follows from the fact that $\tau(\chi)$ vanishes if $\chi$ is a character to the modulus $\alpha>1$ that is not primitive (see~\eqref{eq_lemFestitau1}) and writing the congruence conditions with the help of characters as in the previous proof.
\end{proof}

\subsection{Sieve results}\label{sec: sieve}

An application of a sieve of the form
\begin{align}\label{eq_theta}
\theta(n)=\sum_{d\mid n}\lambda_d   
\end{align}
with range $\mathcal{P}$ on a sequence with local density function $\varphi(d)^{-1}$ gives rise to a main term of the form
\begin{align}
\label{eq_sieve1}\sum_{d|\mathcal{P}}\frac{\lambda_d}{\varphi(d)}=V(\mathcal{P}) \sum_{b|\mathcal{P}}\theta(b) h(b).
\end{align}
Here the function $h$ is multiplicative and given on primes by
\begin{align*}
h(p)=\frac{\varphi(p)^{-1}}{1-\varphi(p)^{-1}}.
\end{align*}

Let $D$ be the level of the sieve, $z=\mathcal{P}^+$ and write $s=\frac{\log D}{\log z}$. If the sieve fulfils a so-called fundamental lemma we have
\begin{align*}
\sum_{d|\mathcal{P}(z)}\frac{\lambda_d}{\varphi(d)}=V(\mathcal{P})\bigl(1+O(e^{-cs})\bigr)
\end{align*}
which is equivalent to
\begin{align}\label{eq_fundlem}
\Big|\sum_{\substack{b|\mathcal{P}\\ b\neq 1}}\theta(b) h(b)\Big|\ll e^{-cs}.
\end{align}
Since our pre-sieves are interacting with the additive problem we consider, we are led to the study of $S_i(q,j,e)$ as given in~\eqref{eq_Sdef} instead of~\eqref{eq_sieve1}. 

We state the results in this subsection in slightly more general terms and use the notation $g(d)$ and $h(d)$ with or without subscripts to denote pairs of multiplicative functions supported on the squarefree numbers with
\begin{align*}
0\leq g(p) < 1
\end{align*}
and satisfying the relation
\begin{align*}
h(p)=\frac{g(p)}{1-g(p)}.
\end{align*}
The effect of $q$ in $S_i(q,j,e)$ can be absorbed into the definition of the local density function, and we consider
\begin{align*}
g(e)\sum_{\substack{\substack{c|j \\ d|\mathcal{P}/(ej)}}}\lambda_{cde}g(d).
\end{align*}

We will see in the next subsection that treating the pseudo-singular series with sieve weights gives rise to a problem that is very similar to, but slightly more delicate than, what is encountered in the study of sieve weights in short intervals. The first lemma we need is a technical identity that translates $S(q,j,e)$ into the language of $\theta$, generalising~\eqref{eq_sieve1}. This result and its proof are motivated by~\cite[proof of Lemma 6.18]{cribro} or more precisely~\cite[proof of Lemma 5.1]{matoalmost} which fixes a mistake in the former. 

\begin{lemma} \label{lem_sie1}
Let $\mathcal{P}$ be a squarefree integer, $\lambda_d$ sieve weights supported on $d|\mathcal{P}$ only. Let $j$, $e$ be integers with $j|\mathcal{P}$, $e|\mathcal{P}$, and $(j,e)=1$. 
Then it holds that
\begin{align*}
g(e)\sum_{\substack{\substack{c|j \\ d|\mathcal{P}/(ej)}}}\lambda_{cde}g(d)=V\frac{h(j)}{g(j)}\mu(e)h(e)\sum_{b|\mathcal{P}/j}\theta(jb)h(b) \frac{\mu((b,e))}{h((b,e))},
\end{align*}
where $V=\prod_{p|\mathcal{P}}\bigl(1-g(p)\bigr)$, and $\theta(n)$ is given by~\eqref{eq_theta}.
\end{lemma}
\begin{proof}
Let $\gamma(d,j)$ be a multiplicative function in $d$, given on primes by
\begin{align*}
\gamma(p,j)=\begin{cases}g(p) & \quad\textnormal{ if }\quad p\nmid j \\
1 & \quad\textnormal{ if }\quad p| j .
\end{cases}
\end{align*}
We have
\begin{align*}
g(e)\sum_{\substack{\substack{c|j \\ d|\mathcal{P}/(ej)}}}\lambda_{cde}g(d)=\sum_{\substack{k|\mathcal{P}\\ k\equiv 0(e)}}\lambda_d \gamma(k,j).
\end{align*}

By M\"{o}bius inversion
\begin{align*}
\lambda_k=\sum_{ab=k}\mu(a)\theta(b)
\end{align*}
and so 
\begin{align}
\sum_{\substack{k|\mathcal{P}\\ k\equiv 0(e)}}\lambda_d \gamma(k,j)&=\sum_{\substack{b|\mathcal{P}\\}}\theta(b)\gamma(b,j) \sum_{\substack{a|\mathcal{P}/b\\\frac{e}{(b,e)}|a}}\mu(a) \gamma(a,j)\nonumber \\
&=\sum_{\substack{b|\mathcal{P}\\}}\theta(b)\gamma(b,j)\mu(e/(e,b))\gamma(e/(e,b),j) \sum_{\substack{a|\mathcal{P}/[b,e]}}\mu(a) \gamma(a,j)\nonumber \\
&=\sum_{\substack{b|\mathcal{P}\\}}\theta(b)\gamma(b,j)\mu(e/(e,b))\gamma(e/(e,b),j) \prod_{p|\mathcal{P}/[b,e]}(1-\gamma(p,j)). \label{eq_lemsie1_1}
\end{align}

We recall $(e,j)=1$ and so $\gamma(e/(e,b),j)=g(e/(e,b))$ and the product in~\eqref{eq_lemsie1_1} vanishes unless $j|b$. If $j|b$ then
\begin{align*}
\prod_{\substack{p|\mathcal{P}/[b,e]\\ }}(1-\gamma(p,j))&=\prod_{p|\mathcal{P}/j}\bigl(1-g(p)\bigr) \prod_{ p | [b,e]/j}\bigl(1-g(p)\bigr)^{-1}\\
&=V \prod_{p|j}\bigl(1-g(p)\bigr)^{-1}\prod_{ p | [b,e]/j}\bigl(1-g(p)\bigr)^{-1}\\
&=V \frac{h(j)}{g(j)}\prod_{ p | [b,e]/j}\bigl(1-g(p)\bigr)^{-1}.
\end{align*}
Plugging this in, recalling again $(e,j)=1$, and writing $b=jb'$, we see that~\eqref{eq_lemsie1_1} is
\begin{align*}
&V \frac{h(j)}{g(j)}\sum_{\substack{b|\mathcal{P}\\j|b}}\theta(b)\gamma(b,j)\mu(e/(e,b))g(e/(e,b))\prod_{ p | [b,e]/j}\bigl(1-g(p)\bigr)^{-1}\\
&=V \frac{h(j)}{g(j)}\sum_{\substack{b'|\mathcal{P}/j\\}}\theta_{b'j}g(b')\mu(e/(e,b'))g(e/(e,b'))\prod_{ p | [b',e]}\bigl(1-g(p)\bigr)^{-1}\\
&=V \frac{h(j)}{g(j)}\sum_{\substack{b'|\mathcal{P}/j\\}}\theta_{b'j}\mu(e)\mu(e,b')g([b',e])\prod_{ p | [b',e]}\bigl(1-g(p)\bigr)^{-1}\\
&=V \frac{h(j)}{g(j)}h(e)\mu(e)\sum_{\substack{b'|\mathcal{P}/j\\}}\theta_{b'j}h(b')\frac{\mu((e,b'))}{h((e,b'))},
\end{align*}
as required. \end{proof}

The fundamental lemma of sieve theory states results of type~\eqref{eq_fundlem}. We require a slightly different form that is easier for us to apply it in the next subsection. This is again motivated by a similar approach for almost primes in short intervals, see~\cite[around eq. (25), (26)]{matoalmost}.

\begin{lemma}\label{lem_fundlem}Let $\theta$ belong to an upper or lower bound $\beta$ sieve with range $\mathcal{P}_0$ with $\mathcal{P}_0^+\leq z$ and level $D$. Write $z_r=z^{\bigl(\frac{\beta-1}{\beta+1}\bigr)^r}$, $s=\frac{\log D}{\log z}$, $\mathcal{P}_0(x,y)=\prod_{\substack{p\in \mathcal{P}_0\\ x < p\leq y}}p$, and $\mathcal{P}_0(x)=\mathcal{P}_0(0,x)$ . Assume that $s>\beta+1$. Then it holds that
\begin{align}
|\theta_{n}-1_{(n,\mathcal{P}_0(z))=1}|\leq \tau(n)^2\sum_{r>( s-\beta-1)/2}2^{-r}1_{(n,\mathcal{P}_0(z_r))=1}.
\end{align}
\end{lemma}

\begin{proof}
The $\beta$ sieve is a combinatorial sieve that is constructed by iteration with certain cutoff parameters, see~\cite[eq. (6.31)--(6.34), (6.54)]{cribro}. From this it follows immediately that in the case of an upper bound sieve
\begin{align*}
\theta(n)=1_{(n,\mathcal{P}_0(z))=1}+ \sum_{r \text{ odd}}V_r(n),
\end{align*}
and in the case of a lower bound sieve
\begin{align*}
\theta(n)=1_{(n,\mathcal{P}_0(z))=1}- \sum_{r \text{ even}}V_r(n),
\end{align*}
where
\begin{align*}
V_r(n)=\sum_{\substack{p_1\ldots p_rd=n\\ p_i\in \mathcal{P}_0(z), \,\,\, p_r<p_{r-1}<\ldots p_1 \\ p_1p_2\ldots p_r p_r^\beta\geq D \\ p_1p_2\ldots p_hp_h^{\beta} <D \text{ for all } h<r, h\equiv r(2) }}1_{(d,\mathcal{P}_0(p_r))=1}.
\end{align*}

Let us take a closer look at this sum. By~\cite[Corollary 6.6]{cribro} we have that $p_r\geq z^{\bigl(\frac{\beta-1}{\beta+1}\bigr)^{r/2}}$. We also have $D\leq p_1^{r+\beta}\leq z^{r+\beta}$, which is impossible unless $r\geq s-\beta$. Finally, the number of choices for $d$ is bounded by $\tau(n)$ and the number of choices for $p_1,\ldots,p_r$ by $2^{\omega(n)-r}\leq \tau(n)2^{-r}$. By writing $r=2r'+1$ or $r=2r'$ respectively for the case of an upper bound or lower bound sieve, the lemma follows.
\end{proof}

\subsection{Singular series}\label{sec: singular}
In this subsection we conclude the auxiliary results by proving Lemmas~\ref{lem_SS1} and~\ref{lem_SS2}. 

\begin{proof}[Proof of Lemma~\ref{lem_SS1}]
We recall the notation $s_0=\frac{\log D_0}{\log P}$ and that we have to show
\begin{align*}
\sum_{q^\dagger\leq P'}\frac{\mathcal{F}(q^\dagger,m)}{\varphi(q^\dagger)^2}=V(\mathcal{P}^\dagger)\frac{\mathfrak{S}(m)}{\mathfrak{S}(m,2\widetilde{\mathcal{P}})}\Bigl(1+O\Big(e^{-cs_0}+\frac{e^{100\sqrt{\log N}}}{P'}\Big)\Bigr),
\end{align*}
where the summation runs (by notation) over $q^\dagger|(\mathcal{P}^{\dagger})^\infty$,
\begin{align*}
\mathcal{F}(q,m)=\sum_{j_i|q}S_1^\dagger(q,j_1,1)S_2^\dagger(q,j_2,1)F(\chi_{0}^{(q)},\chi_{0}^{(q)},j_1,j_2,m),
\end{align*}
and $S_i^{\dagger}$, $F$, are given by~\eqref{eq_Sdef},~\eqref{eq_Fdef} respectively.

We start by applying Lemma~\ref{lem_sie1} with 
\begin{align*}
g(p)=g(p,q)=\begin{cases}\frac{1}{p-1} & \text{ if } p\nmid q \\
0  &\text{ else}
\end{cases}
\end{align*}
to get
\begin{align*}
S_i^\dagger(q,j_i,1)=\prod_{p|\mathcal{P}^\dagger}\bigl(1-g(p,q)\bigr) \sum_{b|\mathcal{P}^\dagger/j_i} \theta_i(j_i b) h(b,q).
\end{align*}
Here
\begin{align*}
h(p,q)=\begin{cases}\frac{1}{p-2} & \text{ if } p\nmid q \\
0  &\text{ else}
\end{cases}
\end{align*}
and
\begin{align*}
\prod_{p|\mathcal{P}^\dagger}\bigl(1-g(p,q)\bigr)&=\frac{\varphi(q)}{\varphi_2(q)}\prod_{p|\mathcal{P}^\dagger}\Bigl(1-\frac{1}{p-1}\Bigr)\\
&=\frac{\varphi(q)}{\varphi_2(q)}V(\mathcal{P^\dagger}).
\end{align*}

Thus, since $j_i|q$, 
\begin{align*}
S_i^\dagger(q,j_i,1)=V(\mathcal{P}^\dagger)\frac{\varphi(q)}{\varphi_2(q)} \sum_{b|\mathcal{P}^\dagger/q} \theta_i(j_i b) h(b),
\end{align*}
where $h(b)=h(b,1)$. Plugging this in and using that $F$ vanishes unless $q$ is squarefree, we arrive at
\begin{align}
&\sum_{\substack{q\leq P'\\ q|(\mathcal{P}^\dagger)^\infty}} \sum_{j_i|q} \frac{S_1^\dagger(q,j_1,1)S_2^\dagger(q,j_2,1)F(\chi^{(q)}_{0},\chi^{(q)}_{0},j_1,j_2,m)}{\varphi(q)^2} \nonumber \\
&= V(\mathcal{P}^\dagger)^2 \sum_{\substack{j_i\leq P' \\ j_i|(\mathcal{P}^\dagger)^\infty}} \sum_{b_i|\mathcal{P}^\dagger/j_i} \theta_1(b_1j_1)h(b_1)\theta_2(b_2j_2)h(b_2) \sum_{\substack{q\leq P'  \\ \substack{[j_1,j_2]|q\\ q|\mathcal{P}^\dagger/((b_1,b_2))}}}\frac{F(\chi^{(q)}_{0},\chi^{(q)}_{0},j_1,j_2,m)}{\varphi_2(q)^2}. \label{eq_SS1}
\end{align}

Recall that $\mathcal{P}^\dagger$ is odd. For any pair $j_1,j_2$ with $j_i|(\mathcal{P}^\dagger)^\infty$ we have by multiplicativity (Lemma~\ref{lem_Fmult})
\begin{align}\label{eq_SS2}
\sum_{\substack{q\leq P'  \\ \substack{[j_1,j_2]|q\\ q|\mathcal{P}^\dagger/((b_1,b_2))}}}\frac{F(\chi^{(q)}_{0},\chi^{(q)}_{0},j_1,j_2,m)}{\varphi_2(q)^2}=\prod_{p|[j_1,j_2]}\bigl(\frac{F(\chi^{(p)}_{0},\chi^{(p)}_{0},j_1,j_2,m)}{(p-2)^2}\bigr)\sum_{\substack{q\leq P'/[j_1,j_2]\\ q|\mathcal{P}^\dagger/((b_1,b_2,j_1,j_2))}} \prod_{p|q}\frac{\sigma(p,m)}{(p-2)^2},
\end{align}
where $\sigma(p,m)=F(\chi^{(p)}_{0},\chi^{(p)}_{0},1,1,m)$ and more generally $F(\chi^{(p)}_{0},\chi^{(p)}_{0},j_1,j_2,m)$ are given explicitly by Lemmas~\ref{lem_Feval} and~\ref{lem_Feval2}. The term $b_1=b_2=j_1=j_2=1$ in~\eqref{eq_SS1} contributes 
\begin{align*}
V(\mathcal{P}^\dagger)^2 \sum_{\substack{q\leq P'\\ q|\mathcal{P}^\dagger }} \prod_{p|q}\frac{\sigma(p,m)}{(p-2)^2}&=V(\mathcal{P}^\dagger)^2  \left(\frac{\mathfrak{S}(m)}{\mathfrak{S}(m,2\widetilde{\mathcal{P}})}-\sum_{\substack{q>P'\\ (q,\widetilde{2\mathcal{P}})=1}} \prod_{p|q}|\mu(q)|\frac{\sigma(p,m)}{(p-2)^2}\right).
\end{align*}

We use~\eqref{eq_Feval1} and Rankin's trick to estimate the tail sum as
\begin{align*}
\sum_{\substack{q>P'\\ (q,2\widetilde{\mathcal{P}})=1}} \prod_{p|q}|\mu(q)|\frac{\sigma(p,m)}{(p-2)^2}&\ll \sum_{\substack{q>P'\\ (q,2)=1}}\prod_{p\mid m(m+2)(m+4)}\frac{20}{p}\prod_{p\mid q, p\nmid m(m+2)(m+4)}\frac{4}{(p-2)^2} \\
&\ll (P')^{-1/2}\prod_{p\mid m(m+2)(m+4)}\left(1+p^{1/2}\cdot\frac{20}{p}\right)\\
&\quad \times \prod_{p\nmid m(m+2)(m+4)}\left(1+p^{1/2}\cdot \frac{4}{(p-2)^2}\right)\\
&\ll (P')^{-1/2}\exp\Bigl(\sum_{p\mid m(m+2)(m+4)}\frac{20}{p^{1/2}}\Bigr)\\
&\ll (P')^{-1/2}\exp(100\omega(m(m+2)(m+4))^{1/2})\\
&\ll (P')^{-1/2}\exp(100(\log N)^{1/2}).
\end{align*}

The remaining part of~\eqref{eq_SS1} is part of the error term and to deal with it we first estimate~\eqref{eq_SS2}. For $p|j_1j_2$ we have by Lemma~\ref{lem_Feval2} the always valid bound $|F(\chi^{(p)}_{0},\chi^{(p)}_{0},j_1,j_2,m)|$ $\leq p-1$. We get
\begin{align*}
&\Bigl|\prod_{p|[j_1,j_2]}\Bigl(\frac{F(\chi^{(p)}_{0},\chi^{(p)}_{0},j_1,j_2,m)}{(p-2)^2}\Bigr)\sum_{\substack{q\leq P/[j_1,j_2]\\ q|\mathcal{P}^\dagger/((b_1,b_2,j_1,j_2))}} \prod_{p|q}\frac{\sigma(p,m)}{(p-2)^2}\Bigr|\\
&\ll \frac{\prod_{p|j_1j_2}\bigl(1+\frac{3}{p}\bigr)}{[j_1,j_2]}\sum_{q} \prod_{p|q}\frac{|\sigma(p,m)|}{\varphi_2(p)^2}\\
&\ll \frac{\prod_{p|j_1j_2}\bigl(1+\frac{3}{p}\bigr)}{[j_1,j_2]} \mathfrak{S}(m),
\end{align*}
where we used that either $\sigma(p,m)>0$ or $\sigma(p,m)=O(1)$.
To complete the proof of Lemma~\ref{lem_SS1} it consequently suffices to show
\begin{align}\label{eq_SS2b}
\sum_{\substack{j_i| \mathcal{P}^\dagger}} \sum_{\substack{b_i|\mathcal{P}^\dagger/j_i \\ b_1j_1b_2j_2\neq 1}}  \frac{|\theta_1(b_1j_1)\theta_2(b_2j_2)|\prod_{p|b_1b_2j_1j_2}\bigl(1+\frac{O(1)}{p}\bigr)}{b_1b_2[j_1,j_2]} &\ll e^{-cs_0}.
\end{align}

Let us first consider the contribution of terms with $b_1j_1\neq 1 $ and $b_2j_2\neq 1$ to~\eqref{eq_SS2b}. We can apply Lemma~\ref{lem_fundlem} (with $z=P$) and write $j_i=d l_i$ with $(l_1,l_2)=1$ to estimate it by
\begin{align}\nonumber
\leq &\sum_{r_i>(s_0-\beta-1)/2}2^{-r_1-r_2} \sum_{d|\mathcal{P}^\dagger}\sum_{\substack{b_il_i|\mathcal{P}^\dagger/d\\(db_il_i,\mathcal{P}^\dagger(P_{r_i}))=1\\}}\frac{\tau(d)^4 \tau(b_1)^2\tau(b_2)^2\tau(l_1)^2\tau(l_2)^2\prod_{p|dl_1l_2b_1b_2}\bigl(1+\frac{O(1)}{p}\bigr)}{dl_1l_2b_1b_2}\\ \nonumber
\leq & \sum_{r_i>(s_0-\beta-1)/2}2^{-r_1-r_2} \sum_{\substack{db_1l_1|\mathcal{P}^\dagger\\ \nonumber (db_1l_1,\mathcal{P}^\dagger(P_{r_1}))=1}} \frac{\tau(d)^4 \tau(b_1)^2\tau(l_1)^2\prod_{p|dl_1b_1}\bigl(1+\frac{O(1)}{p}\bigr)}{dl_1b_1}\\
&\times \sum_{\substack{b_2l_2|\mathcal{P}^\dagger\\ \nonumber (b_2l_2,\mathcal{P}^\dagger(P_{r_2}))=1}} \frac{\tau(b_2)^2\tau(l_2)^2\prod_{p|l_2b_2}\bigl(1+\frac{O(1)}{p}\bigr)}{l_2b_2}\\
&\leq \Bigg(\sum_{r>(s_0-\beta-1)/2}2^{-r} \sum_{\substack{b|\mathcal{P}^\dagger\\ (b,\mathcal{P}^\dagger(P_{r}))=1}}\frac{\tau(b)^{6}\prod_{p|b}\bigl(1+\frac{O(1)}{p}\bigr)}{b}\Bigg)^2. \label{eq_rsum}
\end{align}

By the definition of $P_r=P^{\bigl(\frac{\beta-1}{\beta+1}\bigr)^r}$ in Lemma~\ref{lem_fundlem} we have
\begin{align*}
\sum_{\substack{b|\mathcal{P}^\dagger\\ (b,\mathcal{P}^\dagger(P_{r}))=1}}\frac{\tau(b)^{6}\prod_{p|b}\bigl(1+\frac{O(1)}{p}\bigr)}{b}&\ll \prod_{P_r<p\leq P}\Bigl(1+\frac{2^6}{p}\Bigr)\\
&\ll \Bigl(\frac{\log P}{\log P_r} \Bigr)^{2^{6}}\\
&= \Bigl(\frac{\beta+1}{\beta-1} \Bigr)^{2^{6}r}.
\end{align*}

As long as
\begin{align}\label{eq_betacond1}
\beta>185,
\end{align}
we have $\Bigl(\frac{\beta+1}{\beta-1} \Bigr)^{2^{6}}<2$ and so the tail sum over $r$ in~\eqref{eq_rsum} converges with exponential decay in $s_0$. The required estimate~\eqref{eq_SS2b} follows in the currently considered case $j_1b_1\neq 1$, $j_2b_2\neq 1$. The remaining case that $b_1j_1b_2j_2\neq 1$ but $b_1j_1=1$ or $b_2j_2=1$ works similarly but is easier. In fact, it follows directly from a standard application of the fundamental lemma.
\end{proof}

We now prove the last remaining result, Lemma~\ref{lem_SS2}. While the just proved Lemma~\ref{lem_SS1} asymptotically evaluates the singular series with sieve weights, we now only require an upper bound of correct order of magnitude, but need it for any pseudo-singular series associated with a pair of primitive characters and again including sieve weights. 

  \begin{proof}[Proof of Lemma~\ref{lem_SS2}]
 We recall that we have to show that for any pair of primitive characters $\xi_1,\xi_2$ it holds that
 \begin{align*}
\mathcal{G}^\dagger(\xi_1 \xi_2,m)\widetilde{\mathcal{G}}(\xi_1,\xi_2,m)\ll V(\mathcal{P}^\dagger)^2\mathfrak{S}(m)
 \end{align*}
 and split this task into showing
\begin{align}\label{eq_lemSS2_1b}
\mathcal{G}^\dagger(\xi_1, \xi_2,m)\ll V(\mathcal{P}^\dagger)^2 \frac{\mathfrak{S}(m)}{\mathfrak{S}(m,2\widetilde{\mathcal{P}})}
\end{align}
and
\begin{align}\label{eq_lemSS2_2}
\widetilde{\mathcal{G}}(\xi_1,\xi_2,m)\ll \mathfrak{S}(m,2\widetilde{\mathcal{P}}).
\end{align}

We start with showing~\eqref{eq_lemSS2_1b} and further recall
\begin{align*}
\mathcal{G}^\dagger(\xi_1,\xi_2,m)&=\sum_{\substack{e_i l_i=g_i^\dagger\\ e_i|\mathcal{P}^\dagger}}  \sum_{\substack{q|(\mathcal{P}^{\dagger})^\infty\\  \substack{ l_i|q \\ (q,e_1e_2)=1}}}\sum_{j_i|q} \frac{ |S_1^\dagger(q,j_1,e_1)S_2^\dagger(q,j_2,e_2)F(\xi^{(l_1)}_{1}\chi^{(q)}_{0},\xi^{(l_2)}_{2}\chi^{(q)}_{0},1,1,m)|}{\varphi(q)^2},
\end{align*}
where $g_i$ is the conductor of $\xi_i$ and $g_i^\dagger$ denotes its $\mathcal{P}^\dagger$ component. 

We start as in the proof of Lemma~\ref{lem_SS1} and apply Lemma~\ref{lem_sie1}. It gives us now, observing that $(e_i,q)=1$,
\begin{align*}
S_i^\dagger(q,j_i,e_i)=V(\mathcal{P}^\dagger) \frac{\varphi(q)}{\varphi_2(q)}\mu(e_i)h(e_i)\sum_{b|\mathcal{P}/q}\theta_i(j_i b)h(b)\frac{\mu((e_i,b))}{h((e_i,b))},
\end{align*} 
where $h$ is as before. The condition $(q,e_1e_2)=1$ and $[l_1,l_2]|q$ can only be fulfilled simultaneously if $(e_1e_2,l_1l_2)=1$. Applying the triangle inequality, we thus estimate
\begin{align}\nonumber
\mathcal{G}^\dagger(\xi_1,\xi_2,m)&\leq V(\mathcal{P}^\dagger)^2 \sum_{\substack{\substack{e_il_i=g_i^\dagger\\  (e_1e_2,l_1l_2)=1 }\\ e_i|\mathcal{P}^\dagger}}  \sum_{j_ib_i|\mathcal{P}^{\dagger}}|\theta_1(b_1j_1)\theta_2(b_2j_2)|h([b_1,e_1])h([b_2,e_2]) \\
&\times \sum_{\substack{q|(\mathcal{P}^\dagger)^\infty \\ \substack{[j_1,j_2,l_1,l_2]|q\\ (q,b_1b_2e_1e_2)=1}}} \frac{|F(\xi^{(l_1)}_{1}\chi^{(q)}_{0},\xi^{(l_2)}_{2}\chi^{(q)}_{0},j_1,j_2,m)|}{\varphi_2(q)^2}.\label{eq_lemss2_1c}
\end{align}
Consider the sum over $q$. The multiplicativity of $F$ (Lemma~\ref{lem_Fmult}) and Lemma~\ref{le_Fsimple} give us 
\begin{align*}
&\sum_{\substack{q|(\mathcal{P}^\dagger)^\infty \\ \substack{[j_1,j_2,l_1,l_2]|q\\ (q,b_1b_2e_1e_2)=1}}}\frac{|F(\xi^{(l_1)}_{1}\chi^{(q)}_{0},\xi^{(l_2)}_{2}\chi^{(q)}_{0},j_1,j_2,m)|}{\varphi_2(q)^2}\\
&=\prod_{p^\alpha || [j_1,j_2,l_1,l_2]} \frac{|F(\xi^{((l_1,p^\alpha))}_{1}\chi^{(p^\alpha)}_{0},\xi^{((l_2,p^\alpha))}_{2}\chi^{(p^\alpha)}_{0},j_1,j_2,m)|}{\varphi_2(p^\alpha)^2}\\
&\times\sum_{\substack{q|(\mathcal{P}^\dagger)^\infty\\ (q,b_1b_2e_1e_2j_1j_2l_1 l_2)=1}}\frac{|F(\chi^{(q)}_{0},\chi^{(q)}_{0},1,1,m)|}{\varphi_2(q)^2}\\
&\leq \prod_{p^\alpha || [j_1,j_2,l_1,l_2]} \frac{|F(\xi^{((l_1,p^\alpha))}_{1}\chi^{(p^\alpha)}_{0},\xi^{((l_2,p^\alpha))}_{2}\chi^{(p^\alpha)}_{0},j_1,j_2,m)|}{\varphi_2(p^\alpha)^2 } \sum_{\substack{q\\ (q,2\widetilde{\mathcal{P}}l_1l_2j_1j_2)=1}}\frac{|F(\chi^{(q)}_{0},\chi^{(q)}_{0},1,1,m)|}{\varphi_2(q)^2}\\
&=\prod_{p^\alpha || [j_1,j_2,l_1,l_2]} \frac{|F(\xi^{((l_1,p^\alpha))}_{1}\chi^{(p^\alpha)}_{0},\xi^{((l_2,p^\alpha))}_{2}\chi^{(p^\alpha)}_{0},j_1,j_2,m)|}{\varphi_2(p^\alpha)^2\bigl(1+\frac{|\sigma(p,m)|}{(p-2)^2}\bigr) }\prod_{p\nmid 2\widetilde{\mathcal{P}}} \Bigl(1+\frac{|\sigma(p,m)|}{(p-2)^2}\Bigr) \\
&\ll  \prod_{p^\alpha || [j_1,j_2,l_1,l_2]} \frac{|F(\xi^{((l_1,p^\alpha))}_{1}\chi^{(p^\alpha)}_{0},\xi^{((l_2,p^\alpha))}_{2}\chi^{(p^\alpha)}_{0},j_1,j_2,m)|}{\varphi_2(p^\alpha)^2\bigl(1+\frac{|\sigma(p,m)|}{(p-2)^2}\bigr) }\frac{\mathfrak{S}(m)}{\mathfrak{S}(m,2\widetilde{\mathcal{P}})}.
\end{align*}
Thus, by~\eqref{eq_lemss2_1c} our goal~\eqref{eq_lemSS2_1b} follows if we can show
\begin{align}\begin{split}\label{eq_lemSS2_2a}
&\sum_{\substack{\substack{e_il_i=g_i^\dagger\\ (e_1e_2,l_1l_2)=1 }\\ e_i|\mathcal{P}^\dagger}}  \sum_{j_ib_i|\mathcal{P}^{\dagger}}|\theta_1(b_1j_1)\theta_2(b_2j_2)|h([b_1,e_1])h([b_2,e_2]) \\
& \times \prod_{p^\alpha || [j_1,j_2,l_1,l_2]} \frac{|F(\xi^{((l_1,p^\alpha))}_{1}\chi^{(p^\alpha)}_{0},\xi^{((l_2,p^\alpha))}_{2}\chi^{(p^\alpha)}_{0},j_1,j_2,m)|}{\varphi_2(p^\alpha)^2\bigl(1+\frac{|\sigma(p,m)|}{(p-2)^2}\bigr) }
\ll 1.
\end{split}
\end{align}

We continue estimating the remaining Euler factor with the help of Lemma~\ref{lem_Festi}, starting with the case $p\nmid (j_1,j_2)$, $(l_1,p^\alpha)=(l_2,p^\alpha)=p^\alpha$, and $p|m$. We are then possibly in the bad case of ~\eqref{eq_lemFeststatement1} and together with~\eqref{eq_Feval1} get
\begin{align*}
\frac{|F(\xi_1^{(l_1,p^\alpha)}\chi^{(p^\alpha)}_{0},\xi_2^{(l_2,p^\alpha)}\chi^{(p^\alpha)}_{0},j_1,j_2,m)|}{\varphi_2(p^\alpha)^2\bigl(1+\frac{|\sigma(p,m)|}{(p-2)^2}\bigr) }&\leq \frac{p^{2\alpha}-3p^{2\alpha-1}+1}{\varphi_2(p^\alpha)^2\bigl(1+\frac{(p-4)}{(p-2)^2}\bigr)}\\
&=\frac{1-3p^{-1}+p^{-2\alpha}}{(1-\frac{2}{p})^2\bigl(1+\frac{(p-4)}{(p-2)^2}\bigr) }\\
&=\frac{1-3p^{-1}+O(p^{-2})}{\bigl(1-4p^{-1}+O(p^{-2})\bigr)\bigl(1+p^{-1} +O(p^{-2})\bigr)}\\
&=1+O(p^{-2}).
\end{align*}
If now $p\nmid (j_1,j_2)$ and either $p^{2\alpha}\nmid l_1l_2$ or $p\nmid m$ we are in the better case of~\eqref{eq_lemFeststatement1} and together with ~\eqref{eq_Feval1} obtain now
\begin{align*}
\frac{|F(\xi_1^{(l_1,p^\alpha)}\chi^{(p^\alpha)}_{0},\xi_2^{(l_2,p^\alpha)}\chi^{(p^\alpha)}_{0},j_1,j_2,m)|}{\varphi_2(p^\alpha)^2\bigl(1+\frac{|\sigma(p,m)|}{(p-2)^2}\bigr) }&\leq \frac{p^{2\alpha-1/2}+3p^{2\alpha-1}}{\varphi_2(p^\alpha)^2\bigl(1-\frac{4}{(p-2)^2}\bigr)}\\
&=\frac{p^{-1/2}+3p^{-1}}{(1-\frac{2}{p})^2\bigl(1-\frac{4}{(p-2)^2}\bigr)}\\
&\leq p^{-1/2}+O(p^{-1}).
\end{align*}

Finally, if $p|(j_1,j_2)$ we get with the help of~\eqref{eq_lemFeststatement2}
\begin{align*}
\frac{|F(\xi_1^{(l_1,p^\alpha)}\chi^{(p^\alpha)}_{0},\xi_2^{(l_2,p^\alpha)}\chi^{(p^\alpha)}_{0},j_1,j_2,m)|}{\varphi_2(p^\alpha)^2\bigl(1+\frac{|\sigma(p,m)|}{(p-2)^2}\bigr) }&\leq \frac{2p^{2\alpha-1}}{\varphi_2(p^\alpha)^2\bigl(1-\frac{4}{(p-2)^2}\bigr)} \\
&=\frac{2p^{-1}}{(1-\frac{2}{p})^2\bigl(1-\frac{4}{(p-2)^2}\bigr)}\\
&=2p^{-1}\bigl(1+O(p^{-1}) \bigr).
\end{align*}

Combining the cases, we estimate the Euler factor by
\begin{align*} &\prod_{p^\alpha || [j_1,j_2,l_1,l_2]} \frac{|F(\xi_1^{(l_1)}\chi^{(p^\alpha)}_{0},\xi_2^{(l_2)}\chi^{(p^\alpha)}_{0},j_1,j_2,m)|}{\varphi_2(p^\alpha)^2\bigl(1+\frac{|\sigma(p,m)|}{(p-2)^2}\bigr) }\\
	&\ll  \prod_{p|j_1j_2}\frac{2}{p}(1+O(p^{-1}))\prod_{\substack{p|l_1l_2, p\nmid j_1j_2 \\ (l_1,p^\infty)\neq (l_2,p^\infty)}}(p^{-1/2}+O(p^{-1}))\\
&\leq  \frac{\tau([j_1,j_2])\prod_{p|j_1j_2}(1+2p^{-1})}{[j_1,j_2]} \prod_{\substack{p|l_1l_2, p\nmid j_1j_2 \\ (l_1,p^\infty)\neq (l_2,p^\infty)}}(p^{-1/2}+O(p^{-1})).
\end{align*}
Plugging this in and changing order of summation shows that the left-hand side of~\eqref{eq_lemSS2_2a} is
\begin{align}
&\ll \nonumber \sum_{j_ib_i|\mathcal{P}^{\dagger}}\bigl|\theta_1(b_1j_1)\theta_2(b_2j_2)\bigr|\frac{\tau([j_1,j_2])\prod_{p|j_1j_2}(1+2p^{-1})}{[j_1,j_2]}\sum_{\substack{\substack{e_il_i=g_i^\dagger\\ (e_1e_2,l_1l_2)=1 }\\ e_i|\mathcal{P}^\dagger}} h([b_1,e_1])h([b_2,e_2]) \\
&\times \prod_{\substack{p|l_1l_2, p\nmid j_1j_2 \\ (l_1,p^\infty)\neq (l_2,p^\infty)}}(p^{-1/2}+O(p^{-1}))\label{eq_lemSS2_2b}.
\end{align}

We can express the sum over $e_il_i$ in multiplicative fashion in the following sense 
\begin{align*}
\prod_{\substack{p|g_1^\dagger g_2^\dagger \\ p^{\beta_1}||g_i^\dagger}}\sum_{\substack{\substack{e_il_i=p^{\beta_i}\\ (e_1e_2,l_1l_2)=1 }\\ e_i|\mathcal{P}^\dagger}} h([b_1,e_1])h([b_2,e_2]) \prod_{\substack{p|l_1l_2, p\nmid j_1j_2 \\ (l_1,p^\infty)\neq (l_2,p^\infty)}}(p^{-1/2}+O(p^{-1})).
\end{align*}
If $\beta_2=0$, we give an upper bound for the double sum over $e_il_i$ by
\begin{align*}
h(b_2)\sum_{\substack{\substack{e_1l_1=p^{\beta_1}\\ (e_1,l_1)=1 }\\ e_i|\mathcal{P}^\dagger}} h([b_1,e_1])\prod_{\substack{p|l_1, p\nmid j_1j_2 \\ }}(p^{-1/2}+O(p^{-1}))\leq h(b_2)h(b_1)\begin{cases}2 &  p|b_1j_1j_2\\
p^{-1/2}+O(p^{-1}) & p\nmid b_1j_1j_2.
\end{cases}
\end{align*}
Naturally a similar estimate holds for $\beta_1=0$. If $\beta_1 \beta_2\neq 0$, we estimate the double sum by
\begin{align*}
&\sum_{\substack{\substack{e_il_i=p^{\beta_i}\\ (e_1e_2,l_1l_2)=1 }\\ e_i|\mathcal{P}^\dagger}} h([b_1,e_1])h([b_2,e_2]) \prod_{\substack{p|l_1l_2, p\nmid j_1j_2 \\ (l_1,p^\infty)\neq (l_2,p^\infty)}}(p^{-1/2}+O(p^{-1}))\\
\leq &\sum_{\substack{\substack{e_il_i=p^{\beta_i}\\ (e_1e_2,l_1l_2)=1 }\\ e_i|\mathcal{P}^\dagger}} h([b_1,e_1])h([b_2,e_2]) \\
&=h(b_1)h(b_2)+1_{\beta_1=\beta_2=1}h([b_1,p])h([b_2,p])\\
\leq & h(b_1)h(b_2)\begin{cases}2 &  p|b_1b_2\\
1+O(p^{-2}) & p\nmid b_1b_2.
\end{cases}
\end{align*}

The estimates of the sum over $e_il_i$ show that the contribution of primes that do not divide the sieve weighted coefficients $b_i,j_i$ is $O(1)$. This is crucial for us, as we could compensate for any losses here only by getting more saving out of Gallagher's prime number theorem (Lemma~\ref{lem_GPNT}), i.e. miss power saving. We discard the condition that $p|g_1^\dagger g_2^\dagger$ and get that~\eqref{eq_lemSS2_2b} is
\begin{align*}
\ll \sum_{j_ib_i|\mathcal{P}^{\dagger}}\bigl|\theta_1(b_1j_1)\theta_2(b_2j_2)\bigr|\frac{\tau(j_1)^2\tau(j_2)^2\tau(b_1)\tau(b_2)\prod_{p|j_1j_2b_1b_2}(1+O(p^{-1}))}{b_1 b_2 [j_1,j_2]}.
\end{align*}
This sum is closely related to the one appearing in~\eqref{eq_SS2b}, the only difference being additional divisor functions. The same steps as there show that it is $1+O(e^{-cs_0})$ under the stronger condition
\begin{align}\label{eq_betacond2}
\beta>739
\end{align} 
that ensures $\Bigl(\frac{\beta+1}{\beta-1} \Bigr)^{2^{8}}<2$. The estimate~\eqref{eq_lemSS2_2a} and so~\eqref{eq_lemSS2_1b} follows. 

The proof of~\eqref{eq_lemSS2_2} is similar but considerably simpler as it does not include the sieve weights. By definition
\begin{align*}
\widetilde{\mathcal{G}}(\xi_1,\xi_2,m)&=\sum_{\substack{e_i l_i=2^{g_i(2)}\widetilde{g}_i\\ e_i|\widetilde{\mathcal{P}}}} \frac{1}{\varphi_2(e_1)\varphi_2(e_2)} \sum_{\substack{q|(2\widetilde{\mathcal{P}})^\infty\\  l_i|q \\ (q,e_1e_2)=1}}\frac{|F(\xi^{(l_1)}_{1}\chi^{(q)}_{0},\xi^{(l_2)}_{2}\chi^{(q)}_{0},1,1,m)|}{\varphi_2(q)^2}.
\end{align*}
By similar steps as for $\mathcal{G}^\dagger$ we get
\begin{align*}
\widetilde{\mathcal{G}}(\xi_1,\xi_2,m)&=\sum_{\substack{e_i l_i=2^{g_i(2)}\widetilde{g}_i\\ (e_1e_2,l_1l_2)=1\\ e_i|\widetilde{\mathcal{P}}}} \frac{1}{\varphi_2(e_1)\varphi_2(e_2)} \prod_{p^\alpha||[l_1,l_2]}\frac{F(\xi^{((l_1,p^\alpha))}_{1},\xi^{((l_2,p^\alpha))}_{2},1,1,m)}{\varphi_2(p)^2}\\
&\quad \times \prod_{\substack{p|\widetilde{2\mathcal{P}}\\ p\nmid e_1e_2l_1l_2}}\Bigl(1+\frac{|\sigma(p,m)|}{\varphi_2(p)^2}\Bigr)\\
&\leq \sum_{\substack{e_i l_i=2^{g_i(2)}\widetilde{g}_i\\ (e_1e_2,l_1l_2)=1\\ e_i|\widetilde{\mathcal{P}}}} \frac{1}{\varphi_2(e_1)\varphi_2(e_2)} \prod_{p^\alpha||[l_1,l_2]}\frac{F(\xi^{((l_1,p^\alpha))}_{1},\xi^{((l_2,p^\alpha))}_{2},1,1,m)}{\varphi_2(p)^2 \bigl(1+\frac{|\sigma(p,m)|}{\varphi_2(p)^2}\bigr)} \mathfrak{S}(m,\widetilde{2\mathcal{P}}).
\end{align*}
The required bound
\begin{align*}
\sum_{\substack{e_i l_i=2^{g_i(2)}\widetilde{g}_i\\ (e_1e_2,l_1l_2)=1\\ e_i|\widetilde{\mathcal{P}}}} \frac{1}{\varphi_2(e_1)\varphi_2(e_2)} \prod_{p^\alpha||[l_1,l_2]}\frac{F(\xi^{((l_1,p^\alpha))}_{1},\xi^{((l_2,p^\alpha))}_{2},1,1,m)}{\varphi_2(p)^2 \bigl(1+\frac{|\sigma(p,m)|}{\varphi_2(p)^2}\bigr)}\ll 1
\end{align*}
follows from~\eqref{eq_lemSS2_2a}, as it is the subsum $b_1=b_2=j_1=j_2=1$, and the estimate~\eqref{eq_lemFeststatementp=2} for the case that $[l_1,l_2]$ is even. This completes the proof of Lemma~\ref{lem_SS2}. 
\end{proof}

\bibliographystyle{plain}
\bibliography{chenrefs.bib}{}

\end{document}